\UseRawInputEncoding 

\documentclass[10pt]{article}
\usepackage[letterpaper, left=1in, top=1in, right=1in, bottom=1in, verbose, ignoremp]{geometry}

\usepackage{url}\RequirePackage[colorlinks,citecolor=blue, linkcolor=blue,urlcolor = blue]{hyperref}
\usepackage{latexsym,amssymb,amsmath,amsfonts,graphicx,color,fancyvrb,amsthm,enumitem,subcaption,mathrsfs,thmtools,stmaryrd}
\usepackage[doi=true,isbn=false,style=alphabetic,sorting=nyt,backend=biber,maxnames=5,maxalphanames=5,giveninits=true,bibencoding=utf8]{biblatex}
\bibliography{nonarchimedean_refs}
\usepackage[dvipsnames]{xcolor}
\usepackage{xy}\xyoption{all} \xyoption{poly} \xyoption{knot}
\usepackage{float}
\usepackage{bm}
\usepackage{bbm}
\usepackage{multicol,multirow}
\usepackage{array}
\usepackage{relsize}
\usepackage{chngcntr}
\usepackage{etoolbox}
\usepackage{caption}
\usepackage{tikz}
\usetikzlibrary{decorations.pathreplacing,calc,positioning}
\usetikzlibrary{shapes,backgrounds}
\usetikzlibrary{patterns}
\usetikzlibrary{arrows}
\usepackage{amsmath}
\usetikzlibrary{shapes.geometric, arrows.meta}
\usepackage{amsopn}
\usepackage{calc}
\usepackage{algorithm}
\usepackage[noend]{algpseudocode}
\usepackage{hyperref}
\usepackage{mathtools}
\usepackage{booktabs}
\usepackage{float}
\usepackage{adjustbox}
\usepackage{pdflscape}
\usepackage{forest}

\usepackage{tikz-cd} 
\usepackage{amscd} 
\usepackage{comment}
\usepackage[capitalize,nameinlink,compress]{cleveref}
\crefname{figure}{Figure}{Figures}
\crefname{equation}{}{}
\crefname{assumption}{Assumption}{Assumptions}
\crefname{subsection}{Subsection}{Subsections}

\newcounter{cdrow}

\newtheorem{theorem}{Theorem}[]
\newtheorem*{theorem*}{Theorem}
\newtheorem{corollary}[theorem]{Corollary}
\newtheorem{lemma}[theorem]{Lemma}
\newtheorem{proposition}[theorem]{Proposition}

\newtheorem*{claim*}{Claim}

\theoremstyle{definition}
\newtheorem{definition}[theorem]{Definition}
\newtheorem*{definition*}{Definition}

\newtheorem{remark}[theorem]{Remark}

\newtheorem{example}[theorem]{Example}
\newtheorem*{example*}{Example}
\newtheorem{convention}[theorem]{Convention}

\def\log{{\rm log}}

\makeatletter
\newcommand*{\op}{%
  \DOTSB
  \mathop{\vphantom{\bigoplus}\mathpalette\matt@op\relax}%
  \slimits@
}
\newcommand\matt@op[2]{%
  \vcenter{\m@th\hbox{\resizebox{\widthof{$#1\bigoplus$}}{!}{$\boxplus$}}}%
}
\makeatother

\newcommand{\BA}{\mathbb{A}}
\newcommand{\CC}{\mathbb{C}}
\newcommand{\RR}{\mathbb{R}}
\newcommand{\NN}{\mathbb{N}}
\newcommand{\FF}{\mathbb{F}}
\newcommand{\TT}{\mathbb{T}}
\newcommand{\ZZ}{\mathbb{Z}}
\newcommand{\PP}{\mathbb{P}}
\newcommand{\QQ}{\mathbb{Q}}

\renewcommand{\Im}{\mathrm{Im} \,}

\DeclareMathOperator*{\argmin}{\text{argmin}}

\newcommand{\suchthat}{\;\ifnum\currentgrouptype=16 \middle\fi|\;}
\newcommand{\bigmid}{\left.\vphantom{\Big\{} \suchthat \vphantom{\Big\}}\right.}
\newcommand{\bigslant}[2]{{\raisebox{.2em}{$#1$}\left/\raisebox{-.2em}{$#2$}\right.}}

\newcommand{\calB}{\mathcal{B}}
\newcommand{\calH}{\mathcal{H}}

\newcommand{\sgn}{\text{\sgn}}

\newcommand{\diam}{\text{diam}}

\renewcommand{\vec}[1]{\mathbf{#1}}

\newcommand{\slope}{\text{slope}}

\newcommand{\val}{{\mathrm{val}}}

\DeclareMathOperator{\argmax}{argmax}

\DeclareMathOperator{\conv}{conv}

\makeatletter
\def\@biblabel#1{}
\makeatother

\makeatletter
\patchcmd{\NAT@citex}
  {\@citea\NAT@hyper@{%
     \NAT@nmfmt{\NAT@nm}%
     \hyper@natlinkbreak{\NAT@aysep\NAT@spacechar}{\@citeb\@extra@b@citeb}%
     \NAT@date}}
  {\@citea\NAT@nmfmt{\NAT@nm}%
   \NAT@aysep\NAT@spacechar\NAT@hyper@{\NAT@date}}{}{}

\patchcmd{\NAT@citex}
  {\@citea\NAT@hyper@{%
     \NAT@nmfmt{\NAT@nm}%
     \hyper@natlinkbreak{\NAT@spacechar\NAT@@open\if*#1*\else#1\NAT@spacechar\fi}%
       {\@citeb\@extra@b@citeb}%
     \NAT@date}}
  {\@citea\NAT@nmfmt{\NAT@nm}%
   \NAT@spacechar\NAT@@open\if*#1*\else#1\NAT@spacechar\fi\NAT@hyper@{\NAT@date}}
  {}{}

\makeatother

\setlist[description]{font=\normalfont\space}
\begin{document}

\def\spacingset#1{\renewcommand{\baselinestretch}%
{#1}\small\normalsize} \spacingset{1}

\begin{flushleft}
{\Large{\textbf{Non-Archimedean Polydisc Spaces and\\
Applications to Optimisation}}}
\newline
\\
Paul Lezeau$^{1,2}$, Yiannis Fam$^{1,2}$, Anthea Monod$^{2}$ and Yue Ren$^{3,\dagger}$
\\
\bigskip
\bf{1} London School of Geometry and Number Theory, UK
\\
\bf{2} Department of Mathematics, Imperial College London, UK
\\
\bf{3} Department of Mathematics, Durham University, UK
\\
\bigskip
$\dagger$ Corresponding e-mail: yue.ren2@durham.ac.uk
\end{flushleft}


\section*{Abstract}

We propose a new framework for optimisation over non-Archimedean spaces inspired by Berkovich geometry.  Specifically, we introduce polydisc spaces, which consist of products of closed balls over a non-Archimedean field. These spaces retain the rigid hierarchical structure of the non-Archimedean field whilst acquiring many desirable geometric features absent from it.  
We show that metric trees embed naturally into these spaces, demonstrating their capacity to represent hierarchical data.
We study their metric geometry, establishing properties such as geodesic uniqueness, confirming their compatibility with classical optimisation techniques. We further propose a class of real-valued functions given by linear combinations of absolute values of polynomials. These functions admit a piecewise polynomial description along geodesics and satisfy a universal approximation property. We formulate a theory of optimisation on polydisc spaces: we prove existence of minimisers and explore algorithms for finding them.  We provide an accompanying open-source Julia library implementing the core objects and optimisation procedures introduced.

\paragraph{Keywords:} Non-Archimedean fields, Hierarchical data, Monte-Carlo tree search


\section{Introduction}
\label{sec:intro}

Hierarchical data analysis is ubiquitous in applications.  In many applications the data is naturally hierarchical, such as in phylogenetics \cite{SKSW2003,ChakerianHolmes2012}, genomics \cite{LZH2008}, linguistics \cite{BrennanHale2019}, or sociology \cite{BryanJenkins2016}.  In even more applications, hierarchies are introduced artificially to simplify the data analysis, such as in education \cite{Raudenbush1988}, economics \cite{EDKHDM2013}, computer vision \cite{MPSKPM2022}, and data mining \cite{SillaFreitas2011}. Text provides another important example: linguistic structures are inherently hierarchical, with words combining into phrases, phrases into sentences, and sentences into larger semantic units \cite[e.g.,][]{miller1995wordnet}. As a consequence, hierarchical representations have become a fundamental component of modern language modelling.

A common feature of hierarchical representations is that they naturally induce non-Archimedean geometric structures, which are frequently visualised as metric trees; see \cref{fig:introduction}. This non-Archimedean nature however is difficult to capture over the real numbers.  For example, to realize the distance between the four points in \cref{fig:introduction} over $\RR$ requires embedding them in $\RR^3$.  More generally, embedding $n$ such points is only possible in $\RR^{n-1}$ \cite{FAVER_KOCHALSKI_MURUGAN_VERHEGGEN_WESSON_WESTON_2014}.  Prominent workarounds include working in hyperbolic space \cite{NickelKiela2017} or working in very high dimensions, as is common in LLMs.

Capturing the non-Archimedean nature faithfully, however, requires working over so-called \emph{non-Archimedean fields}, such as the $p$-adic numbers $\QQ_p$, whose distances satisfy the strong triangle inequality:
\begin{equation*}
  |a-b|_p\leq\max(|a-c|_p, |c-b|_p) \qquad\text{for all } a,b,c\in\QQ_p.
\end{equation*}
Because of the strong triangle inequality, their topology naturally organises into strictly nested, non-overlapping sets that form perfect, infinite branching trees. These fields are indispensable in number theory, where their hierarchical structure allows the study of polynomial equations prime by prime.  However, optimisation directly over non-Archimedean spaces remains challenging: their totally disconnected topology prohibits the use of many standard optimisation techniques. Thus, optimisation problems frequently reduce to difficult and computationally expensive combinatorial searches.

\begin{figure}[t]
  \centering
  \begin{tikzpicture}
    \node (tree)
    {
      \begin{tikzpicture}[scale=1.2]
        \coordinate (o) at (0,0);
        \coordinate (v1) at (-1,-1);
        \coordinate (v2) at (1,-1);
        \coordinate (w1) at (-1.5,-2);
        \coordinate (w2) at (-0.5,-2);
        \coordinate (w3) at (0.5,-2);
        \coordinate (w4) at (1.5,-2);
        \draw
        (o) -- node[left,font=\footnotesize] {$1$} (v1)
        (o) -- node[right,font=\footnotesize] {$1$} (v2)
        (v1) -- node[left,font=\footnotesize] {$1$} (w1)
        (v1) -- node[right,font=\footnotesize] {$1$} (w2)
        (v2) -- node[left,font=\footnotesize] {$1$} (w3)
        (v2) -- node[right,font=\footnotesize] {$1$} (w4);
        \node[below] at (w1) {dog};
        \node[below] at (w2) {wolf};
        \node[below] at (w3) {cat};
        \node[below] at (w4) {lion};
      \end{tikzpicture}
    };
    \node[anchor=west,xshift=10mm,yshift=3mm] (matrix) at (tree.east)
    {
      \begin{tikzpicture}
        \matrix (m) [
        matrix of math nodes,
        nodes in empty cells,
        row sep=0.25em,
        column sep=0.25em
        ] {
          & & \text{dog} & \text{wolf} & \text{cat} & \text{lion} \\
          \text{dog}  & & 0 & 2 & 4 & 4 \\
          \text{wolf} & & 2 & 0 & 4 & 4 \\
          \text{cat}  & & 4 & 4 & 0 & 2 \\
          \text{lion} & & 4 & 4 & 2 & 0 \\
        };

        \node[fit=(m-2-3) (m-5-6),inner sep=0pt] (sub) {};
        \draw[thick] (sub.north west) [bend right=20] to (sub.south west);
        \draw[thick] (sub.north east) [bend left=20] to (sub.south east);
      \end{tikzpicture}
    };
    \node[anchor=west,xshift=10mm,yshift=-3mm] (embedding) at (matrix.east)
    {
      $
      \begin{array}{lcl}
        \text{dog} &\mapsto& 0\cdot 2^{-2}+0\cdot 2^{-1}\\
        \text{wolf} &\mapsto& 0\cdot 2^{-2}+1\cdot 2^{-1} \\
        \text{cat} &\mapsto& 1\cdot 2^{-2}+0\cdot 2^{-1}\\
        \text{lion} &\mapsto& 1\cdot 2^{-2}+1\cdot 2^{-1}\\
      \end{array}
      $
    };
  \end{tikzpicture}\vspace{-3mm}
  \caption{Four mammals separated by their evolutionary distance arranged in form of a tree (left), their distance matrix (centre), and four points in $\QQ_2$ realizing the distances (right).}
  \label{fig:introduction}
\end{figure}
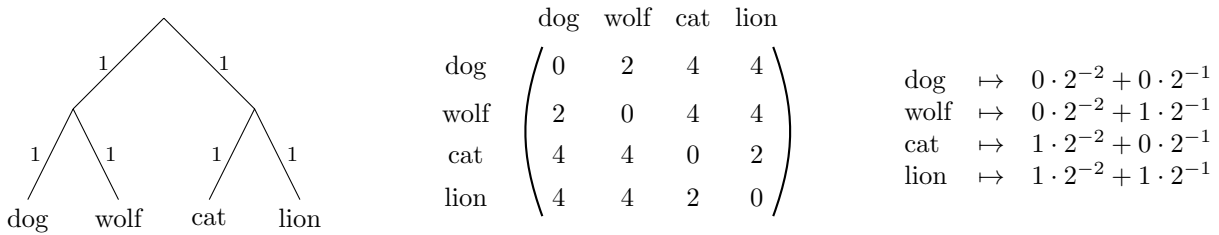

In this paper, we propose an approach to optimisation over non-Archimedean fields inspired by the theory of Berkovich spaces, which were introduced to provide a topologically well-behaved framework for geometry and analysis over non-Archimedean fields \cite{Berkovich90}.  Instead of working in a non-Archimedean space $K^n$, we introduce and work in a family of geometric optimisation domains called \emph{polydisc spaces}, denoted $\mathcal B^n$, which contain $K^n$ as a distinguished subset.  Our polydisc spaces retain the hierarchical structure inherited from $K^n$ while simultaneously acquiring geometric features absent from it, such as path connectedness, geodesics, convexity, and tangent directions.  

We also introduce a class of continuous real valued functions on $\mathcal B^n$ called \emph{absolute polynomials}.  These functions form the foundation of our optimisation framework: we show that large classes of absolute-polynomial optimisation problems admit global minima, and prove an analogue of the universal approximation theorem.

A key feature of our polydisc framework is that it supports two geometric complementary perspectives: On the one hand, the hierarchical structure of non-Archimedean geometry allows metric trees to be embedded easily and often isometrically.  On the other hand, the additional geometric structure of polydisc spaces permits the development of optimisation methods that have no analogue on the underlying non-Archimedean field.  Thus, the same space can simultaneously serve both as a natural representation space for hierarchical data and as a natural domain for optimisation.

A central goal of this work is not only to develop a mathematical framework for optimisation on non-Archimedean spaces, but also to make such a framework computationally viable. To this end, we introduce \textsc{NonArchimedeanMachineLearning.jl} \cite{NAML}, an open-source \textsc{Julia} \cite{Julia} library implementing the geometric structures, function classes, and optimisation algorithms developed in this paper. The library enables practical experimentation with non-Archimedean optimisation problems and provides a software foundation for future research at the interface of non-Archimedean geometry, optimisation, and machine learning.

\subsection{Main Contributions and Structure of the Paper}

This paper develops a framework for representing hierarchical data and performing optimisation within a common geometric setting. Our contributions span the geometry of polydisc spaces, their use as representation spaces for hierarchical structures, a function theory adapted to optimisation, optimisation algorithms tailored to the non-Archimedean setting, and the development of \textsc{NonArchimedeanMachineLearning.jl}---an open-source software library implementing our developed theory. Together, these components provide both a mathematical and computational foundation for optimisation and machine learning on non-Archimedean spaces.

We begin in \cref{sec:polydiscs} by introducing the polydisc spaces $\mathcal B^n$ and $\mathcal H^n$. 
We develop their basic geometric properties, including geodesics, convexity, tangent directions, and behaviour under field extensions. Many of these constructions are inspired by ideas from Berkovich geometry, though our focus is on explicit and computationally tractable objects tailored for optimisation. An important result of this section is \cref{lem:metricTopologicallyEquivalent}, which shows that the metrics $d_{\mathcal B}$ and $d_{\mathcal H}$ induce the same topology on $\mathcal H^n$. This result allows us to move freely between the two metrics, using whichever is most convenient for the task at hand.

Having established the geometry of polydisc spaces, we turn to their role in \cref{sec:embedding} as representation spaces for hierarchical data. We show that metric trees admit natural and often isometric embeddings into $(\mathcal H^1,d_{\mathcal H})$. The central result \cref{prop:metricTreeEmbedding}, provides an explicit criterion guaranteeing the existence of such embeddings. When the criterion is not satisfied over the original field $K$, it yields a constructive procedure for extending the field to one in which an embedding exists. As a corollary, we obtain an extension property for embeddings, allowing embeddings of growing trees to be extended without the need to modify previously embedded subtrees.

In \cref{sec:functions}, we introduce the main class of objective functions considered in this paper, building on so-called \emph{absolute polynomials} and \emph{valuation polynomials}. While their definitions are naturally expressed in terms of suprema and infima, we derive explicit formulas that make them computationally tractable and suitable for practical use. \cref{thm:universalApproximation} is a universal approximation theorem showing that linear combinations of absolute polynomials are expressive enough to approximate broad classes of continuous functions on polydisc spaces.

The optimisation theory is developed in \cref{sec:optimisation}. \cref{thm:globalMinimumLocallyCompact} and \cref{thm:globalMinimumAffinePower} establish the existence of global minima for a large class of positive linear combinations of absolute polynomials, including objective functions such as mean squared error. \cref{thm:globalMinimumUnivariate} provides an explicit starting value for the optimisation in the univariate case. Building on the function theory, we adapt several optimisation methods to the polydisc setting, including best-first descent, best-first gradient descent, Monte-Carlo tree search (MCTS), and deterministic optimistic optimisation (DOO).  Moreover, \cref{thm:best-first-descent-convergence} and \cref{thm:best-first-gradient-descent-convergence} show that best-first and best-first gradient converge for univariate polynomials that split. \cref{thm:doo} shows that DOO converges provably.

Finally, in \cref{sec:applications}, we present \textsc{NonArchimedeanMachineLearning.jl}, an open-source Julia library accompanying this work. The library implements the core geometric objects, function classes, embedding procedures, and optimisation algorithms developed in the preceding sections, providing an end-to-end computational framework for optimisation on non-Archimedean polydisc spaces. We use the library to conduct a range of computational experiments and illustrate the practical behaviour of our proposed methods. Our software contribution is an important component of the work, adapting the theoretical framework developed into a practical platform for experimentation and future research.

\subsection{Connection to Existing Literature}
There is a growing body of work exploring the use of non-Archimedean fields in applications, much of it focused on adapting classical Archimedean theories to a non-Archimedean setting. Early examples include the development of probability theory over non-Archimedean fields by Evans, beginning with their foundational work on Gaussian measures \cite{Evans1989}. More recently, Zúñiga-Galindo and collaborators have investigated neural network architectures defined over non-Archimedean fields \cite{zuniga2024deep,zuniga2024hierarchical}, while Baker, McCallum, and Pattinson \cite{BMP2025}, as well as Martins \cite{martins2025learning}, have studied statistical and machine learning problems over the $p$-adics.

Among these works, the problems considered in \cite{BMP2025,martins2025learning} are most closely related to those we study here. The key distinction, however, is that these works formulate optimisation problems directly over the non-Archimedean space $K^n$, whereas we introduce and work over the larger polydisc spaces $\mathcal B^n$. This change of domain equips the optimisation problem with additional geometric structure while preserving the underlying hierarchical geometry of $K^n$. In particular, viewing $\mathcal B^n$ as a continuous geometric space makes notions such as geodesics, tangent directions, and gradient-based optimisation available in a natural way, whereas \cite[Section 3.3]{BMP2025} argues that gradient methods are not viable over $K^n$. At the same time, the rational skeleton of $\mathcal B^n$ admits a natural directed acyclic graph structure, allowing the use of classical tree-search methods such as Monte-Carlo tree search and deterministic optimistic optimisation.  As opposed to purely descending procedures such as \cite[Algorithm 1]{martins2025learning}, these methods can revisit and compare alternative branches of the search space through established exploration--exploitation strategies.


\section{Polydisc Spaces and their Geometry}
\label{sec:polydiscs}

In this section, we introduce the non-Archimedean domains of the real-valued functions which we will aim to optimise: \emph{polydisc spaces}.  Polydisc spaces are products of infinite regular metric trees.  We discuss their relevant metric and metric geometric properties for optimisation and briefly outline their connection to Berkovich spaces.

\begin{convention}
  \label{con:valuedField}
  In this paper, $K$ will be a field with non-trivial valuation $\val\colon K\rightarrow\RR\cup\{\infty\}$ that is complete with respect to the resulting non-Archimedean absolute value $|\cdot|\colon K\rightarrow \RR_{\geq 0}$. We will generally write $|a|=\exp(-\val(a))$ and $\val(a)=-\log(|a|)$ for $a\in K$ without specifying the base, which traditionally varies with the field $K$ and is a detail that has no effect on our results. Without loss of generality, we may assume that $1\in\val(K)$, and we fix $\pi$ to be an element of $K$ that has valuation 1. We further assume that there is a splitting of the valuation $(\val(K^\ast),+)\rightarrow (K^\ast,\cdot)$ and absolute value $(|K^\ast|,\cdot)\rightarrow (K^\ast,\cdot)$. We use $\pi^v$ and $\pi^{-\log(r)}$ to denote the image of $v\in\val(K^\ast)$ and $r\in|K^\ast|$ under said splitting map, respectively.

  We denote the residue field of $K$ by $\mathfrak K$ and the completion of the algebraic closure of $K$ by $\hat{K}$. We point out the standard fact that $\hat{K}$ is both algebraically closed and complete, and that both the valuation and absolute value of $K$ extend uniquely to $\hat K$ \cite[Section 3.2.4 Theorem 2]{BGR84}.

\end{convention}

We begin by presenting some foundational instances of well-known valued fields, which we will use as working examples throughout the paper.

\begin{example}
  \label{ex:padicNumbers}
  The field of \emph{$p$-adic numbers} for some prime $p$ is a valued field.  It is given by
  \begin{equation*}
    \QQ_p\coloneqq\left\{\sum_{k=k_0}^\infty c_k p^k \bigmid k_0\in\ZZ,\; c_k\in \{0,\dots,p-1\} \right\},
  \end{equation*}
  and endowed with the usual \emph{$p$-adic valuation} $\val_p$ and \emph{$p$-adic absolute value} $|\cdot|_p$:
  \begin{align*}
    \val_p\colon &\QQ_p\rightarrow \RR\cup\{\infty\}, & \sum_{k=k_0}^\infty c_k p^k&\mapsto \min\Big(\{k\in\ZZ_{\geq k_0}\mid c_k\neq 0\}\cup\{\infty\}\Big),\\
    |\cdot|_p\colon&\QQ_p \rightarrow \RR_{\geq 0}, & c&\mapsto p^{-\val_p(c)}.
  \end{align*}
  Its residue field is the finite field $\FF_p$ and the completion of its algebraic closure is also known as the \emph{complex $p$-adic numbers}, $\CC_p$. In these cases, we always take $\pi = p$.
\end{example}

\begin{example}
  \label{ex:LaurentSeries}
  The field of \emph{Laurent series} $F(\!(t)\!)$ over some coefficient field $F$ is a valued field.  It is given by
    \begin{equation*}
    F(\!(t)\!)\coloneqq\Big\{\sum_{k=k_0}^\infty c_k t^k \bigmid k_0\in\ZZ,\; c_k\in F \Big\}
  \end{equation*}
  and endowed with the usual $t$-adic valuation $\val_t$ and $t$-adic absolute value $|\cdot|_t$:
  \begin{align*}
    \val_t\colon &F(\!(t)\!)\rightarrow \RR\cup\{\infty\}, & \sum_{k=k_0}^\infty c_k t^k&\mapsto \min\Big(\{k\in\ZZ_{\geq k_0}\mid c_k\neq 0\}\cup\{\infty\}\Big),\\
    |\cdot|_t\colon&F(\!(t)\!) \rightarrow \RR_{\geq 0}, & c&\mapsto e^{-\val_t(c)}.
  \end{align*}
  Its residue field is the coefficient field $F$.  For $F=\RR$, its algebraic closure is known as the field \emph{complex Puiseux series} $\CC\{\!\{t\}\!\}$ and its completion is a \emph{Levi-Civita field}. In these cases, we always take $\pi = t$.
\end{example}

The field $\QQ_p$ of $p$-adic numbers can be viewed as a non-Archimedean counterpart to floating-point numbers: the valuation identifies the dominant term analogous to the exponent in a floating-point number, while the remaining coefficients encode increasingly fine-scale corrections.

\subsection{Polydisc Spaces}

We now introduce polydisc spaces and describe some basic properties of non-Archimedean polydiscs.

\begin{definition}\label{def:polydiscSpace}
  The \emph{polydisc} with centre $a=(a_1,\dots,a_n)\in K^n$ is a set of the form
  \begin{align*}
      B_{\mathcal{B}}(a,r)\coloneqq& \Big\{ (z_1,\dots,z_n)\in \hat{K}^n \bigmid |a_1-z_1|\leq r_1, \dots, |a_n-z_n|\leq r_n \Big\} \qquad\text{or}\\ 
    B_{\mathcal{H}}(a,v) \coloneqq& \Big\{ (z_1,\dots,z_n)\in \hat{K}^n \bigmid \val(a_1-z_1)\geq v_1, \dots, \val(a_n-z_n)\geq v_n \Big\}
  \end{align*}
  for some \emph{absolute multiradius} $r=(r_1,\dots,r_n)\in \RR_{\geq 0}^n$ or \emph{valuative multiradius} $v=(v_1,\dots,v_n)\in(\RR\cup\{\infty\})^n$.
  The polydiscs $B_{\mathcal{B}}(a, r)$ and $B_{\mathcal{H}}(a, v)$ (or their multiradii $r$ and $v$) are \emph{rational} if $r_i\in |K| = \left\{ |a| \mid a \in K \right\}$ or $v_i\in \val(K)= \left\{ \val(a) \mid a \in K \right\}$ for all $i=1,\dots,n$, respectively.

  The \emph{$n$-polydisc space} $\mathcal B^n$ and the \emph{hyperbolic $n$-polydisc space} $\mathcal H^n$ are the sets of polydiscs given by
  \begin{align*}
    &\mathcal B^n \coloneqq \Big\{ B_{\mathcal{B}}(a,r) \bigmid a\in K^n, r\in \RR_{\geq 0}^n \Big\} = \Big\{ B_{\mathcal{H}}(a,v) \bigmid a\in K^n, v\in(\RR\cup\{\infty\})^n \Big\} \quad\text{and}\\
    &\mathcal H^n \coloneqq \Big\{ B_{\mathcal{B}}(a,r) \bigmid a\in K^n, r\in \RR_{> 0}^n \Big\} = \Big\{ B_{\mathcal{H}}(a,v) \bigmid a\in K^n, v\in \RR^n \Big\}.
  \end{align*}
\end{definition}

There is a natural inclusion $K^n \hookrightarrow \mathcal B^n$ sending $a$ to $B_{\mathcal{B}}(a,0)$. We will often identify $K^n$ with its image in $\mathcal B^n$, and write $a$ for $B_{\mathcal{B}}(a,0)$ and refer to such points as \emph{$K$-points}.  The polydisc space itself is partially ordered by inclusion. We will use the symbols $\le$ and $\subseteq$ interchangeably to denote this inclusion.

Note that polydiscs are defined as discs in $\hat K^n$ rather than $K^n$, as different radii may give rise to the same polydisc over $K$. For example, consider the following discs centered around $a=0$:
    \begin{center}
      \begin{tikzpicture}
        \node[anchor=base east] (K0) at (-1,0) {$\{ z\in \QQ_p\mid |z|\leq p^{-1} \}$};
        \node[anchor=base west] (K1) at (1,0) {$\{ z\in \QQ_p\mid |z|\leq p^{-1/2} \}\subseteq \QQ_p$};
        \node[anchor=base east] (K0hat) at (-1,-1.15) {$\{ z\in \hat\QQ_p\mid |z|\leq p^{-1} \}$};
        \node[anchor=base west] (K1hat) at (1,-1.15) {$\{ z\in \hat\QQ_p\mid |z|\leq p^{-1/2} \}\subseteq \hat \QQ_p.$};
        \node[anchor=base] at (0,0) {$=\;p\cdot\ZZ_p\;=$};
        \node[anchor=base] at (0,-1.15) {$\not\ni\;\;p^{1/2}\;\;\in$};
        \draw[draw opacity=0] (K0) -- (K0hat) node[midway,sloped] {$\subseteq$};
        \draw[draw opacity=0] (K1) -- (K1hat) node[midway,sloped] {$\subseteq$};
      \end{tikzpicture}
    \end{center}

This gives rise to the observation that non-Archimedean discs behave very differently from their Archimedean counterparts: In particular, any two discs are either disjoint or one is contained in the other.

\begin{lemma}
  \label{lem:recentre}
  Let $B_{\mathcal{B}}(a,r), B_{\mathcal{B}}(a',r')\in \mathcal B^1$.
  \begin{enumerate}
  \item[(i)] If $B_{\mathcal{B}}(a,r) \subseteq B_{\mathcal{B}}(a',r')$, then $r\leq r'$ and $B_{\mathcal{B}}(a,r')=B_{\mathcal{B}}(a',r')$.
  \item[(ii)] If $B_{\mathcal{B}}(a,r) \nsubseteq B_{\mathcal{B}}(a',r')$ and $B_{\mathcal{B}}(a,r) \nsupseteq B_{\mathcal{B}}(a',r')$, then $B_{\mathcal{B}}(a,r)\cap B_{\mathcal{B}}(a',r')=\emptyset$.
  \end{enumerate}
\end{lemma}
\begin{proof}
The first statement is immediate.  For the second, we may assume without loss of generality that $r\leq r'$.  Suppose there is some $c\in B_{\mathcal{B}}(a,r)\cap B_{\mathcal{B}}(a',r')\neq\emptyset$.  We want to show that $B_{\mathcal{B}}(a,r)\subseteq B_{\mathcal{B}}(a',r')$.  To do this, let $b\in B_{\mathcal{B}}(a,r)$.  Then
  \begin{equation*}
    |a'-b| = |a'-c+c-a+a-b| \leq \max(\underbrace{|a'-c|}_{\leq r'},\underbrace{|c-a|}_{\leq r},\underbrace{|a-b|}_{\leq r}) = r'. \qedhere
  \end{equation*}
\end{proof}


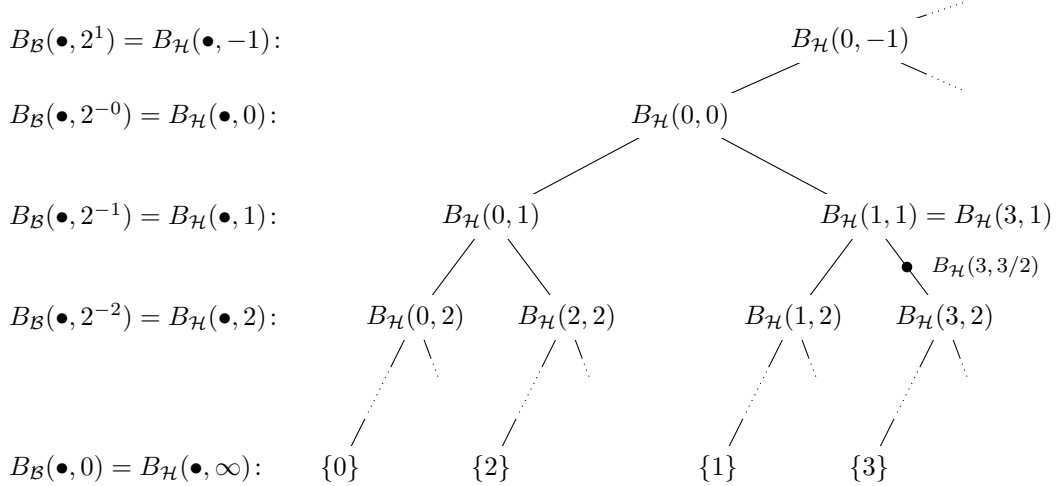
\begin{figure}
  \centering
  \begin{tikzpicture}[y={(0,1.333)},every node/.style={fill=white}]
    \coordinate (b) at (3.5,2); 
    \coordinate (b0) at (1,1); 
    \coordinate (b1) at (6,1);
    \coordinate (b00) at (0,0); 
    \coordinate (b01) at (2,0);
    \coordinate (b10) at (5,0);
    \coordinate (b11) at (7,0);
    \coordinate (e00) at (-1,-1.5); 
    \coordinate (e01) at (1,-1.5);
    \coordinate (e10) at (4,-1.5);
    \coordinate (e11) at (6,-1.5);
    \coordinate (l) at (-5.5,2.75); 
    \coordinate (l0) at (-5.5,2); 
    \coordinate (l1) at (-5.5,1);
    \coordinate (l2) at (-5.5,0);
    \coordinate (l3) at (-5.5,-1.5);
    \coordinate (a) at (5.75,2.75); 
    \draw
    (b) -- (b0)
    (b) -- (b1)
    (b0) -- (b00)
    (b0) -- (b01)
    (b1) -- (b10)
    (b1) -- coordinate[midway] (pt) (b11);
    \draw[shorten <= 15mm]
    (b00) -- (e00);
    \draw[shorten >= 15mm]
    (b00) -- (e00);
    \draw[dotted,shorten <= 8mm,shorten >= 8mm]
    (b00) -- (e00);
    \draw
    (b00) -- ++(0.2,-0.4);
    \draw[dotted]
    (b00)++(0.2,-0.4) -- ++(0.1,-0.2);
    \draw[shorten <= 15mm]
    (b01) -- (e01);
    \draw[shorten >= 15mm]
    (b01) -- (e01);
    \draw[dotted,shorten <= 8mm,shorten >= 8mm]
    (b01) -- (e01);
    \draw
    (b01) -- ++(0.2,-0.4);
    \draw[dotted]
    (b01)++(0.2,-0.4) -- ++(0.1,-0.2);
    \draw[shorten <= 15mm]
    (b10) -- (e10);
    \draw[shorten >= 15mm]
    (b10) -- (e10);
    \draw[dotted,shorten <= 8mm,shorten >= 8mm]
    (b10) -- (e10);
    \draw
    (b10) -- ++(0.2,-0.4);
    \draw[dotted]
    (b10)++(0.2,-0.4) -- ++(0.1,-0.2);
    \draw[shorten <= 15mm]
    (b11) -- (e11);
    \draw[shorten >= 15mm]
    (b11) -- (e11);
    \draw[dotted,shorten <= 8mm,shorten >= 8mm]
    (b11) -- (e11);
    \draw
    (b11) -- ++(0.2,-0.4);
    \draw[dotted]
    (b11)++(0.2,-0.4) -- ++(0.1,-0.2);
    \draw (b) -- (a);
    \draw (a) -- ++(1,-0.3333);
    \draw[dotted] (a)++(1,-0.3333) -- ++(0.5,-0.1667);
    \draw (a) -- ++(1,0.25);
    \draw[dotted] (a)++(1,0.25) -- ++(0.5,0.125);
    \fill (pt) circle (2pt);
    \node at (a) {$B_{\mathcal{H}}(0,-1)$};
    \node at (b) {$B_{\mathcal{H}}(0,0)$};
    \node at (b0) {$B_{\mathcal{H}}(0,1)$};
    \node (b1Text) at (b1) {$B_{\mathcal{H}}(1,1)$};
    \node[anchor=base west,xshift=-1.5mm] at (b1Text.base east) {$=B_{\mathcal{H}}(3,1)$};
    \node at (b00) {$B_{\mathcal{H}}(0,2)$};
    \node at (b01) {$B_{\mathcal{H}}(2,2)$};
    \node at (b10) {$B_{\mathcal{H}}(1,2)$};
    \node at (b11) {$B_{\mathcal{H}}(3,2)$};
    \node[right,xshift=2mm,font=\footnotesize] at (pt) {$B_{\mathcal{H}}(3,3/2)$};
    \node at (e00) {$\{0\}$};
    \node at (e01) {$\{2\}$};
    \node at (e10) {$\{1\}$};
    \node at (e11) {$\{3\}$};
    \node[right] at (l) {$B_{\mathcal{B}}(\bullet,2^{1})=B_{\mathcal{H}}(\bullet,-1)\colon$};
    \node[right] at (l0) {$B_{\mathcal{B}}(\bullet,2^{-0})=B_{\mathcal{H}}(\bullet,0)\colon$};
    \node[right] at (l1) {$B_{\mathcal{B}}(\bullet,2^{-1})=B_{\mathcal{H}}(\bullet,1)\colon$};
    \node[right] at (l2) {$B_{\mathcal{B}}(\bullet,2^{-2})=B_{\mathcal{H}}(\bullet,2)\colon$};
    \node[right] at (l3) {$B_{\mathcal{B}}(\bullet,0)=B_{\mathcal{H}}(\bullet,\infty)\colon$};
  \end{tikzpicture}\vspace{-2mm}
  \caption{The polydisc space $\mathcal B^1$ for $K=\QQ_2$ arranged as an infinite tree.}
  \label{fig:polydiscSpace}
\end{figure}

\begin{example}\label{ex:polydiscSpace}
  \cref{fig:polydiscSpace} illustrates $\mathcal B^1$ for $K=\QQ_2$ arranged as an infinite binary tree.  Note that points on the edges also represent points in $\mathcal B^1$, such as the highlighted point $B_{\mathcal{H}}(3,3/2)$ on the edge between $B_{\mathcal{H}}(3,2)$ and $B_{\mathcal{H}}(3,1)=B_{\mathcal{H}}(1,1)$. However, only the points on the nodes are rational.  Points on the same vertical level represent discs of the same radius, and the edges also indicate inclusion. This reflects the dichotomy established in Lemma~\ref{lem:recentre}.  
\end{example}


\subsection{Metrics on Polydisc Spaces}

Having defined polydisc spaces, in order to study their metric geometric properties, we now define metrics on them, which extend (up to a factor of 2) the metric on the valued field, and describe geodesics on them.  In particular, we will show that the polydisc spaces are uniquely geodesic spaces, which is a key result for the goal of optimisation.

\begin{definition}\label{def:joinAndMetric}
  The set $\mathcal B^n$ has a natural partial order given by inclusion.  We first define the \emph{join} of two discs $B=B_{\mathcal{B}}(a,r)=B_{\mathcal{H}}(a,v),\, B'=B_{\mathcal{B}}(a',r')=B_{\mathcal{H}}(a',v')\in \mathcal B^1$ by
  \begin{align*}
    B\vee B'\coloneqq&
    \begin{cases}
        B\cup B' = B_{\mathcal{B}}(a,\, \max(r,r')) = B_{\mathcal{B}}(a',\, \max(r,r')) &\text{if } B\subseteq B' \text{ or } B\supseteq B',\\
        B_{\mathcal{B}}(a,|a-a'|) = B_{\mathcal{B}}(a',|a-a'|) &\text{otherwise}.
    \end{cases}\\[2mm]
    =&
    \begin{cases}
        B\cup B' = B_{\mathcal{H}}(a,\, \min(v,v')) = B_{\mathcal{H}}(a',\, \min(v,v')) &\text{if } B\subseteq B' \text{ or } B\supseteq B',\\
        B_{\mathcal{H}}(a,\val(a-a')) = B_{\mathcal{H}}(a',\val(a-a')) &\text{otherwise}.
    \end{cases}
  \end{align*}
  For polydiscs $B=B_1\times\dots\times B_n$ and $B'=B_1'\times\dots\times B_n'\in\mathcal B^n$, we define the join coordinate-wise:
  \begin{equation*}
    B\vee B'\coloneqq (B_1\vee B_1')\times\dots\times (B_n\vee B_n').
  \end{equation*}
  
  In $\mathcal B^1$, we define the distance between two discs $B=B_{\mathcal{B}}(a,r)$ and $B'=B_{\mathcal{B}}(a',r')$ to be
  \begin{equation*}
    d_{\mathcal{B}}(B, B') \coloneqq
    \begin{cases}
      |r-r'| &\text{if } B_{\mathcal{B}}(a,r) \subseteq B_{\mathcal{B}}(a',r') \text{ or } B_{\mathcal{B}}(a',r') \subseteq B_{\mathcal{B}}(a,r),\\[2mm]
      d_{\mathcal{B}}(B, B \vee B')+ d_{\mathcal{B}}(B \vee B', B') &\text{otherwise}.
    \end{cases}
  \end{equation*}
  
  In $\mathcal H^1$, we define the distance between two discs $B=B_{\mathcal{H}}(a,v)$ and $B'=B_{\mathcal{H}}(a',v')$ to be
  \begin{equation*}
    d_{\mathcal{H}}(B, B') \coloneqq
    \begin{cases}
      |v-v'| &\text{if } B_{\mathcal{B}}(a,r) \subseteq B_{\mathcal{B}}(a',r') \text{ or } B_{\mathcal{B}}(a',r') \subseteq B_{\mathcal{B}}(a,r),\\[2mm]
      d_{\mathcal{H}}(B, B \vee B')+ d_{\mathcal{H}}(B \vee B', B') &\text{otherwise}.
    \end{cases}
  \end{equation*}

  We use $d_{\mathcal{B}}^p$ and $d_{\mathcal{H}}^p$ to denote the $\ell^p$-metric on $\mathcal B^n$ and $\mathcal H^n$, respectively, i.e., for $B=B_1\times\dots\times B_n$ and $B'=B_1'\times\dots\times B_n'$, in each respective space,
  \begin{equation*}
    d_{\mathcal{B}}^p(B,B')\coloneqq\Big(\sum_{i=1}^n d_{\mathcal{B}}(B_i,B_i')^p\Big)^{\frac{1}{p}} \quad\text{and}\quad d_{\mathcal{H}}^p(B,B')\coloneqq\Big(\sum_{i=1}^n d_{\mathcal{H}}(B_i,B_i')^p\Big)^{\frac{1}{p}}.
  \end{equation*}
\end{definition}

Recall that $K^n$ can be canonically embedded into $\mathcal B^n$ by mapping a point $a\in K^n$ to a polydisc of radius zero $B_{\mathcal{B}}(a,0)\in \mathcal B^n$.  The metric $d_{\mathcal{B}}^p$ then extends the $\ell^p$-metric on $K^n$.

\begin{lemma}
  \label{lem:metricExtendingFieldMetric}
  For $a=(a_1,\dots,a_n)$, $a'=(a_1',\dots,a_n')\in K^n$, we have
  \begin{equation*}
    d_{\mathcal{B}}^p\big(B_{\mathcal{B}}(a,0),\, B_{\mathcal{B}}(a',0)\big) = 2\cdot \Big(\sum_{i=1}^n |a_i-a_i'|^p\Big)^{1/p}.
  \end{equation*}
\end{lemma}
\begin{proof}
  Note that $B_{\mathcal{B}}(a_i,0)\vee B_{\mathcal{B}}(a'_i,0)=B_{\mathcal{B}}(a_i,r_i)=B_{\mathcal{B}}(a'_i,r_i)$ for $r_i\coloneqq |a_i-a_i'|$.  Hence
  \begin{align*}
    &d_{\mathcal{B}}^p\big(B_{\mathcal{B}}(a,0),\, B_{\mathcal{B}}(a',0)\big) = \left(\sum_{i=1}^n d_{\mathcal{B}}\big(B_{\mathcal{B}}(a_i,0),\, B_{\mathcal{B}}(a_i',0)\big)^p\right)^{1/p}\\
    &\quad = \left(\sum_{i=1}^n \Big(d_{\mathcal{B}}\big(B_{\mathcal{B}}(a_i,0),\, B_{\mathcal{B}}(a_i,0)\vee B_{\mathcal{B}}(a_i',0)\big)+d_{\mathcal{B}}\big(B_{\mathcal{B}}(a_i',0),\, B_{\mathcal{B}}(a_i,0)\vee B_{\mathcal{B}}(a_i',0)\big) \Big)^p\right)^{1/p}\\
    &\quad = \left(\sum_{i=1}^n \Big(d_{\mathcal{B}}\big(B_{\mathcal{B}}(a_i,0),B_{\mathcal{B}}(a_i,r_i)\big)+d_{\mathcal{B}}\big(B_{\mathcal{B}}(a_i',0),\, B_{\mathcal{B}}(a_i',r_i)\big) \Big)^p\right)^{1/p}\\
    &\quad = \left(\sum_{i=1}^n \Big(r_i+r_i\Big)^p\right)^{1/p} = 2\cdot \left(\sum_{i=1}^n r_i^p\right)^{1/p} = 2\cdot \left(\sum_{i=1}^n |a_i-a_i'|^p\right)^{1/p}. \qedhere
  \end{align*}
\end{proof}

Finally, note that all $\ell^p$-norms induce the same topology on $\mathcal B^n$ and $\mathcal H^n$, and that the two topologies coincide on $\mathcal H^n$.

\begin{lemma}
  \label{lem:metricTopologicallyEquivalent}
  All $d_{\mathcal{B}}^p$ induce the same topology on $\mathcal B^n$, and all $d_{\mathcal{H}}^p$ induce the same topology on $\mathcal H^n$.  Both topologies coincide when restricted to $\mathcal H^n$.
\end{lemma}
\begin{proof}
  The fact that all $d_{\mathcal{B}}^p$ and all $d_{\mathcal{H}}^p$ induce the same topology is a consequence of the equivalence of $\ell^p$-norms on $\RR^n$.  To show that both topologies coincide when restricted to $\mathcal H^n$, it suffices to consider $n=1$ in which case the result follows from a direct calculation.
\end{proof}

\subsection{Geodesics in Polydisc Spaces}

The geometry of $\mathcal{B}^n$ and $\mathcal{H}^n$ is governed by a particularly tractable class of metric spaces, namely \emph{geodesic spaces}. This structure provides a canonical notion of paths and directional movement, which will play a central role in both the theoretical and algorithmic developments that follow. In particular, it allows us to define and analyse descent procedures in a manner analogous to gradient-based methods.

We recall the relevant notions from the theory of geodesic spaces and adapt them to the polydisc setting; for further background, see \cite{BuragoBuragoIvanov2001, Gromov2007}. Our notation follows \cite{Lang2013}.


\begin{definition}
  \label{def:geodesicsAndGeodesicSpace}
  Let $(X,d)$ be a metric space.  A \emph{geodesic} (with a fixed speed $\lambda\in\RR_{>0}$) is a map $\gamma\colon I\rightarrow X$ from some closed interval $I\subseteq\RR$ to $X$ such that $d(\gamma(t),\gamma(t'))=\lambda \cdot |t-t'|$ for all $t,t'\in I$. $(X,d)$ is \emph{geodesic} if any two points $x,y\in X$ are connected by a geodesic and it is \emph{uniquely geodesic} if any two points $x,y\in X$ are connected by exactly one geodesic of speed one.
\end{definition}

\begin{remark}
  \label{rem:geodesicsAndGeodesicsSpace}
  In \cite{BuragoBuragoIvanov2001} and \cite{Gromov2007}, our geodesics in \cref{def:geodesicsAndGeodesicSpace} are referred to as ``shortest paths'' and ``minimising geodesics'', respectively, while geodesics are locally shortest paths and locally minimising geodesics. By \cite[Lemma 2.2]{Lang2013}, any geodesic space is a ``length space'' in which this distinction is not necessary.  We do not distinguish between these two notions, because $\mathcal B^1$ is a metric tree. 
\end{remark}

The geodesic structure of $\mathcal{B}^n$ and $\mathcal{H}^n$ admits an explicit description, reflecting the hierarchical organisation of discs. In one dimension, geodesics are obtained by first moving ``upwards'' to the smallest disc containing both endpoints and then ``downwards'' to the target. This construction extends naturally to higher dimensions via a product structure. We collect the resulting properties in the following lemma.

\begin{lemma}
  \label{lem:geodesics}
  The geodesics in $\mathcal{B}^n$ and $\mathcal{H}^n$ admit the following explicit description.
  \begin{enumerate}
  \item[(i)] \label{enumitem:geodesic1} Let $B=B_{\mathcal{B}}(a,s)$ and $B'=B_{\mathcal{B}}(a',s') \in\mathcal B^1$ be two discs.  Let $r\coloneqq d_{\mathcal{B}}(B,B')$, $r_1\coloneqq d_{\mathcal{B}}(B,B\vee B')$, and $r_2\coloneqq d_{\mathcal{B}}(B',B\vee B')$, so that $r=r_1+r_2$ and $s+r_1=s'+r_2$.  Then a geodesic from $B$ to $B'$ is given by
  \begin{equation*}
    \gamma_{BB'}\colon\quad [0,1]\rightarrow \mathcal B^n,\quad t\mapsto
    \begin{cases}
        B_{\mathcal{B}}(a, s + r\cdot t)&\text{for }t\leq \frac{r_1}{r},\\
        B_{\mathcal{B}}(a', s'+ r\cdot (1-t))&\text{for }t\geq \frac{r_1}{r}.
    \end{cases}
  \end{equation*}
  Moreover, let $B=B_1\times\dots\times B_n, B'=B_1'\times\dots\times B_n' \in\mathcal B^n$ be two polydiscs.
  Then the product of geodesics in $\mathcal B^1$, 
  $\gamma_{BB'}\coloneqq \gamma_{B_1B_1'}\times\dots\times\gamma_{B_nB_n'}$, is a geodesic in $\mathcal B^n$.
  \item[(ii)] \label{enumitem:geodesic2} Let $B=B_{\mathcal{H}}(a,v), B'=B_{\mathcal{H}}(a',v') \in\mathcal H^1$ be two discs.  Let $w\coloneqq d_{\mathcal{H}}(B,B')$, $w_1\coloneqq d_{\mathcal{H}}(B,B\vee B')$, and $w_2\coloneqq d_{\mathcal{H}}(B',B\vee B')$, so that $w=w_1+w_2$ and $v-w_1=v'-w_2$.  Then a geodesic from $B$ to $B'$ is given by:
  \begin{equation*}
    \gamma_{BB'}\colon\quad [0,1]\rightarrow \mathcal H^n,\quad t\mapsto
    \begin{cases}
        B_{\mathcal{H}}(a, v - w\cdot t)&\text{for }t\leq \frac{w_1}{w},\\
        B_{\mathcal{H}}(a', v'- w\cdot (1-t))&\text{for }t\geq \frac{w_1}{w}.
    \end{cases}
  \end{equation*}
  Moreover, let $B=B_1\times\dots\times B_n, B'=B_1'\times\dots\times B_n' \in\mathcal H^n$ be two polydiscs.
  Then the product of geodesics in $\mathcal H^1$, 
  $\gamma_{BB'}\coloneqq \gamma_{B_1B_1'}\times\dots\times\gamma_{B_nB_n'}$, is a geodesic in $\mathcal H^n$.
  \item[(iii)] \label{enumitem:geodesic3} The polydisc spaces $(\mathcal B^n,d_{\mathcal{B}}^p)$ and $(\mathcal H^n,d_{\mathcal{H}}^p)$ are geodesic spaces for all $n\in\ZZ_{>0}$ and all $1\leq p\leq \infty$.  They are uniquely geodesic spaces if $n=1$ or $1<p<\infty$.
  \end{enumerate}
\end{lemma}
\begin{proof}\
  \begin{enumerate}
  \item[(i)] Let $\hat s\in \RR$ so that $B\vee B'=B_\mathcal{B}(a,\hat s)=B_\mathcal{B}(a',\hat s)$.  First note that $\gamma_{BB'}$ is well-defined, since
    \begin{align*}
      B_{\mathcal{B}}(a,s+d_{\mathcal{B}}(B,B\vee B')) & = B_{\mathcal{B}}(a,s+(\hat s-s)) = B_{\mathcal{B}}(a,\hat s) = B_{\mathcal{B}}(a',\hat s) = B_{\mathcal{B}}(a,s'+(\hat s-s')) \\
      & = B_{\mathcal{B}}(a',s'+d_{\mathcal{B}}(B',B\vee B')) = B_{\mathcal{B}}(a',s'+(d_{\mathcal{B}}(B,B')-d_{\mathcal{B}}(B,B\vee B')))
    \end{align*}
    and $\gamma_{BB'}$ is a continuous path connecting $B$ and $B'$.

    To show that $\gamma_{BB'}$ is a minimising geodesic, pick $t_1,t_2\in [0,d_{\mathcal{B}}(B,B')]$ with $t_1<t_2$.  If $t_1,t_2\leq d_{\mathcal{B}}(B,B\vee B')$, we have
    \begin{equation*}
      d_{\mathcal{B}}\big(\gamma(t_1),\, \gamma(t_2)\big) = d_{\mathcal{B}}\big(B_{\mathcal{B}}(a, s + t_1),\, B_{\mathcal{B}}(a, s + t_2)\big) = \big|(s + t_1) - (s + t_2 )\big| = |t_1-t_2|.
    \end{equation*}
    If $t_1,t_2\geq d_{\mathcal{B}}(B,B\vee B')$, we can show the equality $d\big(\gamma(t_1),\, \gamma(t_2)\big)=|t_1-t_2|$ using similar arguments, and, if $t_1<d_{\mathcal{B}}(B,B\vee B')<t_2$, the same equality follows from the previous two cases using the fact that $d_{\mathcal{B}}\big(\gamma(t_1),\, \gamma(t_2)\big) = d_{\mathcal{B}}\big(\gamma(t_1),\gamma(d_{\mathcal{B}}(B,B\vee B'))\big) + d_{\mathcal{B}}\big(\gamma(d_{\mathcal{B}}(B,B\vee B')),\gamma(t_2)\big)$.
  \item[(ii)] Follows similarly to the proof of (i).
  \item[(iii)] We will show the statement for $\mathcal B^n$.  The statement for $\mathcal H^n$ follows similarly.

    The fact that $(\mathcal B^n,d_{\mathcal{B}}^p)$ is a geodesic space follows from (i).  The fact that  $(\mathcal B^1,d_{\mathcal{B}})$ is a uniquely geodesic space follows from it being a metric tree and any two points on a tree being connected by a unique path.
    The fact that $(\mathcal B^n,d_{\mathcal{B}}^p)$ for all $1<p<\infty$ is due to \cite[Proposition 3.1 (2)]{Lang2013}. \qedhere
  \end{enumerate}
\end{proof}

The description formalised in Lemma \ref{lem:geodesics} can be visualised concretely in the one-dimensional case, where the geodesic structure reflects the underlying tree of discs and their joins. The following example illustrates both the metrics and the corresponding geodesic paths.

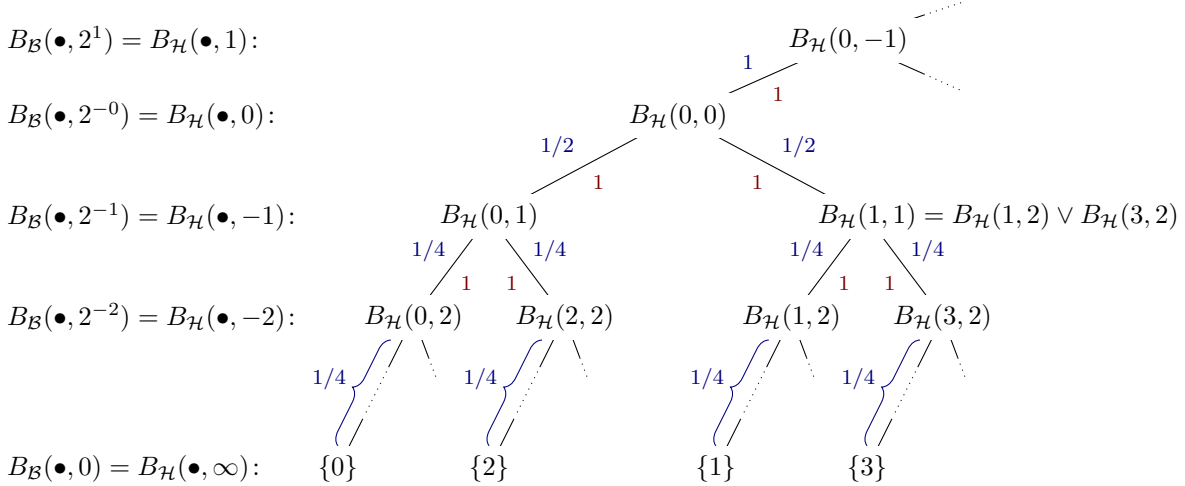
\begin{figure}
  \centering
  \begin{tikzpicture}[y={(0,1.333)},every node/.style={fill=white}]
    \coordinate (b) at (3.5,2); 
    \coordinate (b0) at (1,1); 
    \coordinate (b1) at (6,1);
    \coordinate (b00) at (0,0); 
    \coordinate (b01) at (2,0);
    \coordinate (b10) at (5,0);
    \coordinate (b11) at (7,0);
    \coordinate (e00) at (-1,-1.5); 
    \coordinate (e01) at (1,-1.5);
    \coordinate (e10) at (4,-1.5);
    \coordinate (e11) at (6,-1.5);
    \coordinate (l) at (-5.5,2.75); 
    \coordinate (l0) at (-5.5,2);
    \coordinate (l1) at (-5.5,1);
    \coordinate (l2) at (-5.5,0);
    \coordinate (l3) at (-5.5,-1.5);
    \coordinate (a) at (5.75,2.75); 
    \draw
    (b) -- node[blue!50!black,fill=none,anchor=south east,font=\footnotesize] {$1/2$} node[red!50!black,fill=none,anchor=north west,font=\footnotesize] {$1$} (b0)
    (b) -- node[blue!50!black,fill=none,anchor=south west,font=\footnotesize] {$1/2$} node[red!50!black,fill=none,anchor=north east,font=\footnotesize] {$1$} (b1)
    (b0) -- node[blue!50!black,fill=none,anchor=south east,font=\footnotesize,inner sep=2pt] {$1/4$} node[red!50!black,fill=none,anchor=north west,font=\footnotesize] {$1$} (b00)
    (b0) -- node[blue!50!black,fill=none,anchor=south west,font=\footnotesize,inner sep=2pt] {$1/4$} node[red!50!black,fill=none,anchor=north east,font=\footnotesize] {$1$} (b01)
    (b1) -- node[blue!50!black,fill=none,anchor=south east,font=\footnotesize,inner sep=2pt] {$1/4$} node[red!50!black,fill=none,anchor=north west,font=\footnotesize] {$1$} (b10)
    (b1) -- node[blue!50!black,fill=none,anchor=south west,font=\footnotesize,inner sep=2pt] {$1/4$} node[red!50!black,fill=none,anchor=north east,font=\footnotesize] {$1$} (b11);
    \draw[shorten <= 15mm]
    (b00) -- (e00);
    \draw[shorten >= 15mm]
    (b00) -- (e00);
    \draw[dotted,shorten <= 8mm,shorten >= 8mm]
    (b00) -- (e00);
    \draw
    (b00) -- ++(0.2,-0.4);
    \draw[dotted]
    (b00)++(0.2,-0.4) -- ++(0.1,-0.2);
    \draw[shorten <= 15mm]
    (b01) -- (e01);
    \draw[shorten >= 15mm]
    (b01) -- (e01);
    \draw[dotted,shorten <= 8mm,shorten >= 8mm]
    (b01) -- (e01);
    \draw
    (b01) -- ++(0.2,-0.4);
    \draw[dotted]
    (b01)++(0.2,-0.4) -- ++(0.1,-0.2);
    \draw[shorten <= 15mm]
    (b10) -- (e10);
    \draw[shorten >= 15mm]
    (b10) -- (e10);
    \draw[dotted,shorten <= 8mm,shorten >= 8mm]
    (b10) -- (e10);
    \draw
    (b10) -- ++(0.2,-0.4);
    \draw[dotted]
    (b10)++(0.2,-0.4) -- ++(0.1,-0.2);
    \draw[shorten <= 15mm]
    (b11) -- (e11);
    \draw[shorten >= 15mm]
    (b11) -- (e11);
    \draw[dotted,shorten <= 8mm,shorten >= 8mm]
    (b11) -- (e11);
    \draw
    (b11) -- ++(0.2,-0.4);
    \draw[dotted]
    (b11)++(0.2,-0.4) -- ++(0.1,-0.2);
    \draw (b) -- node[blue!50!black,fill=none,anchor=south east,font=\footnotesize] {$1$} node[red!50!black,fill=none,anchor=north west,font=\footnotesize] {$1$} (a);
    \draw (a) -- ++(1,-0.3333);
    \draw[dotted] (a)++(1,-0.3333) -- ++(0.5,-0.1667);
    \draw (a) -- ++(1,0.25);
    \draw[dotted] (a)++(1,0.25) -- ++(0.5,0.125);
    \node at (a) {$B_{\mathcal{H}}(0,-1)$};
    \node at (b) {$B_{\mathcal{H}}(0,0)$};
    \node (b0Text) at (b0) {$B_{\mathcal{H}}(0,1)$};
    \node (b1Text) at (b1) {$B_{\mathcal{H}}(1,1)$};
    \node[anchor=base west,xshift=-1.5mm] at (b1Text.base east) {$=B_{\mathcal{H}}(1,2)\vee B_{\mathcal{H}}(3,2)$};
    \node (b00Text) at (b00) {$B_{\mathcal{H}}(0,2)$};
    \node (b01Text) at (b01) {$B_{\mathcal{H}}(2,2)$};
    \node (b10Text) at (b10) {$B_{\mathcal{H}}(1,2)$};
    \node (b11Text) at (b11) {$B_{\mathcal{H}}(3,2)$};
    \node (e00Text) at (e00) {$\{0\}$};
    \node (e01Text) at (e01) {$\{2\}$};
    \node (e10Text) at (e10) {$\{1\}$};
    \node (e11Text) at (e11) {$\{3\}$};
    \draw [decorate,decoration={brace,amplitude=5pt,mirror,raise=0mm},blue!50!black]
    (b00Text.225) -- (e00Text.90) node[midway,anchor=south east,xshift=-2mm,inner sep=1pt,font=\footnotesize] {$1/4$};
    \draw [decorate,decoration={brace,amplitude=5pt,mirror,raise=0mm},blue!50!black]
    (b01Text.225) -- (e01Text.90) node[midway,anchor=south east,xshift=-2mm,inner sep=1pt,font=\footnotesize] {$1/4$};
    \draw [decorate,decoration={brace,amplitude=5pt,mirror,raise=0mm},blue!50!black]
    (b10Text.225) -- (e10Text.90) node[midway,anchor=south east,xshift=-2mm,inner sep=1pt,font=\footnotesize] {$1/4$};
    \draw [decorate,decoration={brace,amplitude=5pt,mirror,raise=0mm},blue!50!black]
    (b11Text.225) -- (e11Text.90) node[midway,anchor=south east,xshift=-2mm,inner sep=1pt,font=\footnotesize] {$1/4$};
    \node[right] at (l) {$B_{\mathcal{B}}(\bullet,2^{1})=B_{\mathcal{H}}(\bullet,1)\colon$};
    \node[right] at (l0) {$B_{\mathcal{B}}(\bullet,2^{-0})=B_{\mathcal{H}}(\bullet,0)\colon$};
    \node[right] at (l1) {$B_{\mathcal{B}}(\bullet,2^{-1})=B_{\mathcal{H}}(\bullet,-1)\colon$};
    \node[right] at (l2) {$B_{\mathcal{B}}(\bullet,2^{-2})=B_{\mathcal{H}}(\bullet,-2)\colon$};
    \node[right] at (l3) {$B_{\mathcal{B}}(\bullet,0)=B_{\mathcal{H}}(\bullet,\infty)\colon$};
  \end{tikzpicture}\vspace{-2mm}
  \caption{The metric $d_{\mathcal{B}}$ (blue, top) on $\mathcal B^1$ and $d_{\mathcal{H}}$ (red, bottom) on $\mathcal H^1$ for $K=\QQ_2$.}
  \label{fig:metric}
\end{figure}

\begin{example}
  \label{ex:metric}
  \cref{fig:metric} illustrates the metric $d_{\mathcal{B}}$ on $\mathcal B^1$ and $d_{\mathcal{H}}$ on $\mathcal H^1$ over $\QQ_2$.  Note in particular how $d_{\mathcal{B}}(B_{\mathcal{B}}(a,0),d_{\mathcal{B}}(a',0)) = 2\cdot |a-a'|_2$ for $a,a'\in\{0,1,2,3\}$ as shown in \cref{lem:metricExtendingFieldMetric}.  
  Also, any two discs are connected by a unique geodesic path, which can be read off directly from the tree structure, as described in \cref{lem:geodesics}.
\end{example}

\subsection{Convexity in Polydisc Spaces}\label{sec:convexity}

The geodesic structure described above provides a natural notion of convexity in polydisc spaces. As in general metric geometry, convexity is defined by requiring that geodesics between points remain inside the set.

\begin{definition}
  \label{def:convexity}
  A subset $S\subseteq \mathcal{B}^n$ is said to be \emph{convex} if for any $B, B' \in S$ there is a geodesic $\gamma\colon I\rightarrow\mathcal B^n$ connecting $B$ to $B'$ with $\gamma(I)\subseteq S$.
  The \emph{convex hull} $\conv(S)$ of a subset $S \subseteq \mathcal B^n$ is the smallest convex subset of $\mathcal B^n$ containing the distinguished geodesics $\gamma_{BB'}$ for $B,B'\in S$.
\end{definition}

\begin{remark}
  \label{rem:convexity}
  In \cite[Definition 3.6.5]{BuragoBuragoIvanov2001}, convexity is defined by requiring that every geodesic between two points lies in the set. In our setting, this distinction becomes relevant when geodesics are not unique. If $n=1$ or $1<p<\infty$, then $(\mathcal{B}^n, d_{\mathcal{B}}^p)$ and $(\mathcal{H}^n, d_{\mathcal{H}}^p)$ are uniquely geodesic by \cref{lem:geodesics} and $\conv(S)$ is the unique convex set containing $S$. On the other hand, if $n>1$ and $p\in \{1,\infty\}$, geodesics need not be unique and different choices of geodesics lead to different convex sets. The definition above selects a canonical convex hull by working with the distinguished geodesics $\gamma_{BB'}$. In contrast, if we allowed all geodesics for $p = 1, \infty$ then in higher dimensions the convex hull of two points would no longer be a one dimensional object, similarly to convex hulls in $\mathbb R ^ n$ when working with the $\ell^1$ and $\ell^\infty$ metrics.
\end{remark}

This choice of convexity ensures that convex hulls maintain a tractable combinatorial structure that is relevant to the underlying tree geometry of $\mathcal{B}^n$.

\begin{example}
  \label{ex:convexHull1d}
  In \cref{fig:metric}, the edges labelled with their distances form the convex hull of $B_{\mathcal{B}}(0,0)$, $B_{\mathcal{B}}(1,0)$, $B_{\mathcal{B}}(2,0)$, $B_{\mathcal{B}}(3,0)$ and $B_{\mathcal{B}}(0,-1)=B_{\mathcal{H}}(0,2)$.
\end{example}

\begin{example}
  \label{ex:convexHull2d}
  Let $S\subseteq\mathcal H^2$ consist of the following four polydiscs over the $2$-adic numbers $\QQ_2$:
  \begin{equation*}
    B_\mathcal{H}(0,1)\times B_\mathcal{H}(0,1),\quad
    B_\mathcal{H}(1,1)\times B_\mathcal{H}(0,1),\quad
    B_\mathcal{H}(0,1)\times B_\mathcal{H}(1,1),\quad
    B_\mathcal{H}(1,1)\times B_\mathcal{H}(1,1).
  \end{equation*}
  \cref{fig:convexHull2d} illustrates $\conv(S)$. The polydiscs within a box are equal to each other.  Horizontal and vertical lines are geodesics, e.g., $B_\mathcal{H}(1,1)\times B_\mathcal{H}(0,0)$ lies on the red geodesic from $B_\mathcal{H}(1,1)\times B_\mathcal{H}(0,1)$ to $B_\mathcal{H}(1,1)\times B_\mathcal{H}(1,1)$.  Points within the shaded area also belong to $\conv(S)$, e.g. $B_\mathcal{H}(1,\frac12)\!\times\! B_\mathcal{H}(0,\frac12)$ lies on the red dashed geodesic from $B_\mathcal{H}(1,\frac12)\!\times\! B_\mathcal{H}(0,0)$ to $B_\mathcal{H}(1,\frac12)\!\times\! B_\mathcal{H}(0,1)$.
\end{example}

\begin{figure}[t]
  \centering
  \begin{tikzpicture}
    \fill[blue!20] (-6.5,-2.5) rectangle (6.5,2.5);

    \node[draw, fill=blue!10, inner sep=2pt] (B0000) at (0,0)
    {
      \begin{tikzpicture}[every node/.style={font=\footnotesize}]
        \node[anchor=south west,xshift=1mm,yshift=1mm] (b00) {$B(0,0)\!\times\! B(0,0)$};
        \node[anchor=south east,xshift=-1mm,yshift=1mm] (b10) {$B(1,0)\!\times\! B(0,0)$};
        \node[anchor=north west,xshift=1mm,yshift=-1mm] (b01) {$B(0,0)\!\times\! B(1,0)$};
        \node[anchor=north east,xshift=-1mm,yshift=-1mm] (b11) {$B(1,0)\!\times\! B(1,0)$};
        \draw[draw opacity=0]
        (b00) -- node[sloped] {$=$} (b10)
        (b00) -- node[sloped] {$=$} (b01)
        (b10) -- node[sloped] {$=$} (b11)
        (b01) -- node[sloped] {$=$} (b11);
      \end{tikzpicture}
    };
    \node[draw, fill=blue!10, inner sep=2pt] (B0001) at (0,2.5)
    {
      \begin{tikzpicture}[every node/.style={font=\footnotesize}]
        \node[anchor=west,xshift=1mm] (b00) {$B(0,0)\!\times\! B(0,1)$};
        \node[anchor=east,xshift=-1mm] (b10) {$B(1,0)\!\times\! B(0,1)$};
        \draw[draw opacity=0]
        (b00) -- node[sloped] {$=$} (b10);
      \end{tikzpicture}
    };
    \node[draw, fill=blue!10, inner sep=2pt] (B0011) at (0,-2.5)
    {
      \begin{tikzpicture}[every node/.style={font=\footnotesize}]
        \node[anchor=west,xshift=1mm] (b01) {$B(0,0)\!\times\! B(1,1)$};
        \node[anchor=east,xshift=-1mm] (b11) {$B(1,0)\!\times\! B(1,1)$};
        \draw[draw opacity=0]
        (b01) -- node[sloped] {$=$} (b11);
      \end{tikzpicture}
    };
    \node[draw, fill=blue!10, inner sep=2pt] (B0100) at (6.5,0)
    {
      \begin{tikzpicture}[every node/.style={font=\footnotesize}]
        \node[anchor=south,yshift=1mm] (b00) {$B(0,1)\!\times\! B(0,0)$};
        \node[anchor=north,yshift=-1mm] (b01) {$B(0,1)\!\times\! B(1,0)$};
        \draw[draw opacity=0]
        (b00) -- node[sloped] {$=$} (b01);
      \end{tikzpicture}
    };
    \node[draw, fill=blue!10, inner sep=2pt] (B1100) at (-6.5,0)
    {
      \begin{tikzpicture}[every node/.style={font=\footnotesize}]
        \node[anchor=south,yshift=1mm] (b10) {$B(1,1)\!\times\! B(0,0)$};
        \node[anchor=north,yshift=-1mm] (b11) {$B(1,1)\!\times\! B(1,0)$};
        \draw[draw opacity=0]
        (b10) -- node[sloped] {$=$} (b11);
      \end{tikzpicture}
    };

    \node[draw, fill=blue!10, inner sep=2pt, font=\footnotesize] (B0101) at (6.5,2.5) {$B(0,1)\!\times\! B(0,1)$};
    \node[draw, fill=blue!10, inner sep=2pt, font=\footnotesize] (B1101) at (-6.5,2.5) {$B(1,1)\!\times\! B(0,1)$};
    \node[draw, fill=blue!10, inner sep=2pt, font=\footnotesize] (B0111) at (6.5,-2.5) {$B(0,1)\!\times\! B(1,1)$};
    \node[draw, fill=blue!10, inner sep=2pt, font=\footnotesize] (B1111) at (-6.5,-2.5) {$B(1,1)\!\times\! B(1,1)$};

    \draw[very thick]
    (B0000) -- (B0100)
    (B0000) -- (B0001)
    (B0000) -- (B1100)
    (B0000) -- (B0011);
    \draw[very thick]
    (B0100) -- (B0101)
    (B0100) -- (B0111)
    (B0001) -- (B0101)
    (B0001) -- (B1101)
    (B0011) -- (B0111)
    (B0011) -- (B1111);
    \draw[very thick, red!90!black]
    (B1100) -- (B1101)
    (B1100) -- (B1111);

    \draw[draw opacity=0] (B0000) -- node (bb0) {} (B1100);
    \draw[draw opacity=0] (B0001) -- node (bb1) {} (B1101);
    \fill[white,draw=black] (bb0) circle (2.5pt);
    \fill[white,draw=black] (bb1) circle (2.5pt);
    \node[anchor=north, font=\footnotesize] at (bb0) {$B(1,\frac12)\!\times\! B(0,0)$};
    \node[anchor=south, font=\footnotesize] at (bb1) {$B(1,\frac12)\!\times\! B(0,1)$};
    \draw[very thick, red, dashed] (bb0) -- node (bb2) {} (bb1);
    \fill[white,draw=black] (bb2) circle (2.5pt);
    \node[anchor=west, font=\footnotesize] at (bb2) {$B(1,\frac12)\!\times\! B(0,\frac12)$};
  \end{tikzpicture}
  \caption{The convex hull of four points in $\mathcal B^2$ over $\QQ_2$ from \cref{ex:convexHull2d}.}
  \label{fig:convexHull2d}
\end{figure}
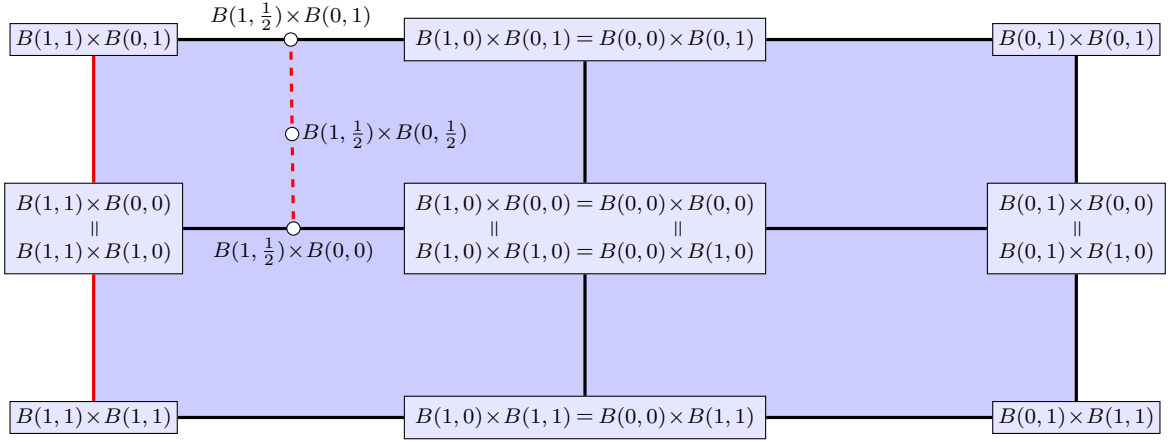

\subsection{Tangents on Polydisc Spaces}

To develop calculus on $\mathcal{B}^n$, we introduce a notion of tangent directions that captures how points can be infinitesimally perturbed within the space. This will be used further on in \cref{sec:optimisation} to define directional derivatives and guide optimisation procedures.

In contrast to Euclidean spaces, the geometry of $\mathcal{B}^n$ is inherently tree-like, as illustrated in \cref{ex:polydiscSpace}. Tangent directions should then be interpreted in terms of geodesic paths emanating from a point, rather than linear directions.

\begin{definition}
  \label{def:tangentSpace}
  A \emph{tangent direction} at $B\in\mathcal B^n$ is an equivalence class of geodesics $\gamma_{BB'}$ emanating from $B$, where two geodesics are equivalent if they share a common initial segment, i.e., the set of tangent directions at $B\in\mathcal B^n$ is given by
  \begin{equation*}
    T_B\mathcal B^n \coloneqq \bigslant{ \Big\{ \gamma_{BB'}\colon [0,1]\rightarrow\mathcal B^n \bigmid B'\in \mathcal B^n\setminus\{B\} \Big\} }{\sim},
  \end{equation*}
  where $\gamma_1\sim \gamma_2$ if and only if $\gamma_1([0,s_1]) = \gamma_2([0,s_2])$ for some $s_1,s_2>0$.  We write $\vec{v}_{BB'}\coloneqq \overline{\gamma_{BB'}}\in T_B\mathcal B^n$. We set $T_B \mathcal H^n = T_B \mathcal B^n$.
\end{definition}

Definition \ref{def:tangentSpace} is inspired by the notion of tangent directions in Berkovich spaces \cite{baker2008introduction} and more generally by the concept of directions in metric geometry \cite[Section 3.6.6]{BuragoBuragoIvanov2001}. In classical metric geometry, directions are defined via angles between geodesics: two geodesics have the same direction if the angle between them is zero, which is not well suited to polydisc spaces. As we will see in the example below, distinct (non-equivalent) geodesics can have angle zero or the angle may fail to be well-defined.

\begin{example}
    Consider the metric $d = d_{\mathcal B}^p$ for $1 < p < \infty$, so that 
    $\mathcal B$ is uniquely geodesic. Let 
    \begin{align*}
    B=B_{\mathcal B}((0,0),(1,1)),\;\;
    B'=B_{\mathcal B}((0,0),(2,1)),\;\;
    B''=B_{\mathcal B}((0,0),(1,2)) \in \mathcal B^2
    \end{align*}
    Let $\gamma'$ and $\gamma''$ be the unit-speed geodesics from $B$ to $B'$ and $B$ to $B''$, respectively and both defined on the interval $[0,1]$ (so that, for example, $\gamma'(t) = B_{\mathcal B}((0,0), (1+t,1))$). These geodesics are not equivalent in the sense of Definition \ref{def:tangentSpace}. 
    
    Using the definition of angle from \cite[Definition 3.6.26]{BuragoBuragoIvanov2001}, the angle between $\gamma'$ and $\gamma''$ is given by
    $$
    \lim\limits_{s,t \to 0} \mathrm{arccos}\frac{d(\gamma'(s),B)^2+d(\gamma''(t),B)^2 - d(\gamma'(s),\gamma''(t))^2}{2d(\gamma'(s),B)d(\gamma''(t),B)} = \mathrm{arccos} \lim\limits_{s,t \to 0}\frac{s^2+t^2 - (s^p+t^p)^{2/p}}{2st},
    $$
    if it exists. Along lines $t=sx$, as we allow $x$ to vary, we have 
    $$
    \lim\limits_{s \to 0}\frac{s^2+(sx)^2 - (s^p+(sx)^p)^{2/p}}{2s^2x} = \frac{1+x^2-(1+x^p)^{2/p}}{2x}.
    $$
    As this depends on $x$ (unless $p=1,2$) the limit does not exist and so the angle between $\gamma'$ and $\gamma''$ is not well defined.

    We can obtain a more robust notion if using the \emph{upper angle} (\cite[Definition 3.6.32]{BuragoBuragoIvanov2001}), which depends only on arbitrarily small initial segments of geodesics and so respects the equivalence relation in Definition \ref{def:tangentSpace} and thus agrees with our notion of tangent directions. First notice that for any paths $\gamma', \gamma'': [0,1] \to \mathcal B^n$ emanating from $B$, the angle between $\gamma'$ and $\gamma''$ depends only on $\gamma_i\mid_{[0,\epsilon]}$ for small $\epsilon >0$. Equivalent geodesics thus always have the same upper angle. 
    
    Conversely, if two geodesics are not equivalent, then their upper angle is strictly positive: suppose that $\gamma' = \gamma_{BB'}$ and $\gamma''=\gamma_{BB''}$ are non-equivalent geodesics. By shrinking the geodesics, we may assume that $B' = B_1' \times \dots \times B_n'$ and $B''=B_1'' \times \dots \times B_n''$ are such that, for $B= B_1 \times \dots \times B_n$, all $B_i$ and $B_i'$ may be compared (in the sense that $B_i \geq B_i'$ or $B_i \leq B_i'$), as can all $B_i$ and $B_i''$. The \emph{comparison angle} (\cite[Definition 3.6.25]{BuragoBuragoIvanov2001}) between $B'$ and $B''$ at $B$ is defined to be
    $$\angle(B'BB'') =  \mathrm{arccos}\frac{d(B',B)^2+d(B'',B)^2 - d(B',B'')^2}{2d(B',B)d(B'',B)}.$$
    
    Denote by $\frac{1}{2}B_i'$ and $\frac{1}{2}B_i''$ the midpoints of a geodesic of any speed from $B_i$ to $B_i'$ and $B_i$ to $B_i''$ in $\mathcal B^1$ respectively. Denote $\frac{1}{2}B' = \prod \frac{1}{2}B_i'$ and $\frac{1}{2}B'' = \prod \frac{1}{2}B_i''$. Then $d(\frac{1}{2}B',B) = \frac{1}{2}d(B',B)$ (and similarly for $B''$). Moreover, notice that in $\mathcal B^1$, $d(\frac{1}{2}B_i',\frac{1}{2}B_i'') = \frac{1}{2}d(B_i',B_i'')$. Therefore, $d(\frac{1}{2}B',\frac{1}{2}B'')=\frac{1}{2}d(B',B'')$, and so $\angle((\frac{1}{2}B')B(\frac{1}{2}B'')) = \angle(B'BB'')$. The upper angle between $\gamma'$ and $\gamma''$ is the upper limit of the comparison angle between points on these geodesics as they approach $\mathcal B$. This argument implies that this upper angle is at least the comparison angle $\angle(B'BB'')$. It therefore suffices to show that this angle is not 0. For it to be 0, we must have $d(B',B'') = |d(B',B)-d(B'',B)|$. Without loss of generality, assume $d(B',B)\geq d(B'',B)$. Then $d(B',B)=d(B'',B)+d(B',B'')$ and so the composition of the geodesic from $B'$ to $B''$ with the geodesic from $B''$ to $B$ is a geodesic from $B'$ to $B$. By uniqueness of geodesics this must be $\gamma'$. But then $\gamma'$ and $\gamma''$ are equivalent, a contradiction.
\end{example}

This example shows that the equivalence relation in Definition \ref{def:tangentSpace} captures precisely the intrinsic notion of direction in $\mathcal{B}^n$, while also avoiding the degeneracies of angle-based definitions.  

We next describe tangent directions explicitly in dimension one. Removing a disc $B$ from the tree $\mathcal B^1$ separates the remaining space into connected components; these components are precisely the possible directions emanating from $B$. The form of this decomposition depends on whether $B$ is a singleton, an irrational disc, or a rational disc of positive radius. We now propose the following notation to keep track of these components, inspired by the description of tangent directions in Berkovich spaces \cite[Section 3]{baker2008introduction}.

\begin{definition}
  \label{def:tangentDirectionResidueField}
  Let $B=B_{\mathcal{B}}(a,r)\in\mathcal B^1$ or $B=B_{\mathcal{H}}(a,v)\in\mathcal H^1$ be a disc.  If $B$ is rational and $r>0$, we define for each element of the residue field $\bar c\in\mathfrak K$:
  \begin{align*}
    C_{\mathcal B^1\setminus B}(\bar c)&\coloneqq
    \big\{B_{\mathcal{B}}(a',r')\in \mathcal B^1 \mid B_{\mathcal{B}}(a',r')\subsetneq B \text{ and } \overline{\pi^{-\log(r)}\cdot (a-a')}=\bar c \big\},\\
    C_{\mathcal H^1\setminus B}(\bar c)&\coloneqq
    \big\{B_{\mathcal{H}}(a',v')\in \mathcal H^1 \mid B_{\mathcal{H}}(a',v')\subsetneq B \text{ and } \overline{\pi^{-v}\cdot (a-a')}=\bar c \big\},
  \end{align*}
  where $\overline{(\cdot)}$ maps an element of $K$ with non-negative valuation to its residue in $\mathfrak K$.
  If $B$ is irrational, we only define for $0\in\mathfrak K$:
  \begin{align*}
    C_{\mathcal B^1\setminus B}(\bar 0)\coloneqq
    \big\{B'\in \mathcal B^1 \mid B' \subsetneq B \big\}\quad\text{and}\quad
    C_{\mathcal H^1\setminus B}(\bar 0)\coloneqq
    \big\{B' \in \mathcal H^1 \mid B'\subsetneq B \big\}.
  \end{align*}  
  Furthermore, we define for all $B$:
  \begin{align*}
    C_{\mathcal B^1\setminus B}(\infty)\coloneqq
    \big\{B'\in \mathcal B^1 \mid B' \nsubseteq B \big\}\quad\text{and}\quad
    C_{\mathcal H^1\setminus B}(\infty)\coloneqq
    \big\{B' \in \mathcal H^1 \mid B'\nsubseteq B \big\}.
  \end{align*}
  We write \(\vec{v}_{B,c}\coloneqq \vec{v}_{BB'}\) for any \(B'\in C_{\mathcal B^1\setminus B}(c)\) or \(B'\in C_{\mathcal H^1\setminus B}(c)\).
  
\end{definition}

The lemmas that we will present further on (\cref{lem:tangentDirections0}, \ref{lem:tangentDirectionsIrrational} and \ref{lem:tangentDirectionsRational}) will show that this notation is well-defined.

The following example illustrates the notation pictorially and shows the two basic behaviours over $\QQ_2$: rational discs of positive radius branch according to residue classes, while irrational discs have only an inward and an outward direction.

\begin{example}
  \label{ex:tangentDirectionResidueField}
  Consider the space of discs $\mathcal B^1$ over the $2$-adic numbers $\QQ_2$.
  \begin{enumerate}
  \item For $B=B_\mathcal{H}(0,1)\in \mathcal B^1$, $\mathcal B^1\setminus B$ consists of three connected components, see \cref{fig:tangentDirectionsResidueField}(1):
    \begin{align*}
      C_{\mathcal B^1\setminus B}(\infty) &= \{ B_\mathcal{H}(a,v) \mid v<1 \text{ or } a = 1+a' \text{ for some } a'\in\QQ_2 \text{ with }\val(a')>0 \}, \\
      C_{\mathcal B^1\setminus B}(\bar 0) &= \{ B_\mathcal{H}(0+a,v) \mid v>1 \text{ and } a\in \QQ_2 \text{ with } \val(a)>1 \}, \\
      C_{\mathcal B^1\setminus B}(\bar 1) &= \{ B_\mathcal{H}(2+a,v) \mid v>1 \text{ and } a\in \QQ_2 \text{ with } \val(a)>1 \}.
    \end{align*}
  \item For $B=B_\mathcal{H}(0,\frac12)\in \mathcal B^1$, $\mathcal B^1\setminus B$ consists of two connected components, see \cref{fig:tangentDirectionsResidueField}(2):
    \begin{align*}
      C_{\mathcal B^1\setminus B}(\infty) &= \{ B_\mathcal{H}(a,v) \mid v<\textstyle\frac 12 \text{ or } a = 1+a' \text{ for some } a'\in\QQ_2 \text{ with }\val(a')>0 \}, \\
      C_{\mathcal B^1\setminus B}(\bar 0) &= \{ B_\mathcal{H}(0+a,v) \mid v>\textstyle\frac 12 \text{ and } a\in \QQ_2 \text{ with } \val(a)>\frac 12 \}.
    \end{align*}
  \end{enumerate}
\end{example}
\begin{figure}[t]
  \centering
  \begin{tikzpicture}
    \node (left) at (-3.5,0)
    {
      \begin{tikzpicture}
        \useasboundingbox (-2,3.25) rectangle (3,-1.75);
        \node (B00) at (1,2.5) {$B(0,0)$};
        \node (B11) at (2,1.5) {$B(1,1)$};
        \node (B01) at (0,0) {$B(0,1)$};
        \node (B02) at (-1.5,-1) {$B(0,2)$};
        \node (B12) at (1.5,-1) {$B(2,2)$};
        \draw[thick]
        (B00) -- (B11)
        (B00) -- (B01) node[pos=1] (innerUp) {}
        (B01) -- (B02) node[pos=0] (innerDownLeft) {}
        (B01) -- (B12) node[pos=0] (innerDownRight) {};
        \draw[thick]
        (B11) -- ++(-0.333,-0.5) node (veryTop0) {}
        (B11) -- ++(0.333,-0.5) node (veryTop1) {}
        (B00) -- ++(0,0.5) node (up) {}
        (B02) -- ++(-0.333,-0.5) node (down1) {}
        (B02) -- ++(0.333,-0.5) node (down2) {}
        (B12) -- ++(-0.333,-0.5) node (down3) {}
        (B12) -- ++(0.333,-0.5) node (down4) {};
        \draw[thick,dotted]
        (veryTop0.center) -- ++(-0.15,-0.25)
        (veryTop1.center) -- ++(0.15,-0.25)
        (up.center) -- ++(0,0.333)
        (down1.center) -- ++(-0.222,-0.333)
        (down2.center) -- ++(0.222,-0.333)
        (down3.center) -- ++(-0.222,-0.333)
        (down4.center) -- ++(0.222,-0.333);
        \draw[->,very thick,red!70!black,shorten >=15mm]
        (B01) -- (B00) node[right,pos=0.2,font=\footnotesize,xshift=1mm] {$\vec{v}_{B,\infty}$};
        \draw[->,very thick,red!70!black,shorten >=3mm]
        (B01) -- (B02) node[below,pos=0.3,font=\footnotesize,yshift=-1mm] {$\vec{v}_{B,\bar 0}$};
        \draw[->,very thick,red!70!black,shorten >=3mm]
        (B01) -- (B12) node[below,pos=0.3,font=\footnotesize,yshift=-1mm] {$\vec{v}_{B,\bar 1}$};
        \draw[blue] plot [smooth] coordinates
        {
          ($(innerUp)+(-1.5,3)$)
          ($(innerUp)+(-0.5,0.25)$)
          ($(innerUp)+(0,-0.05)$)
          ($(innerUp)+(0.5,-0.06)$)
          ($(innerUp)+(2.6,-0.07)$)
        };
        \draw[blue] plot [smooth] coordinates
        {
          ($(innerDownLeft)+(-2.25,0.3)$)
          ($(innerDownLeft)+(0.05,0.05)$)
          ($(innerDownLeft)+(0.35,-1.75)$)
        };
        \draw[blue] plot [smooth] coordinates
        {
          ($(innerDownRight)+(2.25,0.3)$)
          ($(innerDownRight)+(-0.05,0.05)$)
          ($(innerDownRight)+(-0.35,-1.75)$)
        };
        \node[blue,anchor=west] at ($(innerUp)+(-1.4,2.75)$) {$C_{\mathcal B^1\setminus B}(\infty)$};
        \node[below,blue,yshift=0.25mm] at ($(innerDownLeft)+(-1.75,0.25)$) {$C_{\mathcal B^1\setminus B}(\bar 0)$};
        \node[below,blue,yshift=0.25mm] at ($(innerDownRight)+(1.75,0.25)$) {$C_{\mathcal B^1\setminus B}(\bar 1)$};
      \end{tikzpicture}
    };
    \node[anchor=north west] at (left.north west) {$(1)$};
    \node (right) at (3.5,0)
    {
      \begin{tikzpicture}
        \useasboundingbox (-2,3.25) rectangle (3,-1.75);
        \node (B00) at (1,2.5) {$B(0,0)$};
        \node (B11) at (2,1.5) {$B(1,1)$};
        \node (B01) at (0,0) {$B(0,1)$};
        \node (B02) at (-1.5,-1) {$B(0,2)$};
        \node (B12) at (1.5,-1) {$B(2,2)$};
        \draw[thick]
        (B00) -- (B11)
        (B00) -- (B01) node[pos=0.5] (pt) {}
        (B01) -- (B02)
        (B01) -- (B12);
        \draw[thick]
        (B11) -- ++(-0.333,-0.5) node (veryTop0) {}
        (B11) -- ++(0.333,-0.5) node (veryTop1) {}
        (B00) -- ++(0,0.5) node (up) {}
        (B02) -- ++(-0.333,-0.5) node (down1) {}
        (B02) -- ++(0.333,-0.5) node (down2) {}
        (B12) -- ++(-0.333,-0.5) node (down3) {}
        (B12) -- ++(0.333,-0.5) node (down4) {};
        \draw[thick,dotted]
        (veryTop0.center) -- ++(-0.15,-0.25)
        (veryTop1.center) -- ++(0.15,-0.25)
        (up.center) -- ++(0,0.333)
        (down1.center) -- ++(-0.222,-0.333)
        (down2.center) -- ++(0.222,-0.333)
        (down3.center) -- ++(-0.222,-0.333)
        (down4.center) -- ++(0.222,-0.333);
        \draw[->,very thick,red!70!black,shorten >=3mm]
        (pt) -- (B00) node[right,pos=0.525,font=\footnotesize,xshift=0.5mm] {$\vec{v}_{B,\infty}$};
        \draw[->,very thick,red!70!black,shorten >=3mm]
        (pt) -- (B01) node[right,pos=0.65,font=\footnotesize,xshift=1mm] {$\vec{v}_{B,\bar 0}$};
        \fill[white] (pt) circle (5pt);
        \fill[black] (pt) circle (2pt);
        \draw[blue] plot [smooth] coordinates
        {
          ($(pt)+(-1.5,2)$)
          ($(pt)+(0,0.15)$)
          ($(pt)+(2.5,-1.5)$)
        };
        \draw[blue] plot [smooth] coordinates
        {
          ($(pt)+(-2.5,0.25)$)
          ($(pt)+(0,-0.15)$)
          ($(pt)+(2.5,-1.75)$)
        };
        \node[blue,anchor=south west,xshift=1mm,yshift=-4mm] at ($(pt)+(-1.5,2)$) {$C_{\mathcal B^1\setminus B}(\infty)$};
        \node[blue,anchor=north west,yshift=-1mm] at ($(pt)+(-2.5,0.25)$) {$C_{\mathcal B^1\setminus B}(\bar 0)$};
      \end{tikzpicture}
    };
    \node[anchor=north west] at (right.north west) {$(2)$};
  \end{tikzpicture}
  \caption{Connected components of $\mathcal B^1\setminus B$ in \cref{ex:tangentDirectionResidueField}.}
  \label{fig:tangentDirectionsResidueField}
\end{figure}
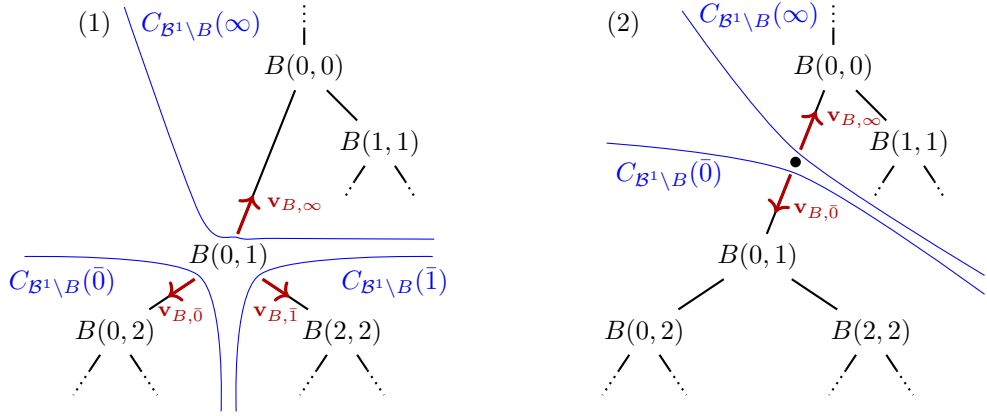

We now verify that the sets introduced above are exactly the connected components of the complement of $B$ and thus determine the tangent directions at $B$. The verification splits naturally into three cases: singleton discs; irrational discs; and rational discs of positive radius. The singleton case turns out to be degenerate: a point of \(K\subset \mathcal B^1\) has only one outgoing direction in \(\mathcal B^1\).

\begin{lemma}
  \label{lem:tangentDirections0}
  Let $B=B_{\mathcal{B}}(a,0)\in\mathcal B^1$.  Then $C_{\mathcal B^1\setminus B}(\infty)$ is the only connected component of $\mathcal B^1\setminus B$ and we have $\vec{v}_{BB_1}=\vec{v}_{BB_2}$ for all $B_1,B_2\in C_{\mathcal B^1\setminus B}(\infty)$.
\end{lemma}
\begin{proof}
  To show that $C_{\mathcal B^1\setminus B}(\infty)$ is a single connected component, let $B_1,B_2\in\mathcal B^1$ with $B_1\neq B\neq B_2$ and consider the path $\gamma_{B_1B_2}$.  If $B\in \Im(\gamma_{B_1B_2})$, then either $B\in \Im(\gamma_{B_1,B_1\vee B_2})$ or $B\in \Im(\gamma_{\gamma_{B_2,B_1\vee B_2}})$.  This implies that either $B_1\subsetneq B$ or $B_2\subsetneq B$, contradicting that $r=0$.

  To show that $\vec{v}_{BB_1}=\vec{v}_{BB_2}$ for $B_1,B_2\in\mathcal B^1$ with $B_1\neq B\neq B_2$, note that $r=0$ implies $B_i\vee B\neq B$.  Since $\gamma_{BB_i}\sim \gamma_{B,B_i\vee B}$ we may assume without loss of generality that $B_i\supsetneq B$.  By \cref{lem:recentre}, we then have either $B_1\subseteq B_2$ or $B_2\subseteq B_1$, which shows that $\gamma_{BB_1}\sim\gamma_{BB_2}$.
\end{proof}

For irrational discs, the complement has two components: the discs strictly contained in $B$ and the discs not contained in $B$, which yields exactly two tangent directions.

\begin{lemma}
  \label{lem:tangentDirectionsIrrational}
  Let $B=B_{\mathcal{B}}(a,r)\in\mathcal B^1$ or $B=B_{\mathcal{H}}(a,v)\in\mathcal H^1$ be an irrational disc.  Then
  \begin{enumerate}
  \item[(i)] $\mathcal B^1\setminus B = \bigcup_{c\in\{\bar 0,\infty\}} C_{\mathcal B^1\setminus B}(c)$, and each $C_{\mathcal B^1\setminus B}(c)$ is a connected component of $\mathcal B^1\setminus B$,
  \item[(ii)] $\mathcal H^1\setminus B = \bigcup_{c\in\{\bar 0,\infty\}} C_{\mathcal H^1\setminus B}(c)$ and each $C_{\mathcal H^1\setminus B}(c)$ is a connected component of $\mathcal H^1\setminus B$,
  \item[(iii)] $\vec{v}_{BB_1}=\vec{v}_{BB_2}$ if and only if $B_1,B_2\in C_{\mathcal B^1\setminus B}(c)$ or $B_1,B_2\in C_{\mathcal H^1\setminus B}(c)$ for some $c\in\{0,\infty\}$.
  \end{enumerate}
\end{lemma}
\begin{proof}\
  \begin{enumerate}
  \item[(i)] We have $\mathcal B^1\setminus B = \bigcup_{c\in\{\bar 0,\infty\}} C_{\mathcal B^1\setminus B}(c)$ by definition, and the fact that $C_{\mathcal B^1\setminus B}(\infty)$ is connected follows similarly as in \cref{lem:tangentDirections0}.

    To see that $C_{\mathcal B^1\setminus B}(\bar 0)$ is a connected component, let $B_1, B_2 \in C_{\mathcal B^1\setminus B}(\bar 0)$ with $B_1\neq B_2$ and consider the path $\gamma_{B_1B_2}$. Since $B_1,B_2\subseteq B$, we have $B_1\vee B_2\subseteq B$ by definition.  Moreover, the radius of $B_1\vee B_2$ must be rational, hence for $B_1\vee B_2\subsetneq B$. This shows that $\Im(\gamma_{B_1B_2})\subseteq C_{\mathcal B^1\setminus B}(\bar 0)$.
  \item[(ii)] Follows similar to (i).
  \item[(iii)] The fact that $\vec{v}_{B B_1}= \vec{v}_{B B_2}$ for $B_1, B_2\in C_{\mathcal B^1\setminus B}(\bar 0)$ or $B_1, B_2\in C_{\mathcal H^1\setminus B}(\bar 0)$ follows from $B_1\vee B_2\subsetneq B$ and hence $\gamma_{B B_1}\sim \gamma_{B,B_1\vee B_2} \sim \gamma_{B B_2}$.
  
  The proof that $\vec{v}_{B B_1}= \vec{v}_{B B_2}$ for $B_1, B_2\in C_{\mathcal B^1\setminus B}(\infty)$ or $B_1, B_2\in C_{\mathcal H^1\setminus B}(\infty)$ is similar to that of \cref{lem:tangentDirections0}.

    Finally, $\vec{v}_{B B_1}\neq \vec{v}_{B B_2}$ for $B_1\in C_{\mathcal B^1\setminus B}(\bar 0)$ and $B_2\in C_{\mathcal B^1\setminus B}(\infty)$ or $B_1\in C_{\mathcal H^1\setminus B}(\bar 0)$ and $B_2\in C_{\mathcal H^1\setminus B}(\infty)$ follows from the fact that the sets are disjoint. \qedhere
  \end{enumerate}
\end{proof}

For rational discs of positive radius, the inward component splits according to residue classes, so the tangent directions are indexed by $\mathbb P^1_{\mathfrak K}=\mathfrak K\cup\{\infty\}$, with $\infty$ corresponding to the outward direction.

\begin{lemma}
  \label{lem:tangentDirectionsRational}
  Let $B=B_{\mathcal{B}}(a,r)\in\mathcal B^1$, $r>0$, or $B=B_{\mathcal{H}}(a,v)\in\mathcal H^1$ be a rational disc.  Then
  \begin{enumerate}
  \item[(i)] $\mathcal B^1\setminus B = \bigcup_{c\in\PP^1_{\mathfrak K}} C_{\mathcal B^1\setminus B}(c)$ and each $C_{\mathcal B^1\setminus B}(c)$ is a connected component of $\mathcal B^1\setminus B$;
  \item[(ii)] $\mathcal H^1\setminus B = \bigcup_{c\in\PP^1_{\mathfrak K}} C_{\mathcal H^1\setminus B}(c)$ and each $C_{\mathcal H^1\setminus B}(c)$ is a connected component of $\mathcal H^1\setminus B$; and
  \item[(iii)] $\vec{v}_{BB_1}=\vec{v}_{BB_2}$ if and only if $B_1,B_2\in C_{\mathcal B^1\setminus B}(c)$ or $B_1,B_2\in C_{\mathcal H^1\setminus B}(c)$ for some $c\in\PP^1_{\mathfrak K}$.
  \end{enumerate}
\end{lemma}
\begin{proof}\
  \begin{enumerate}
  \item[(i)] We have $\mathcal B^1\setminus B = \bigcup_{c\in\PP_{\mathfrak K}^1} C_{\mathcal B^1\setminus B}(c)$ by definition, and the fact that $C_{\mathcal B^1\setminus B}(\infty)$ is connected follows similarly as in \cref{lem:tangentDirections0}.

    To show that $C_{\mathcal B^1\setminus B}(\bar c)$ is connected for $\bar c\in\mathfrak K$, let $B_1=B_{\mathcal{B}}(a_1,r_1), B_2=B_{\mathcal{B}}(a_2,r_2) \in C_{\mathcal B^1\setminus B}(\bar c)$ with $B_1\neq B_2$ and consider some $B'=B_{\mathcal{B}}(a',r')\in\Im(\gamma_{B_1B_2})$.  We now show that $B'\in C_{\mathcal B^1\setminus B}(\bar c)$, i.e., $B'\subseteq B$ and $\overline{t^{-\log(r)}(a-a')}=\bar c$.  The first follows from $B'\subseteq B_1\vee B_2\subseteq B$, in which the first inclusion follows from $B'\in\Im(\gamma_{B_1B_2})$ and the last inclusion follows from $B_1,B_2\subseteq B$ and the definition of join.  For the second, assume that $\overline{t^{-\log(r)}(a-a')}\neq \bar c$.  Pick $i\in\{1,2\}$ such that $B_i\subseteq B'$. Then we have
    \begin{equation*}
      0\neq \overline{t^{-\log(r)}(a-a')} - \overline{t^{-\log(r)}(a-a_i)} =\overline{t^{-\log(r)}(a_i-a')},
    \end{equation*}
    which implies that $|a_i-a'|\geq r$.  This contradicts $B_i,B'\subsetneq B$.
  \item[(ii)] Follows similar to (i).
  \item[(iii)] The fact that $\vec{v}_{B B_1}= \vec{v}_{B B_2}$ for $B_1, B_2\in C_{\mathcal B^1\setminus B}(\bar c)$ for some $\bar c\in\PP_{\mathfrak K}^1$ follows from $B_1\vee B_2\in C_{\mathcal B^1\setminus B}(\bar c)$ and hence $\gamma_{B B_1}\sim \gamma_{B,B_1\vee B_2} \sim \gamma_{B B_2}$.

    The proof that $\vec{v}_{B B_1}= \vec{v}_{B B_2}$ for $B_1, B_2\in C_{\mathcal B^1\setminus B}(\infty)$ or $B_1, B_2\in C_{\mathcal H^1\setminus B}(\infty)$ is similar to that of \cref{lem:tangentDirections0}.

    Finally, $\vec{v}_{B B_1}\neq \vec{v}_{B B_2}$ for $B_1\in C_{\mathcal B^1\setminus B}(\bar c_1)$ and $B_2\in C_{\mathcal B^1\setminus B}(\bar c_2)$ for $\bar c_1\neq \bar c_2$ follows from the fact that the sets are disjoint. \qedhere
  \end{enumerate}
\end{proof}

Combining \cref{lem:tangentDirections0}, \ref{lem:tangentDirectionsIrrational}, and \ref{lem:tangentDirectionsRational} we obtain a complete classification of tangent directions in dimension one. In particular, the tangent space is determined by whether the base disc is a point, an irrational disc, or a rational disc of positive radius.

\begin{proposition}
  \label{prop:tangentSpaceDisk}
  The set of tangent directions $T_B=T_B\mathcal B^1$ or $T_B=T_B\mathcal H^1$ around a disc $B\in \mathcal B^1$ or $B\in \mathcal H^1$ is given by
  \begin{equation*}
    \label{eq:tangentSpaceDisk}
    T_B = \{\vec{v}_{B,\bar c}\mid \bar c\in O \},
  \end{equation*}
  where $O\subseteq\PP_{\mathfrak K}^1$ is given by
  \begin{enumerate}
  \item[(i)] $O=\{\infty\}$ if $B$ has zero radius,
  \item[(ii)] $O=\{0,\infty\}$ if $B$ is irrational,
  \item[(iii)] $O=\PP_{\mathfrak K}^1$ if $B$ is rational and has positive radius.
  \end{enumerate}
\end{proposition}
\begin{proof}
  The statement follows by identifying tangent directions with the connected components of the complement of $B$, as described in \cref{lem:tangentDirections0,lem:tangentDirectionsIrrational,lem:tangentDirectionsRational}.
\end{proof}

The one-dimensional description extends to polydisc spaces by keeping track of both a tangent direction and a relative speed (in each coordinate). Thus a tangent direction in $\mathcal B^n$ is not only a choice of direction in each factor $\mathcal B^1$, but also a choice of how fast the geodesic moves in each coordinate, which leads to the following product description.

\begin{proposition}
  \label{prop:tangentSpacePolydisk}
  The set of tangent directions $T_B=T_B\mathcal B^n$ or $T_B=T_B\mathcal H^n$ at a polydisc $B=B_1\times\dots\times B_n\in\mathcal B^n$ or $B=B_1\times\dots\times B_n\in\mathcal H^n$ is given by
  \begin{equation*}
    T_B = \bigslant{
            \Bigl(
              O\times (\mathbb{R}_{\geq0}^n\setminus\{0\})
            \Bigr)}{\sim}
  \end{equation*}
  where $O\coloneqq O_1\times\dots\times O_n$ with
  \begin{enumerate}
    \item[(i)] $O_i=\{\infty\}$ if $B_i$ has zero radius,
    \item[(ii)] $O_i=\{0,\infty\}$ if $B_i$ irrational,
    \item[(iii)] $O_i=\PP_{\mathfrak K}^1$ if $B_i$ rational and has positive radius,
  \end{enumerate}
  and we identify
  \begin{equation*}
    ((c_1,\dots,c_n),(u_1,\dots,u_n)) \sim ((c_1',\dots,c_n'),(u_1',\dots,u_n'))
  \end{equation*}
  if and only if there is some $\lambda>0$ such that
  \begin{enumerate}
  \item[(a)] $u_i = u_i' = 0$ for all $i \in [n]$ with $c_i \neq c_i'$,
  \item[(b)] $u_i = \lambda\cdot u_i'$ for all $i \in [n]$ with $c_i = c_i'$.
  \end{enumerate}
\end{proposition}
\begin{proof}
  We will show the statement for $B=B_1\times\dots\times B_n\in\mathcal B^n$, the statement for $B\in\mathcal H^n$ follows similarly.

  Note that given any $\overline{((c_1,\dots,c_n),(u_1,\dots,u_n))}\in (O\times (\mathbb{R}_{\geq0}^n\setminus\{0\})) / \sim$, we may assume without loss of generality that $0\leq u_i\leq r_i$. This allows us to define a map
  \begin{equation*}
    \varphi\colon \bigslant{\Bigl(O\times (\mathbb{R}_{\geq0}^n\setminus\{0\})\Bigr)}{\sim} \longrightarrow T_B, \quad \overline{\big((c_1,\dots,c_n),(u_1,\dots,u_n)\big)} \longmapsto \vec{v}_{B,B_1'\times\dots\times B_n'},
  \end{equation*}
  where $B_i'\in C_{\mathcal B^1\setminus B_i}(c_i)$ is any disk with $\vec{v}_{\mathcal B}(B_i,B_i')=u_i$.  We show three properties:
  \begin{description}
  \item[\textbf{$\varphi$ is well-defined:}]
    First, we show that $\vec{v}_{B,B_1'\times\dots\times B_n'}$ is independent of the choice of $B_i'\in C_{\mathcal B^1\setminus B_i}(c_i)$ with $\vec{v}_{\mathcal B}(B_i,B_i')=u_i$.  Consider $B_i''\in C_{\mathcal B^1\setminus B_i}(c_i)$ with $\vec{v}_{\mathcal B}(B_i,B_i'')=u_i$. Set $\hat B_i\coloneqq B_i'\vee B_i''$ and $s_i\coloneqq \frac{\vec{v}_{\mathcal B}(B_i,\hat B_i)}{\vec{v}_{\mathcal B}(B_i,B_i')}$ so that $0<s_i\leq 1$.  Let $s\coloneqq\min(s_1,\dots,s_n)$.  Then $\gamma_{B,B_1'\times\dots B_n'}([0,s])=\gamma_{B,B_1''\times \dots \times B_n''}([0,s])$ so that indeed $\vec{v}_{B,B_1'\times\dots B_n'}=\vec{v}_{B,B_1''\times \dots \times B_n''}$.

    Finally, suppose that $((c_1,\dots,c_n),(u_1,\dots,u_n))\sim ((c_1',\dots,c_n'),(u_1',\dots,u_n'))$, i.e., $u_i = u_i' = 0$ for all $c_i \neq c_i'$ and $u_i = \lambda\cdot u_i'$ otherwise for some $\lambda>0$.  Assume without loss of generality that $\lambda>1$. Let $B_i'\in C_{\mathcal B^1\setminus B}(c_i)$ with $d_{\mathcal B}(B_i,B_i')=u_i$, and set $B_i''\coloneqq \gamma_{B_i,B_i'}(\lambda^{-1})$.  Because $\gamma_{B_i,B_i'}$ is of constant speed, we have $d_{\mathcal B}(B_i,B_i'')=\lambda^{-1}\cdot u_i=u_i'$.  By construction $\gamma_{B,B_1'\times\dots B_n'}([0,\lambda^{-1}])=\gamma_{B,B_1''\times \dots \times B_n''}([0,1])$ so that $\vec{v}_{B,B_1'\times\dots B_n'}=\vec{v}_{B,B_1''\times \dots \times B_n''}$.
  \item[\textbf{$\varphi$ is injective:}]  Let $(c,u)=((c_1,\dots,c_n),(u_1,\dots,u_n)), (c',u')=((c_1',\dots,c_n'),(u_1',\dots,u_n'))\in O\times (\RR_{\geq 0}^n\setminus\{0\})$.  Let $B_i'\in C_{\mathcal B^1\setminus B_i}(c_i)$ and $B_i''\in C_{\mathcal B^1\setminus B_i}(c_i')$ with $d_{\mathcal B^1}(B_i,B_i')=u_i$ and $d_{\mathcal B^1}(B_i,B_i'')=u_i'$.  Suppose that $(c,u)\not\sim (c',u')$, then we distinguish between two cases:
    \begin{enumerate}
    \item[(i)] There exists $i\in [n]$ with $c_i\neq c_i'$ but either $u_i\neq 0$ or $u_i'\neq 0$.  Assume without loss of generality that $u_i\neq 0$.  If $u_i'=0$, then $\gamma_{B_iB_i''}$ is constant while $\gamma_{B_iB_i'}$ is not, hence $\gamma_{B_iB_i''}\not\sim\gamma_{B_iB_i'}$ which shows that $\vec{v}_{B,B_1'\times\dots B_n'}\neq \vec{v}_{B,B_1''\times \dots \times B_n''}$.  If $u_i'\neq 0$, then $\gamma_{B_iB_i''}\not\sim\gamma_{B_iB_i'}$ by \cref{prop:tangentSpaceDisk} and thus $\vec{v}_{B,B_1'\times\dots B_n'}\neq \vec{v}_{B,B_1''\times \dots \times B_n''}$.
    \item[(ii)] There are $i,j\in [n]$, $i\neq j$, with $c_i=c_i'$ and $c_j=c_j'$ but $u_i=\lambda_i\cdot u_i'$ and $u_j=\lambda_j\cdot u_j'$ for $\lambda_i\neq \lambda_j$.  This implies that $\gamma_{B_iB_i'}\times \gamma_{B_jB_j'}$ and $\gamma_{B_iB_i''}\times \gamma_{B_jB_j''}$ cannot share an initial segment and thus $\vec{v}_{B,B_1'\times\dots B_n'}\neq \vec{v}_{B,B_1''\times \dots \times B_n''}$.
    \end{enumerate}
  \item[\textbf{$\varphi$ is surjective:}] This is immediate since for any $\vec{v}_{B,B_1'\times\dots\times B_n'}\in T_B\mathcal B^n$, we may find a preimage $$\overline{((c_1,\dots,c_n),(u_1,\dots,u_n))}\in (O\times (\mathbb{R}_{\geq0}^n\setminus\{0\})) / \sim$$ by picking $c_i\in O_i$ with $B_i'\in C_{\mathcal B^n\setminus B_i}(c_i)$ and $u_i\coloneqq d_{\mathcal B}(B_i,B_i')$. \qedhere
  \end{description}
\end{proof}
One notable consequence of this result is that the points in $\mathcal B ^n$ that have ``the most tangent directions'' are the rational polydiscs with positive radii. Intuitively, these are the points at which the local structure of the space is the most complicated. 
Later in \cref{sec:optimisation}, we study algorithms that navigate this subspace by choosing tangent directions to move along.

\subsection{Algebra in Polydisc Spaces}

To work effectively with functions on $\mathcal B^n$, it is useful to extend the basic algebraic operations from the field $K$ to polydiscs. Intuitively, a polydisc represents a set of possible values, so addition and multiplication should reflect the range of outcomes obtained by combining elements from two such sets. This leads to natural definitions of Minkowski-type operations on discs.

\begin{definition}
  \label{def:algebra}
  For $B,B'\in\mathcal B^1$ we define
  \begin{equation*}
    B+B' \coloneqq \big\{ b+b'\in \hat K\mid b\in B, b'\in B' \big\}.
  \end{equation*}
  Similarly, we define $B\cdot B'$ to be the smallest disc containing the set
  $$
  \big\{ b\cdot b'\in \hat K\mid b\in B, b'\in B' \big\}.
  $$
  For $B=B_1\times\dots\times B_n$, $B'=B_1'\times\dots\times B_n'\in\mathcal B^n$, we define
  \begin{equation*}
    B+B'\coloneqq (B_1+B_1')\times\dots\times (B_n+B_n') \quad\text{and}\quad B\cdot B'\coloneqq (B_1\cdot B_1')\times\dots\times(B_n\cdot B_n').
  \end{equation*}
\end{definition}

These operations admit explicit descriptions in terms of centres and radii, reflecting the ultrametric structure of $K$.

\begin{lemma}
  \label{lem:algebraDefinition}
  Let $B=B_{\mathcal{B}}(a,r)$ and $B'=B_{\mathcal{B}}(a',r')\in \mathcal B^1$. Then
  \begin{align*}
    B+B' &= B_{\mathcal{B}}(a+a',\, \max(r,r')),\\
    B\cdot B' &= B_{\mathcal{B}}\big(aa',\, \max(|a|r',|a'|r,rr')\big).
  \end{align*}
\end{lemma}
\begin{proof}
The inclusion $B+B'\subseteq B_{\mathcal{B}}(a+a',\max(r,r'))$ follows from the ultrametric inequality. 
Conversely, if $z\in B_{\mathcal{B}}(a+a',\max(r,r'))$, then when $r\geq r'$ we can write $z=(a+z-a-a')+a'$, and when $r'\geq r$ we can write $z=a+(a'+z-a-a')$. 
This proves the formula for $B+B'$.

For the product formula, write elements of $B$ and $B'$ as $a+x$ and $a'+y$, with $|x|\leq r$ and $|y|\leq r'$. 
Then note that
$$
  (a+x)(a'+y)-aa'=ay+a'x+xy,
$$
so every product lies in $B_{\mathcal{B}}\big(aa',R\big)$, where $R=\max(|a|r',|a'|r,rr')$. 
We claim that this radius is minimal: since $aa'$ itself is a product, it suffices to find products whose distance from $aa'$ is arbitrarily close to $R$. 
If $R=0$, this is trivial, so we can assume that $R \ne 0$.
If $R=|a|r'$ with $a\neq 0$, then take $x=0$ and choose $y$ with $|y|$ arbitrarily close to $r'$. 
The case $R=|a'|r$ is analogous. 
Finally consider the case where $R=rr'$ and we have $r>|a|$ and $r'>|a'|$. 
Then we can pick $x,y$ with $|x|$ and $|y|$ arbitrarily close to $r$ and $r'$, respectively, and $|x|>|a|$ and $|y|>|a'|$, so that
$d((a+x)(a'+y),aa') = \lvert x y \rvert$, and we are done.
In each of these cases, we used the fact that the value group $\lvert \hat K ^ \times \rvert$ is dense in $\RR_{>0}$.
\end{proof}

The explicit formulas in \cref{lem:algebraDefinition} show that these operations behave well with respect to the ultrametric structure. In particular, they satisfy the expected algebraic properties:

\begin{lemma}
  \label{lem:algebraRules}
    On $\mathcal B^1$, the operations $+$ and $\cdot$ are commutative and associative, with identities
    $B_{\mathcal{B}}(0,0)$ and $B_{\mathcal{B}}(1,0)$, respectively.  Moreover, $+$ distributes over the join operation:
    \[
      B+(B'\vee B'')=(B+B')\vee(B+B'') \qquad\text{for all }B,B',B''\in\mathcal B^1.
    \]
    The coordinate-wise operations on $\mathcal B^n$ satisfy the same properties.
\end{lemma}
\begin{proof}
Commutativity of $+$ and $\cdot$ follows directly from the formulas in
\cref{lem:algebraDefinition}. Associativity of $+$ is immediate.  For
associativity of $\cdot$, we verify from the product formula that both
$(B_{\mathcal{B}}(a,r)\cdot B_{\mathcal{B}}(b,s))\cdot B_{\mathcal{B}}(c,t)$
and
$B_{\mathcal{B}}(a,r)\cdot (B_{\mathcal{B}}(b,s)\cdot B_{\mathcal{B}}(c,t))$
have centre $abc$ and radius
\[
  \max(|ab|t,\,  |ac|s,\,  |bc|r,\,  |a|st,\,  |b|rt,\, |c|rs,\, rst).
\]
The identity statements are also immediate from the formulae.

For distributivity over joins, let
$B=B_{\mathcal{B}}(a,r)$, $B'=B_{\mathcal{B}}(b,s)$, and
$B''=B_{\mathcal{B}}(c,t)$.  Then
\[
  B+(B'\vee B'')=B_{\mathcal{B}}(a+b,\, \max(r,s,t,|b-c|)),
\]
and the same expression is obtained from $(B+B')\vee(B+B'')$.
\end{proof}

\begin{example}
  The operation $\cdot$ does not, in general, distribute over $+$. For example,
if $A=B_{\mathcal{B}}(0,1)$, $B=B_{\mathcal{B}}(1,0)$ and
$C=B_{\mathcal{B}}(-1,0)$, then $A\cdot(B+C)=B_{\mathcal{B}}(0,0)$
whereas $(A\cdot B)+(A\cdot C)=B_{\mathcal{B}}(0,1)$.
\end{example}

This failure of distributivity demonstrates that $(\mathcal B^1,+,\cdot)$ does not form a semiring in the usual sense. However, a closely related structure does arise by modifying the operations slightly.

Appending an absorbing empty disc $\emptyset$
gives rise to an extended disc space $\overline{\mathcal B^1}=\mathcal B^1\cup\{\emptyset\}$. Equipping this space with addition given by the join $\vee$ and multiplication given by $+$, this becomes an idempotent semiring: $\emptyset$ is the additive identity, $B_{\mathcal{B}}(0,0)$ is the multiplicative identity, and the distributive law follows from \cref{lem:algebraRules}. The same construction extends coordinate-wise to $\mathcal B^n$.


\subsection{Extending Polydisc Spaces}

We consider the behaviour of polydisc spaces under extensions of the base field. Many natural constructions on polydisc spaces require passing to larger valued fields, for example, when the value group or the residue field of $K$ is not sufficiently large to realise the embeddings. These extensions will play an important role later when studying embeddings of polydisc spaces in \cref{sec:embedding}.

From a geometric perspective, such extensions have a simple interpretation: extending the value group refines the metric structure (introducing new intermediate rational radii), while extending the residue field increases the branching at rational discs. The following lemma makes this notion precise.

\begin{lemma}
  \label{lem:extension}
  Let $K\subseteq L$ be a valued field extension and let $\mathfrak K\subseteq \mathfrak L$ be the corresponding residue field extension.  Let $B_K(\cdot)$ and $B_L(\cdot)$ denote disks over $K$ and $L$, respectively, and let $\mathcal B^1(K)$, $\mathcal B^1(L)$ denote their respective disc spaces.  We will consider $\mathcal B^1(K)$ as a subset of $\mathcal B^1(L)$ under the inclusion
  \begin{equation*}
    B^1(K)\lhook\joinrel\longrightarrow B^1(L), \qquad B_K(a,r) \longmapsto B_L(a,r).
  \end{equation*}
  Then for $B=B_L(a,r)\in\mathcal B^1(L)$ with positive radius, we have
  \begin{enumerate}
  \item $B\in \mathcal B^1(K)$ if and only if $a=a'+a''$ for some $a'\in K$ and $|a''|\leq r$, and
  \item if $B\in \mathcal B^1(K)$, then
  \begin{equation*}
    T_B\mathcal B^n(L) = T_B\mathcal B^n(K) \cup \{ \vec{v}_{B,\bar c}\mid \bar c\in O_{B,K\hookrightarrow L} \},
  \end{equation*}
  where
  \begin{enumerate}
  \item $O_{B,K\hookrightarrow L}=\emptyset$ if $B\in \mathcal B^1(L)$ irrational,
  \item $O_{B,K\hookrightarrow L}=\mathfrak L\setminus\{0\}$ if $B\in \mathcal B^1(L)$ rational and $B\in \mathcal B^1(K)$ irrational, and
  \item $O_{B,K\hookrightarrow L}=\mathfrak L\setminus\mathfrak K$ if $B\in \mathcal B^1(K)$ rational.
  \end{enumerate}
\end{enumerate}
\end{lemma}
\begin{proof}
  The statement follows from the description of tangent directions in \cref{prop:tangentSpaceDisk}.
\end{proof}

The lemma shows that field extensions affect the local geometry of $\mathcal B^1$ in two distinct ways: new radii appear when the value group is enlarged and new tangent directions appear when the residue field is extended.

\begin{example}
  \label{ex:extension}\
  Let $\FF_2$ be the field with two elements and let $K=\FF_2(\!(t)\!)$ be the field of Laurent series thereover. Consider the following two extensions:
  \begin{enumerate}
  \item[(i)] $\FF_2(\!(t)\!)\subseteq L_1\coloneqq \FF_2(\!(t^{1/2})\!)=\{\sum_{k\geq k_0}a_kt^{k/2}\mid k_0\in\ZZ, a_k\in\FF_2\}$.
  \item[(ii)] $\FF_2(\!(t)\!)\subseteq L_2\coloneqq \FF_4(\!(t)\!)=\{\sum_{k\geq k_0}a_kt^{k}\mid k_0\in\ZZ, a_k\in\FF_4\}$, where $\FF_4$ is the field with $4$ elements, i.e., $\FF_4=\{0,1,\alpha,\alpha+1\}$ with $\alpha+\alpha=0$ and $\alpha\cdot\alpha=\alpha+1$.
  \end{enumerate}

These two extensions illustrate the two phenomena described above. \cref{fig:extension} illustrates $\mathcal B^1(K)\subseteq \mathcal B^1(L_1)$ and $\mathcal B^1(K)\subseteq \mathcal B^1(L_2)$.  Note that $L_1$ extends the value group and hence $\mathcal B^1(L_1)$ has edges of length $1/2$ while $L_2$ extends the residue field and hence $\mathcal B^1(L_2)$ has vertices of degree $5$.
\end{example}

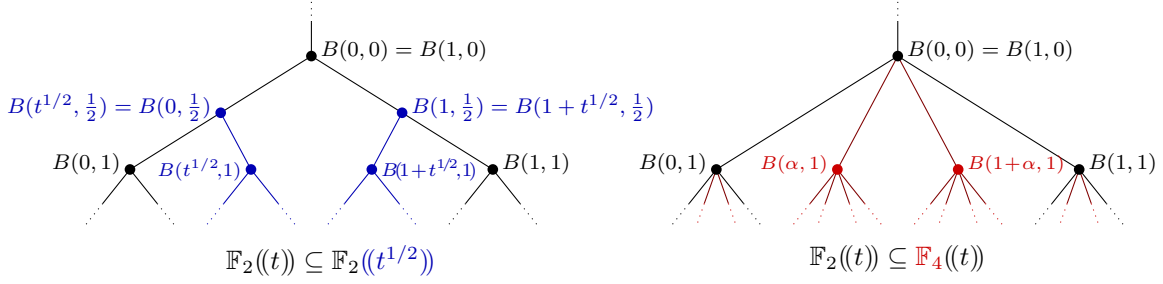
\begin{figure}[t]
  \centering
  \begin{tikzpicture}
    \node (left) at (0,0)
    {
      \begin{tikzpicture}[x={(1.2,0)},y={(0,1.5)}]
        \coordinate (v0) at (0,0);
        \coordinate (v1) at (-2,-1);
        \coordinate (v2) at (2,-1);
        \draw
        (v0) -- ++(0,0.3)
        (v0) -- node[pos=0.5,inner sep=0] (w1) {} (v1)
        (v0) -- node[pos=0.5,inner sep=0] (w2) {} (v2)
        (v1) -- ++(-0.3,-0.3)
        (v1) -- ++(0.3,-0.3)
        (v2) -- ++(-0.3,-0.3)
        (v2) -- ++(0.3,-0.3);
        \draw[dotted]
        (v0) ++(0,0.3) -- ++(0,0.2)
        (v1) ++(-0.3,-0.3) -- ++(-0.2,-0.2)
        (v1) ++(0.3,-0.3) -- ++(0.2,-0.2)
        (v2) ++(-0.3,-0.3) -- ++(-0.2,-0.2)
        (v2) ++(0.3,-0.3) -- ++(0.2,-0.2);
        \fill
        (v0) circle (2pt)
        (v1) circle (2pt)
        (v2) circle (2pt);
        \coordinate (w3) at (-0.666,-1);
        \coordinate (w4) at (0.666,-1);
        \fill[blue!70!black]
        (w1) circle (2pt)
        (w2) circle (2pt)
        (w3) circle (2pt)
        (w4) circle (2pt);
        \draw[blue!70!black]
        (w1) -- (w3)
        (w2) -- (w4)
        (w3) -- ++(-0.3,-0.3)
        (w3) -- ++(0.3,-0.3)
        (w4) -- ++(-0.3,-0.3)
        (w4) -- ++(0.3,-0.3);
        \draw[dotted,blue!70!black]
        (w3) ++(-0.3,-0.3) -- ++(-0.2,-0.2)
        (w3) ++(0.3,-0.3) -- ++(0.2,-0.2)
        (w4) ++(-0.3,-0.3) -- ++(-0.2,-0.2)
        (w4) ++(0.3,-0.3) -- ++(0.2,-0.2);
        \node[anchor=base west,font=\footnotesize] at (v0) {$B(0,0)=B(1,0)$};
        \node[anchor=base east,font=\footnotesize] at (v1) {$B(0,1)$};
        \node[anchor=base west,font=\footnotesize] at (v2) {$B(1,1)$};
        \node[anchor=base east,font=\footnotesize,blue!70!black] at (w1) {$B(t^{1/2},\frac 12)=B(0,\frac 12)$};
        \node[anchor=base west,font=\footnotesize,blue!70!black] at (w2) {$B(1,\frac 12)=B(1+t^{1/2},\frac 12)$};
        \node[anchor=east,font=\scriptsize,blue!70!black] at (w3) {$B(t^{1/2}\!,\!1)$};
        \node[anchor=west,font=\scriptsize,blue!70!black] at (w4) {$B\!(\!1\!+\!t^{1\!/\!2}\!,\!1\!)$};
      \end{tikzpicture}
    };
    \node (right) at (7.5,0)
    {
      \begin{tikzpicture}[x={(1.2,0)},y={(0,1.5)}]
        \coordinate (v0) at (0,0);
        \coordinate (v1) at (-2,-1);
        \coordinate (v2) at (2,-1);
        \draw
        (v0) -- ++(0,0.3)
        (v0) -- (v1)
        (v0) -- (v2)
        (v1) -- ++(-0.3,-0.3)
        (v1) -- ++(0.3,-0.3)
        (v2) -- ++(-0.3,-0.3)
        (v2) -- ++(0.3,-0.3);
        \draw[dotted]
        (v0) ++(0,0.3) -- ++(0,0.2)
        (v1) ++(-0.3,-0.3) -- ++(-0.2,-0.2)
        (v1) ++(0.3,-0.3) -- ++(0.2,-0.2)
        (v2) ++(-0.3,-0.3) -- ++(-0.2,-0.2)
        (v2) ++(0.3,-0.3) -- ++(0.2,-0.2);
        \coordinate (w1) at (-0.666,-1);
        \coordinate (w2) at (0.666,-1);
        \draw[red!50!black]
        (v0) -- (w1)
        (v0) -- (w2)
        (v1) -- ++(-0.133,-0.3)
        (v1) -- ++(0.133,-0.3)
        (v2) -- ++(-0.133,-0.3)
        (v2) -- ++(0.133,-0.3)
        (w1) -- ++(-0.3,-0.3)
        (w1) -- ++(-0.133,-0.3)
        (w1) -- ++(0.133,-0.3)
        (w1) -- ++(0.3,-0.3)
        (w2) -- ++(-0.3,-0.3)
        (w2) -- ++(-0.133,-0.3)
        (w2) -- ++(0.133,-0.3)
        (w2) -- ++(0.3,-0.3);
        \draw[red!80!black,dotted]
        (v1) ++(-0.133,-0.3) -- ++(-0.087,-0.2)
        (v1) ++(0.133,-0.3) -- ++(0.087,-0.2)
        (v2) ++(-0.133,-0.3) -- ++(-0.087,-0.2)
        (v2) ++(0.133,-0.3) -- ++(0.087,-0.2)
        (w1) ++(-0.3,-0.3) -- ++(-0.2,-0.2)
        (w1) ++(-0.133,-0.3) -- ++(-0.087,-0.2)
        (w1) ++(0.133,-0.3) -- ++(0.087,-0.2)
        (w1) ++(0.3,-0.3) -- ++(0.2,-0.2)
        (w2) ++(-0.3,-0.3) -- ++(-0.2,-0.2)
        (w2) ++(-0.133,-0.3) -- ++(-0.087,-0.2)
        (w2) ++(0.133,-0.3) -- ++(0.087,-0.2)
        (w2) ++(0.3,-0.3) -- ++(0.2,-0.2);
        \fill
        (v0) circle (2pt)
        (v1) circle (2pt)
        (v2) circle (2pt);
        \fill[red!80!black]
        (w1) circle (2pt)
        (w2) circle (2pt);
        \node[anchor=base west,font=\footnotesize] at (v0) {$B(0,0)=B(1,0)$};
        \node[anchor=base east,font=\footnotesize] at (v1) {$B(0,1)$};
        \node[anchor=base west,font=\footnotesize] at (v2) {$B(1,1)$};
        \node[anchor=base east,font=\scriptsize,red!80!black] at (w1) {$B(\alpha,1)$};
        \node[anchor=base west,font=\scriptsize,red!80!black] at (w2) {$B(1\!+\!\alpha,1)$};
      \end{tikzpicture}
    };
    \node[below] at (left.south) {$\FF_2(\!(t)\!)\subseteq \FF_2\textcolor{blue!70!black}{(\!(t^{1/2})\!)}$};
    \node[below] at (right.south) {$\FF_2(\!(t)\!)\subseteq \textcolor{red!80!black}{\FF_4}(\!(t)\!)$};
  \end{tikzpicture}\vspace{-2mm}
  \caption{Disc spaces under field extensions (all radii in valuation).}
  \label{fig:extension}
\end{figure}

In summary, extending the value group refines distances, while extending the residue field increases branching. We point out that in practice, standard computer algebra packages support working over arbitrary algebraic extensions (see for instance \cite{OSCAR}).

\subsection{Polydisc Spaces and Berkovich Geometry}\label{sec:Berkovich}

The polydisc spaces we introduce in this paper are related to Berkovich analytic geometry, but are designed for a rather different purpose. While Berkovich spaces provide a highly general framework for non-Archimedean analytic geometry, our focus is on constructing a concrete geometric setting suitable for optimisation, metric geometry, and computation. In particular, the spaces $\mathcal B^n$ and $\mathcal H^n$ are equipped with explicit order structures, metrics, convexity notions, tangent directions, and hierarchical decompositions that support algorithmic procedures developed later in the paper.

Nevertheless, there is a natural connection between our construction and Berkovich affine space. To a polydisc \(B=B_{\mathcal{B}}(a,r)\in\mathcal B^n\), we associate the multiplicative seminorm
\[
  |f|_B \coloneqq \sup_{z\in B}|f(z)|,\qquad f\in K[x_1,\dots,x_n].
\]
Equivalently, if $f=\sum_\alpha c_\alpha (x-a)^\alpha$, then
\[
  |f|_B=\max_\alpha |c_\alpha| r^\alpha,
\]
with the usual multi-index convention $r^\alpha=\prod_i r_i^{\alpha_i}$. This is precisely the Gauss seminorm associated with the closed polydisc \(B\), and therefore determines a canonical embedding
\(
  \mathcal B^n \hookrightarrow \BA_K^{n,\mathrm{an}}.
\)

Accordingly, the real-valued functions considered throughout this paper, such as $B \mapsto |f|_B$, may be viewed as restrictions of the standard evaluation functions on Berkovich affine space.

In dimension one, this embedding recovers the familiar disc-type points of the Berkovich affine line. When $K$ is algebraically closed, the points of radius zero correspond to type I points, rational discs of positive radius correspond to type II points, and irrational discs correspond to type III points \cite{baker2008introduction,BakerRumely2010}. Type IV points, represented by nested families of discs with empty intersection, are intentionally excluded from our framework: restricting to genuine discs preserves the explicit tree-like and hierarchical structures needed for the optimisation procedures developed later in the paper. For general $K$, the full Berkovich affine line contains additional points not represented by discs centred in $K$ \cite[Corollary I.5.6]{poineau2021berkovich}; in this sense, $\mathcal B^1$ should be viewed as the space of closed discs over $\widehat K$ whose centres lie in $K$.

For $n>1$, the distinction from the full Berkovich analytification becomes much more apparent. The space $\mathbb A_K^{n,\mathrm{an}}$ contains a large collection of multiplicative seminorms with no comparably explicit classification \cite{jonsson2014}, whereas $\mathcal B^n$ consists only of product Gauss points associated with closed polydiscs. This restriction is deliberate: it yields a tractable geometric space with explicit combinatorial and metric structure, making it possible to define and analyse the optimisation procedures developed later in the paper. In particular this allows us to define very explicit additional structure such as the order, joins, metrics, tangent spaces, and hierarchical decompositions studied here.


\section{A Structural Example: Embedding Metric Trees in $\mathcal{H}^1$}

\label{sec:embedding}
We illustrate the hierarchical nature of polydisc spaces by studying isometric embeddings of metric trees into the hyperbolic polydisc space $\mathcal{H}^1$. Metric trees form a fundamental class of geodesic metric spaces characterised by their branching structure and provide a natural setting in which the properties developed in the previous section, such as the geodesic structure, partial order, and combinatorial nature of the space, can be showcased explicitly.

We show that the embedding problem into $\mathcal{H}^1$ admits a simple and explicit characterisation in terms of the local combinatorics of the tree and the arithmetic of the base field. In particular, the existence of an isometric embedding can be expressed in terms of vertex degrees and edge lengths, which reflects the discrete and hierarchical structure of both the tree and the ambient space.

\begin{proposition}
  \label{prop:metricTreeEmbedding1}
  Let $\Gamma$ be a metric tree and let $\Gamma'$ be the metric tree obtained by suppressing all degree~$2$ vertices of $\Gamma$.  Then there is an isometric embedding $\Gamma\hookrightarrow \mathcal H^1$ if and only if there is an isometric embedding $\Gamma'\hookrightarrow \mathcal H^1$.
\end{proposition}
\begin{proof}
  The statement follows directly from the fact that suppressing degree 2 vertices does not change the intrinsic metric structure of the tree.
\end{proof}

A key feature of embeddings into $\mathcal{H}^1$ is their local nature. As we will now see, the existence of an embedding depends only on local information: Given a subtree $\Gamma \subseteq \Gamma'$, the extendability of an embedding to $\Gamma'$ is determined entirely by the newly added edges and their lengths. In particular, whenever $\Gamma'$ is embeddable, any embedding of $\Gamma$ extends to an embedding of $\Gamma'$. 

\begin{proposition}
  \label{prop:metricTreeEmbedding}
  Let $\Gamma=(V,E,d_\Gamma)$ be a metric tree with no degree $2$ vertices, vertex set $V$, edge set $E\subseteq \binom{V}{2}$, and metric $d_\Gamma$. Let $K$ be the underlying field of the disk space.  Then there is an isometric embedding $\Gamma\hookrightarrow \mathcal H^1$ if and only if
  \begin{enumerate}
  \item \label{enumitem:metricTreeEmbedding1} $\max(\{\deg_\Gamma(v)\mid v\in V\})\leq |\mathfrak K|+1$, and
  \item \label{enumitem:metricTreeEmbedding2} $\{ d_\Gamma(v,v')\mid (v,v')\in E, (v,v') \text{ not a leaf}\} \subseteq \val(K^\ast)$.
  \end{enumerate}
\end{proposition}
\begin{proof}
  The proof of the ``$\Rightarrow$'' direction is straightforward.
  For the ``$\Leftarrow$'' direction, we construct an embedding explicitly:  Pick any $v_0\in V$ and any rational $B_0=B(a_0,v_0)\in \mathcal H^1$ with $a_0\in K$ and $v_0\in\val(K)$.  By Assumption \ref{enumitem:metricTreeEmbedding1}, we can pick a distinct $c_i\in \PP^1_{\mathfrak K}$ for any $v_i\in V$ with $(v_0,v_i)\in E$. Let $B_i=(a_i,v_i)$ be given by either
  \begin{itemize}
  \item $a_i\coloneqq a_0$ and $v_i\coloneqq v_0-d(v_0,v_i)$ if $c_i=\infty$, or
  \item $a_i\in K$ with $\overline{t^{-v_0}(a_0-a_i)}=c_i$ and $v_i\coloneqq v_0+d(v_0,v_i)$ if $c_i\in\mathfrak K$.
  \end{itemize}
  This defines an isometric embedding of $v_0$ and its adjacent vertices $v_i$.  Repeating this process yields an isometric embedding of $\Gamma$ into $\mathcal H^1$.
\end{proof}

\begin{example}
  \label{ex:metricTreeEmbedding}
  Note that \cref{prop:metricTreeEmbedding} not only shows when embedding a metric tree into $\mathcal H^1$ is possible. Combined with \cref{lem:extension}, it also explains how the valued field $K$ must be extended to make it possible if it is not; see \cref{fig:metricTreeEmbedding}.

  Let $\Gamma_1$ be the metric tree with vertices $1,2,3,4$, edges $(1,2),(2,3),(2,4)$, and all edge lengths $1$.  Then $\Gamma_1$ is embeddable into $\mathcal H^1(\FF_2(\!(t)\!))$ via
  \begin{equation*}
    \iota_1\colon \quad \Gamma_1\rightarrow \mathcal H^1(\FF_2(\!(t)\!)), \quad 1\mapsto B(1,0), \quad 2\mapsto B(1,1), \quad 3\mapsto B(1,2), \quad 4\mapsto B(1+t,2).
  \end{equation*}

  Let $\Gamma_2$ be $\Gamma_1$ with two extra vertices $5,6$ and extra edges $(3,5)$, $(3,6)$ which have length $1/2$.  While $\Gamma_2$ cannot be rationally embedded into $\mathcal H^1(\FF_2(\!(t)\!))$ since $d_{\Gamma_2}(5,6)=\frac 12\notin \val(\FF_2(\!(t)\!))$, it can be embedded into $\mathcal H^1(\FF_2(\!(t^{1/2})\!))$ by extending $\iota_1$ by
  \begin{equation*}
    \iota_2\colon \quad \Gamma_2\rightarrow \mathcal H^1(\FF_2(\!(t^{1/2})\!)), \quad 5\mapsto B(1,2.5), \quad 6\mapsto B(1+t^2,2.5).
  \end{equation*}

  Let $\Gamma_3$ be $\Gamma_2$ with two extra vertices $7,8$ and extra edges $(2,7)$, $(2,8)$ which have length $1$.  Then $\Gamma_3$ cannot be embedded into $\mathcal H^1(\FF_2(\!(t^{1/2})\!))$ since $\deg_{\Gamma}(2)=5>3=|\FF_2|+1$, however it can be embedded into $\mathcal H^1(\FF_4(\!(t^{1/2})\!))$ by extending $\iota_2$ by
  \begin{equation*}
    \iota_3\colon \quad \Gamma_3\rightarrow \mathcal H^1(\FF_4(\!(t^{1/2})\!)), \quad 7\mapsto B(1+\alpha\cdot t,2), \quad 8\mapsto B(1+(1+\alpha)\cdot t,2),
  \end{equation*}
  where, as in \cref{ex:extension}, $\FF_4$ is the field with $4$ elements, i.e., $\FF_4=\{0,1,\alpha,\alpha+1\}$ with $\alpha+\alpha=0$ and $\alpha\cdot\alpha=\alpha+1$.
\end{example}

\begin{figure}[t]
  \centering
  \begin{tikzpicture}
    \node (m11) at (0,0)
    {
      \begin{tikzpicture}[x={(1,0)},y={(0,1.2)}]
        \useasboundingbox (-1.75,0.25) rectangle (1.75,-2.25);
        \coordinate (p1) at (0,0);
        \coordinate (p2) at (0,-1);
        \coordinate (p3) at (-1,-2);
        \coordinate (p4) at (1,-2);
        \coordinate (p5) at (-1.4,-2.75);
        \coordinate (p6) at (-0.6,-2.75);
        \coordinate (p7) at (0,-2);
        \coordinate (p8) at (2,-2);
        \node[font=\small] (v1) at (p1) {$1$};
        \node[font=\small] (v2) at (p2) {$2$};
        \node[font=\small] (v3) at (p3) {$3$};
        \node[font=\small] (v4) at (p4) {$4$};
        \draw[thick]
        (v1) -- node[left, font=\scriptsize] {$1$} (v2)
        (v2) -- node[left, font=\scriptsize] {$1$} (v3)
        (v2) -- node[right, font=\scriptsize] {$1$} (v4);
      \end{tikzpicture}
    };
    \node (m21) at (0,-4)
    {
      \begin{tikzpicture}[x={(1,0)},y={(0,1.2)}]
        \useasboundingbox (-1.75,0.25) rectangle (1.75,-3);
        \coordinate (p1) at (0,0);
        \coordinate (p2) at (0,-1);
        \coordinate (p3) at (-1,-2);
        \coordinate (p4) at (1,-2);
        \coordinate (p5) at (-1.4,-2.75);
        \coordinate (p6) at (-0.6,-2.75);
        \coordinate (p7) at (0,-2);
        \coordinate (p8) at (2,-2);
        \node[font=\small] (v1) at (p1) {$1$};
        \node[font=\small] (v2) at (p2) {$2$};
        \node[font=\small] (v3) at (p3) {$3$};
        \node[font=\small] (v4) at (p4) {$4$};
        \node[font=\small,red!80!black] (v5) at (p5) {$5$};
        \node[font=\small,red!80!black] (v6) at (p6) {$6$};
        \draw[thick]
        (v1) -- (v2)
        (v2) -- (v3)
        (v2) -- (v4);
        \draw[thick,red!80!black]
        (v3) -- node[left, font=\scriptsize] {$0.5$} (v5)
        (v3) -- node[right, font=\scriptsize] {$0.5$} (v6);
      \end{tikzpicture}
    };
    \node (m31) at (0,-8.5)
    {
      \begin{tikzpicture}[x={(1,0)},y={(0,1.2)}]
        \useasboundingbox (-1.75,0.25) rectangle (1.75,-3);
        \coordinate (p1) at (0,0);
        \coordinate (p2) at (0,-1);
        \coordinate (p3) at (-1,-2);
        \coordinate (p4) at (1,-2);
        \coordinate (p5) at (-1.4,-2.75);
        \coordinate (p6) at (-0.6,-2.75);
        \coordinate (p7) at (0,-2);
        \coordinate (p8) at (2,-2);
        \node[font=\small] (v1) at (p1) {$1$};
        \node[font=\small] (v2) at (p2) {$2$};
        \node[font=\small] (v3) at (p3) {$3$};
        \node[font=\small] (v4) at (p4) {$4$};
        \node[font=\small] (v5) at (p5) {$5$};
        \node[font=\small] (v6) at (p6) {$6$};
        \node[font=\small,red!80!black] (v7) at (p7) {$7$};
        \node[font=\small,red!80!black] (v8) at (p8) {$8$};
        \draw[thick]
        (v1) -- (v2)
        (v2) -- (v3)
        (v2) -- (v4)
        (v3) -- (v5)
        (v3) -- (v6);
        \draw[thick,red!80!black]
        (v2) -- node[left, font=\scriptsize] {$1$} (v7)
        (v2) -- node[right, xshift=2mm, font=\scriptsize] {$1$} (v8);
      \end{tikzpicture}
    };
    \node[anchor=north west,xshift=3mm,yshift=-13mm] at (m11.north west) {$\Gamma_1:$};
    \node[anchor=north west,xshift=3mm,yshift=-13mm] at (m21.north west) {$\Gamma_2:$};
    \node[anchor=north west,xshift=3mm,yshift=-13mm] at (m31.north west) {$\Gamma_3:$};
    \node (m12) at (7,0)
    {
      \begin{tikzpicture}[x={(1.5,0)},y={(0,1.2)}]
        \useasboundingbox (-1.75,0.25) rectangle (2.75,-2.25);
        \coordinate (p1) at (0,0);
        \coordinate (p2) at (0,-1);
        \coordinate (p3) at (-1,-2);
        \coordinate (p4) at (1,-2);
        \coordinate (p5) at (-1.4,-2.75);
        \coordinate (p6) at (-0.6,-2.75);
        \coordinate (p7) at (0,-2);
        \coordinate (p8) at (2.2,-2);
        \node[font=\scriptsize] (v1) at (p1) {$B(1,0)$};
        \node[font=\scriptsize] (v2) at (p2) {$B(1,1)$};
        \node[font=\scriptsize] (v3) at (p3) {$B(1,2)$};
        \node[font=\scriptsize] (v4) at (p4) {$B(1\!+\!t,2)$};
        \draw[thick]
        (v1) -- (v2)
        (v2) -- (v3)
        (v2) -- (v4);
      \end{tikzpicture}
    };
    \node (m22) at (7,-4)
    {
      \begin{tikzpicture}[x={(1.5,0)},y={(0,1.2)}]
        \useasboundingbox (-1.75,0.25) rectangle (2.75,-3);
        \coordinate (p1) at (0,0);
        \coordinate (p2) at (0,-1);
        \coordinate (p3) at (-1,-2);
        \coordinate (p4) at (1,-2);
        \coordinate (p5) at (-1.5,-2.75);
        \coordinate (p6) at (-0.5,-2.75);
        \coordinate (p7) at (0,-2);
        \coordinate (p8) at (2.2,-2);
        \node[font=\scriptsize] (v1) at (p1) {$B(1,0)$};
        \node[font=\scriptsize] (v2) at (p2) {$B(1,1)$};
        \node[font=\scriptsize] (v3) at (p3) {$B(1,2)$};
        \node[font=\scriptsize] (v4) at (p4) {$B(1\!+\!t,2)$};
        \node[font=\scriptsize,red!80!black] (v5) at (p5) {$B(1,2.5)$};
        \node[font=\scriptsize,red!80!black] (v6) at (p6) {$B(1\!+\!t^2,2.5)$};
        \draw[thick]
        (v1) -- (v2)
        (v2) -- (v3)
        (v2) -- (v4);
        \draw[thick,red!80!black]
        (v3) -- (v5)
        (v3) -- (v6);
      \end{tikzpicture}
    };
    \node (m32) at (7,-8.5)
    {
      \begin{tikzpicture}[x={(1.5,0)},y={(0,1.2)}]
        \useasboundingbox (-1.75,0.25) rectangle (2.75,-3);
        \coordinate (p1) at (0,0);
        \coordinate (p2) at (0,-1);
        \coordinate (p3) at (-1,-2);
        \coordinate (p4) at (1,-2);
        \coordinate (p5) at (-1.5,-2.75);
        \coordinate (p6) at (-0.5,-2.75);
        \coordinate (p7) at (0,-2);
        \coordinate (p8) at (2.2,-2);
        \node[font=\scriptsize] (v1) at (p1) {$B(1,0)$};
        \node[font=\scriptsize] (v2) at (p2) {$B(1,1)$};
        \node[font=\scriptsize] (v3) at (p3) {$B(1,2)$};
        \node[font=\scriptsize] (v4) at (p4) {$B(1\!+\!t,2)$};
        \node[font=\scriptsize] (v5) at (p5) {$B(1,2.5)$};
        \node[font=\scriptsize] (v6) at (p6) {$B(1\!+\!t^2,2.5)$};
        \node[font=\scriptsize,red!80!black] (v7) at (p7) {$B(1\!+\!\alpha t,2)$};
        \node[font=\scriptsize,red!80!black] (v8) at (p8) {$B(1\!+\!(1\!+\!\alpha)t,2)$};
        \draw[thick]
        (v1) -- (v2)
        (v2) -- (v3)
        (v2) -- (v4)
        (v3) -- (v5)
        (v3) -- (v6);
        \draw[thick,red!80!black]
        (v2) -- (v7)
        (v2) -- (v8);
      \end{tikzpicture}
    };
    \node[anchor=north west,xshift=-30mm,yshift=-13mm] at (m12.north east) {$\subseteq \mathcal H^1(\FF_2(\!(t)\!))$};
    \node[anchor=north west,xshift=-30mm,yshift=-13mm] at (m22.north east) {$\subseteq \mathcal H^1(\FF_2\textcolor{red!80!black}{(\!(t^{1/2})\!)})$};
    \node[anchor=north west,xshift=-30mm,yshift=-13mm] at (m32.north east) {$\subseteq \mathcal H^1(\textcolor{red!80!black}{\FF_4}(\!(t^{1/2})\!))$};
    \draw[right hook-latex] (m11) -- node[above] {$\iota_1$} (m12);
    \draw[right hook-latex] ($(m21.east)+(0,0.5)$) -- node[above] {$\iota_2$} ($(m22.west)+(0,0.5)$);
    \draw[right hook-latex] ($(m31.east)+(0,0.5)$) -- node[above] {$\iota_3$} ($(m32.west)+(0,0.5)$);
    \draw[left hook-latex] (m11) -- (m21);
    \draw[left hook-latex] (m21) -- (m31);
    \draw[left hook-latex] ($(m12.south)+(-0.75,0)$) -- ($(m22.north)+(-0.75,0)$);
    \draw[left hook-latex] ($(m22.south)+(-0.75,0)$) -- ($(m32.north)+(-0.75,0)$);
  \end{tikzpicture}\vspace{-3mm}
  \caption{Embeddings in \cref{ex:metricTreeEmbedding}.}
  \label{fig:metricTreeEmbedding}
\end{figure}

The local dependence discussed above in Proposition \ref{prop:metricTreeEmbedding} contrasts with isometric embeddings into the full polydisc space $\mathcal{B}^1$, where additional global constraints arise and create additional difficulties:

\begin{example}
  \label{ex:metricTreeEmbeddingTechnicalities}
  Let $\Gamma_1$ and $\Gamma_2$ both be binary trees given by
  \begin{equation*}
    V(\Gamma_i) = \{1,\dots,8\}, \quad E(\Gamma_i) = \{(1,2),(2,3),(2,4),(3,5),(3,6),(4,7),(4,8)\}
  \end{equation*}
  and edge length $d_{\Gamma_1}(2,3)=d_{\Gamma_1}(2,4)=e^{-1}$. Let $K=\FF_2(\!(t)\!)$ as in \cref{ex:metricTreeEmbedding}.
  Then $\Gamma_1$ and $\Gamma_2$ are isometrically embeddable into $\mathcal B^1$.  However, because vertices $2,3,4$ have higher degree, they must be mapped to rational discs, and because of the edge lengths these polydiscs must have radii $0,1,1$ respectively.
  Consequently, the only gluing of $\Gamma_1$ and $\Gamma_2$ that can be embedded into $\mathcal B^1$ is by identifying $1\in V(\Gamma_1)$ and $1\in V(\Gamma_2)$.
  This shows that a metric tree may be locally embeddable into $\mathcal B^1$ but not globally.
\end{example}


\section{Function Theory on Polydisc Spaces}

\label{sec:functions}
We now turn to the study of real-valued functions on polydisc spaces $\mathcal{B}^n$. We are interested in identifying classes of functions that are both compatible with the underlying non-Archimedean structure and admit explicit analyses.

We focus on functions arising from polynomials with coefficients in $\mathcal{B}$ through their associated absolute values and valuations. Such functions form a tractable class adapted to the geometry of polydisc spaces: they admit explicit evaluation formulas; are piecewise polynomial or piecewise linear along geodesics in their behaviour; and encode directional information through their slopes. In particular, they provide a natural link between the algebraic structure of their coefficients and the metric geometry of the ambient space. We furthermore show that linear combinations of absolute values of polynomials satisfy a universal approximation property on compact subsets. Together, these results provide a natural function-theoretic counterpart to the metric geometry of polydisc spaces.\\

We begin by introducing the basic function class considered throughout this section. These functions arise from polynomials whose coefficients are themselves discs and are evaluated on polydiscs by either taking the extremal absolute value or the extremal valuation over all choices of coefficients and points in the underlying discs. This construction arises naturally in the non-Archimedean perspective and gives rise to functions that retain both algebraic and geometric information.

\begin{definition}
  \label{def:discPolynomials}
  Using the multi-index notation, a \emph{polynomial over $\mathcal B$} in the variables $x=(x_1,\dots,x_n)$ is a formal sum of the form
  \begin{equation*}
    f\coloneqq \sum_{\alpha\in\ZZ_{\geq 0}^n} C_\alpha\cdot x^\alpha \quad\text{where } C_\alpha=B_\mathcal{B}(0,0) \text{ for all but finitely many } \alpha\in\ZZ_{\geq 0}^n,
  \end{equation*}
  and we denote the set of polynomials over $\mathcal B$ by $\mathcal B[x]$. 
  Given a polynomial $f\in \mathcal B[x]$ and a polydisc $B\in\mathcal B^n$, we define the \emph{absolute value} and the \emph{valuation} of $f$ at $B$ by
  \begin{equation}
    \label{eq:absolutePolynomial}
      |f(B)|\coloneqq \sup_{\substack{c_\alpha\in C_\alpha\\ z\in B}}\Big|\sum_{\alpha\in\ZZ_{\geq 0}^n} c_\alpha z^\alpha\Big|, \qquad\text{and}\qquad
      \val(f(B))\coloneqq \inf_{\substack{c_\alpha\in C_\alpha\\ z\in B}} \val\Big(\sum_{\alpha\in\ZZ_{\geq 0}^n} c_\alpha z^\alpha\Big).
  \end{equation}
  We will refer to the function of the form $|f|\colon\mathcal B^n\rightarrow\RR_{\geq 0}, B\mapsto |f(B)|$ and $\val(f)\colon\mathcal B^n\rightarrow\RR \cup \{\infty\}, B\mapsto \val(f(B))$ as \emph{absolute polynomials} and \emph{valuation polynomials}, respectively.
\end{definition}

The two quantities in~\eqref{eq:absolutePolynomial} are related by $|f(B)|= \exp(-\val(f(B)))$.  Before presenting examples, we note that while both quantities are defined via suprema and infima, their values can be computed explicitly.

\begin{lemma}
  \label{lem:formula}
  Let $f=\sum_{|\alpha|\leq d} C_\alpha\cdot x^\alpha\in \mathcal B[x]$ be a polynomial over $\mathcal B$ and let $B\in\mathcal B^n$ be a polydisc, say $C_\alpha=B_{\mathcal{B}}(c_\alpha,s_\alpha)=B_{\mathcal{H}}(c_\alpha,w_\alpha)$ for some $c_\alpha\in K$, $s_{\alpha}\in\RR$, $w_\alpha\in\TT$, and $B=B_{\mathcal{B}}(a,r)=B_{\mathcal{H}}(a,v)$ for some $a=(a_1,\dots,a_n)\in K^n$, $r=(r_1,\dots,r_n)\in\RR^n$, $v=(v_1,\dots,v_n)\in \TT^n$.
  Pick $c_\beta'\in K$ for $|\beta|\leq d$ and $c_{\alpha,\beta}'\in K$ for $|\alpha|\leq d$, $\beta\leq\alpha$, so that
  \begin{equation}
      \label{eq:formula}
      h(y_\alpha+c_\alpha; x+a) \eqqcolon 
      \sum_{|\beta|\leq d} c_{\beta}'\cdot x^\beta + \sum_{\substack{|\alpha|\leq d\\ \beta\leq\alpha}} c_{\alpha,\beta}'\cdot y_\alpha\cdot x^\beta  
  \end{equation}
  for $h\coloneqq \sum_{|\alpha|\leq d} y_\alpha x^\alpha$ in the polynomial ring $K[y,x]\coloneqq K[y_\alpha\mid \alpha\in\ZZ_{\geq 0}^n, |\alpha|\leq d][x_i\mid i=1,\dots,n]$.  Then
  \begin{align*}
    |f(B)| &= \max\left\{
      \max_{|\beta|\leq d} |c_\beta'|\cdot r_1^{\beta_1}\cdot \ldots\cdot r_n^{\beta_n},
      \max_{\substack{|\alpha|\leq d\\ \beta \le \alpha}} |c_{\alpha,\beta}'|\cdot s_\alpha\cdot r_1^{\beta_1}\cdot \ldots\cdot r_n^{\beta_n}
    \right\} \quad \text{and}\\
    \val(f(B)) &= \min\left\{
      \min_{|\beta|\leq d} \val(c_\beta') + \beta\cdot v,
      \min_{\substack{|\alpha|\leq d\\ \beta \le \alpha}} \val(c_{\alpha,\beta}') + w_\alpha + \beta\cdot v
    \right\}.
  \end{align*}
\end{lemma}
\begin{proof}
  We will prove the formula for $|f(B)|$, the formula for $\val(f(B))$ then follows from it straightforwardly.
  Let $g\coloneqq \sum_{|\alpha|\leq d} y_\alpha x^\alpha\in K[y,x]$ so that $|g(\prod_{|\alpha|\leq d}C_\alpha\times B)|=|f(B)|$.  By definition, we have for all $|\beta|\leq d$:
  \begin{align*}
    &\sup_{z\in B}|(z_1-a_1)^{\beta_1}\cdot\ldots\cdot (z_n-a_n)^{\beta_n}|\\
    &\qquad = \sup_{z_1\in B(a_1,r_1)}\Big(|(z_1-a_1)|^{\beta_1}\Big) \cdot\ldots\cdot \sup_{z_n\in B(a_n,r_n)}\Big(|(z_n-a_n)|^{\beta_n}\Big)
    = r_1^{\beta_1}\cdots r_n^{\beta_n}.
  \end{align*}
  Similarly, for all $|\alpha|\leq d$ and $\beta\leq\alpha$, we have
  \begin{align*}
    &\sup_{\substack{z_\alpha\in C_\alpha\\ z\in B}} |(z_\alpha-c_\alpha)(z_1-a_1)^{\beta_1}\cdot\ldots\cdot (z_n-a_n)^{\beta_n}|\\
    &\qquad = \sup_{z_\alpha\in C_\alpha} |z_\alpha-c_\alpha| \cdot \sup_{z_1\in B(a_1,r_1)}\Big(|(z_1-a_1)|^{\beta_1}\Big) \cdot\ldots\cdot \sup_{z_n\in B(a_n,r_n)}\Big(|(z_n-a_n)|^{\beta_n}\Big)\\
    &\qquad = s_\alpha r_1^{\beta_1}\cdots r_n^{\beta_n}.
  \end{align*}
  Combining the above with Equation \eqref{eq:formula} and the strong triangle inequality thus yields
  \begin{equation*}
    \left|g\Big(\prod_{|\alpha|\leq d}C_\alpha\times B\Big)\right|\leq
    \max\left\{
      \max_{|\beta|\leq d}|c_\beta'|\cdot r_1^{\beta_1}\cdots r_n^{\beta_n},
      \max_{\substack{|\alpha|\leq d\\ \beta \le \alpha}} |c_{\alpha,\beta}'|\cdot s_\alpha\cdot r_1^{\beta_1}\cdots r_n^{\beta_n}
    \right\}.
  \end{equation*}
  If the maximum is attained once, then we have equality.  If not, then perturbing the $s_\alpha$ and $r_i$ will make it so and we obtain equality via continuity.
\end{proof}

As an immediate corollary, we obtain the following result.

\begin{corollary}\label{cor:absolutePolynomialPiecewisePolynomial}
  The absolute polynomial $|f|\colon \mathcal B^n\rightarrow\RR_{\geq 0}$ is finite piecewise polynomial along geodesics.  The valuative polynomial $\val(f)\colon\mathcal H^n\rightarrow\RR$ is piecewise linear (or a tropical polynomial) along geodesics.
\end{corollary}

For polynomials with coefficients in $K$, \cref{lem:formula} simplifies as follows.

\begin{corollary}
  \label{cor:formulaCoefficientsInField}
  For $f\coloneqq \sum_{|\alpha|\leq d} c_\alpha\cdot x^\alpha \in K[x_1, \dotsc, x_n]$ and $B=B_{\mathcal B}(a,r)\in\mathcal B^n$ or $B=B_{\mathcal H}(a,v)\in\mathcal H^n$ we have
  \begin{equation*}
    |f(B)| = \max_{|\alpha|\leq d} \big(|c_\alpha'| \cdot r_1^{\alpha_1}\cdot\ldots\cdot r_n^{\alpha_n}\big) \quad \text{and} \quad \val(f(B)) = \min_{|\alpha|\leq d} \big(\val(c_\alpha') + \alpha_1 \cdot v_1 +\dots+ \alpha_n\cdot v_n\big),
  \end{equation*}
  where $c_\alpha'\in K$ are chosen such that $f(x+a)\eqqcolon\sum_{|\alpha|\leq d} c_\alpha'\cdot x^\alpha$.
\end{corollary}

We have seen that absolute and valuation polynomials exhibit highly structured behaviour along geodesics and are controlled by finitely many polynomial or linear regimes. We now give examples to illustrate this piecewise structure in explicit situations and demonstrate how it reflects the algebraic properties of the underlying polynomial. In particular, we will see how both the location of the roots and the radii of the coefficients influence the resulting piecewise structure.

\begin{example}
  \label{ex:formula}
  Let $K = \FF_2(\!(t)\!)$ and consider $f=(x+1)(x+t^2)(x+t^4)\in K[x]\subseteq \mathcal B[x]$ which can also be written as follows:
  {
    \setlength{\arraycolsep}{2pt}
    \begin{align*}
      \begin{array}{rcrcrrcrrcr}
        f&=&(x+1)^3&+&(t^4 + t^2)&(x+1)^2&+&(t^6 + t^4 + t^2 + 1)&(x+1)\phantom{.}&&\\[1mm]
         &=&(x+t)^3&+&(t^4 + t^2 + t + 1)&(x+t)^2&+&(t^6 + t^4)&(x+t)\phantom{.}&+&(t^7 + t^5 + t^4 + t^2)\\[1mm]
         &=&(x+t^2)^3&+&(t^4 + 1)&(x+t^2)^2&+&(t^6 + t^2)&(x+t^2)\phantom{.}\\[1mm]
         &=&(x+t^4)^3&+&(t^2 + 1)&(x+t^4)^2&+&(t^8 + t^6 + t^4 + t^2)&(x+t^4).    \end{array}
    \end{align*}
  }
  By \cref{cor:formulaCoefficientsInField}, we then have for $v\in\RR$:
  \begin{align*}
    \val(f(B_{\mathcal{H}}(1,v)))
    &=\min(0+3v,2+2v,0+v)
    &\val(f(B_{\mathcal{H}}(t^2,v)))
    &=\min(0+3v,0+2v,2+v)\\
    &=
      \begin{cases}
        3v &\text{if }v\leq 0\\
        v&\text{if }v\geq 0
      \end{cases}
    & &=
        \begin{cases}
          3v&\text{if }v\leq 0\\
          2v&\text{if }0\leq v\leq 2\\
          2+v&\text{if }v\geq 2
        \end{cases}
    \\
    \val(f(B_{\mathcal{H}}(t,v)))&=\min(0+3v,0+2v,4+v,2)
    &\val(f(B_{\mathcal{H}}(t^4,v)))&=\min(0+3v,0+2v,2+v)\\
    &=
      \begin{cases}
        3v&\text{if }v\leq 0\\
        2v&\text{if }0\leq v\leq 1\\
        2&\text{if }v\geq 1
      \end{cases}
    & &=
        \begin{cases}
          3v&\text{if }v\leq 0\\
          2v&\text{if }0\leq v\leq 2\\
          2+v&\text{if }v\geq 2
        \end{cases}
  \end{align*}
  \cref{fig:formula} illustrates the values of $\val(f(B))$ on the geodesics $B_{\mathcal{H}}(a,v)$ for $a\in\{1,t,t^2,t^4\}$ as well as on $\mathcal H^1$. Notice how the function diverges to $\infty$ on branches ending at the roots of $f$ and eventually becomes constant on the other branches.
\end{example}

\begin{figure}[t]
  \centering
  \begin{tikzpicture}
    \coordinate (b0) at (0,0);
    \coordinate (b1) at (-1,-1);
    \coordinate (b2) at (-2,-2);
    \coordinate (e0) at (3,-3);
    \coordinate (e1) at (1,-3);
    \coordinate (e2) at (-1,-3);
    \coordinate (e4) at (-3,-3);
    \draw
    (b0) -- node[blue!70!black,right,inner sep=2pt] {$3v$} ++(0,1)
    (b0) -- node[blue!70!black,above,sloped,inner sep=2pt] {$v$} (e0)
    (b0) -- node[blue!70!black,above,sloped,inner sep=2pt] {$2v$} (b1)
    (b1) -- node[blue!70!black,above,sloped,inner sep=2pt] {$2v$} (b2)
    (b1) -- node[blue!70!black,above,sloped,inner sep=2pt] {$2$}  (e1)
    (b2) -- node[blue!70!black,above,sloped,inner sep=2pt] {$v$} (e2)
    (b2) -- node[blue!70!black,above,sloped,inner sep=2pt] {$v$} (e4);
    \draw[dotted]
    (e4) -- ++(-0.25,-0.25) node[anchor=north east] {$t^4$}
    (e2) -- ++(0.25,-0.25) node[anchor=north west] {$t^2$}
    (e1) -- ++(0.25,-0.25) node[anchor=north west] {$t$}
    (e0) -- ++(0.25,-0.25) node[anchor=north west] {$1$};
    \fill (b0) circle (2pt);
    \fill (b1) circle (2pt);
    \fill (b2) circle (2pt);
    \node[anchor=south east] at (b0) {$B_{\mathcal{H}}(1,0)$};
    \node[anchor=south east] at (b1) {$B_{\mathcal{H}}(t,1)$};
    \node[anchor=south east] at (b2) {$B_{\mathcal{H}}(t^2,2)$};
    \draw[draw opacity=0.5,very thin,loosely dashed]
    (-3.5,-2) -- (4,-2) node[right] {$v=2$}
    (-3.5,-1) -- (4,-1) node[right] {$v=1$}
    (-3.5,-0) -- (4,0) node[right] {$v=0$};
  \end{tikzpicture}\vspace{-3mm}
  \caption{Values of $\val(f)$ on $\mathcal H^1$ for $f$ from \cref{ex:formula}.}
  \label{fig:formula}
\end{figure}
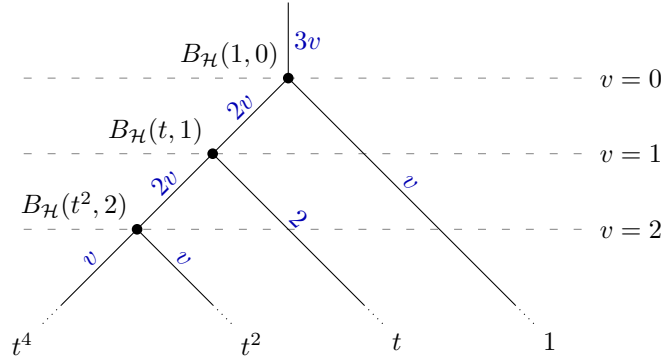

The next example illustrates how the coefficient radii prevent $\val(f(B))$ from attaining arbitrarily high values as $B$ converges to a root of $f_0$:

\begin{example}
  \label{ex:formula2}
  Suppose $K = \FF_2(\!(t)\!)$ and consider the following linear univariate polynomial over $\mathcal B$:
  \begin{align*}
    f = B_1\cdot x+B_0 \qquad\text{with}\quad B_1=B_{\mathcal{H}}\Big(\sum_{i=0}^{\infty} t^i,w_1\Big) \quad\text{and}\quad B_0=B_{\mathcal{H}}\Big(\sum_{i=0}^{\infty}t^{2i},w_0\Big) \text{ for some } w_1,w_0\geq 0,
  \end{align*}
  and consider the two families of discs $B_\calH(0,v)$ and $B_\calH(\sum_{i=0}^{\infty}t^i,v)$ for $v\in\RR$.  In the notation of \cref{lem:formula}, we then have $h = y_1x+y_0$ and
  \begin{align*}
     h\Big(y_1 + \sum_{i=0}^{\infty} t^i, y_0 + \sum_{i=0}^{\infty} t^{2i}, x+0\Big) &= \underbrace{\Big(\sum_{i=0}^{\infty} t^i\Big)}_{\val(\ldots)=0}\cdot x + \Big[\underbrace{\Big(\sum_{i=0}^{\infty} t^i\Big)\cdot 0 + \Big(\sum_{i=0}^{\infty} t^{2i}\Big)}_{\val(\ldots)=0}\Big] + y_1x+0\cdot y_1 + y_0 \\
     h\Big(y_1 + \sum_{i=0}^{\infty} t^i, y_0 + \sum_{i=0}^{\infty} t^{2i}, x+\sum_{i=0}^{\infty}t^i\Big)&= \underbrace{\Big(\sum_{i=0}^{\infty} t^i\Big)}_{\val(\ldots)=0}\cdot x + \Big[\underbrace{\Big(\sum_{i=0}^{\infty}t^i\Big)^2 + \Big(\sum_{i=0}^{\infty} t^{2i}\Big)}_{\val(\ldots)=\infty}\Big] + y_1x+\Big(\sum_{i=0}^{\infty}t^i\Big)\cdot y_1 + y_0.
  \end{align*}
  By \cref{lem:formula} we thus have:
  \begin{align*}
    \val(f(B_\calH(0,v))) &= \min\Big(0+v,0,0+w_1+v,\infty+w_1,0+w_0\Big) = \min(0,v,w_0) \qquad\text{and}\\
    \val(f(B_\calH(\textstyle\sum_{i=0}^{\infty}t^i,v))) &= \min\Big(0+v,\infty,0+w_1+v,0+w_1,0,0+w_0\Big) = \min(v,w_1,w_0).
  \end{align*}
  As in \cref{ex:formula}, we see that evaluating $\val(f)$ at $B=B_{\mathcal{H}}(0,v)$ not centered around a root of $f$ leads to $\val(f(B))$ having a natural lower bound of $0$, which is not the case for evaluating $\val(f)$ at $B=B_{\mathcal{H}}(\sum_{i=0}^{\infty}t^i,v)$.  In contrast to \cref{ex:formula}, the coefficient radii $w_1$, $w_0$ naturally bound $\val(f)$ from below.
\end{example}

\cref{cor:absolutePolynomialPiecewisePolynomial} motivates the following definition:

\begin{definition}
  Let $f \in \mathcal B[x_1,\dots,x_n]$ be a polynomial and let $\vec{v}_{BB'}\in T_B \mathcal{B}^n$ be a tangent vector. Let $\gamma^{\mathcal{B}}_{BB'}$ and $\gamma^{\mathcal H}_{BB'}$ be the geodesics from $B$ to $B'$ with respect to the standard and hyperbolic metrics. By \cref{cor:absolutePolynomialPiecewisePolynomial}, we have for $t\in[0,1]$ sufficiently small
  \begin{equation*}
    |f(\gamma^{\mathcal B}_{BB'}(t))|=p(t) \qquad\text{and}\qquad \val(f(\gamma^{\mathcal H}_{BB'}(t)))=c'-s \cdot t \qquad\text{ for some }p\in\RR[t] \text{ and } c,c',s\in \RR.
  \end{equation*}
  As a consequence of \cref{lem:formula}, we have $\deg(p)=s$, and we refer to it as the \emph{degree} of $|f|$ or the \emph{slope} of $\val(f)$ in direction $\vec{v}_{BB'}$. We denote it by $\deg_{\vec{v}_{BB'}}(|f|)$ or $\slope_{\vec{v}_{BB'}}(\val(f))$, respectively.
\end{definition}

The observation in \cref{ex:formula} for polynomials over $K$ can be formalised as follows.

\begin{lemma}
  \label{lem:slopeAndRootsUnivariateField}
  Let $f\in K[x]$ be a univariate polynomial over $K$, let $B=B_{\calB}(a,r)\in\mathcal B^1$ be a disc, and let $\vec{v}\in T_B\mathcal B$ be a tangent direction.  Without loss of generality, we may assume that $\vec{v}=\vec{v}_{B,0}$ or $\vec{v}=\vec{v}_{B,\infty}$.  Then
  \begin{equation*}
    \deg_{\vec{v}_{B,\infty}}(|f|)=- \slope_{\vec{v}_{B,\infty}}(\val(f))=|V_{\hat K}(f)\cap B| \quad\text{and}\quad \deg_{\vec{v}_{B,0}}(|f|)=\slope_{\vec{v}_{B,0}}(\val(f))=|V_{\hat K}(f)\cap B^\circ|
  \end{equation*}
  where $V_{\hat K}(f)$ denotes the solutions of $f$ in $\hat K$, $B^\circ\coloneqq \{z\in \hat K\mid |z-a|<r\}$ is the open disc that is the interior of $B$, and $|\cdot|$ counts solutions with multiplicity.
\end{lemma}

\begin{proof}
  Write $f = c \cdot \prod_{i = 1}^d (x - z_i)$ where $c, z_i \in \hat{K}$.  Let $B_\varepsilon\coloneqq B_\calB(a,r+\varepsilon)$.  We then have
  \begin{align*}
    \lvert f(B_\varepsilon) \rvert
    = \lvert c \rvert \cdot \max_{z\in B_\varepsilon} \Big(\prod_{i=1}^d |z-z_i|\Big)
    = \lvert c \rvert \cdot \Big(\prod_{\substack{i=1,\dots,d\\ z_i\in B_\varepsilon}} (r+\varepsilon) \Big) \cdot \Big( \prod_{\substack{i=1,\dots,d\\ z_i\notin B}} \lvert a - z_i \rvert \Big).
  \end{align*}
  For $\varepsilon\geq 0$ sufficiently small, we have $z_i\in B_\varepsilon$ if and only if $z_i\in B$, which implies the claim in direction $\vec{v}_{B,\infty}$.  For $\varepsilon<0$ sufficiently large, we have $z_i\in B_\varepsilon$ if and only if $z_i\in B^\circ$, which implies the claim in direction $\vec{v}_{B,0}$.
\end{proof}

We conclude our discussion on function theory for polydisc spaces by establishing a universal approximation theorem for the function classes introduced, which shows that linear combinations of absolute and valuation polynomials are sufficiently rich to approximate arbitrary continuous functions on compact subsets of $\mathcal{B}^n$. In this sense, these functions play a role analogous to classical polynomial or algebraic function classes in approximation theory.

\begin{theorem}
  \label{thm:universalApproximation}
  Let $D \subseteq \mathcal B^n$ be a compact subset, $f\colon D \to \RR$ a continuous function, and $\varepsilon > 0$.  Then there exist linear combinations of absolute polynomials $g\colon \mathcal B^n \to \RR$ such that $\lvert \lvert f - g \rvert \rvert _\infty \le \varepsilon$.
\end{theorem}
\begin{proof}
  It suffices to show that absolute polynomials with coefficients in $K$ are universal approximators.
  This follows from the Stone--Weierstrass theorem \cite[Theorem 7.32]{Rudin1976}, since the set of $\RR$-linear combinations of absolute polynomials is a subalgebra of the ring of functions $\mathcal{C}(D, \RR)$ that separates points and vanishes nowhere.
\end{proof}


\section{Optimisation on Polydisc Spaces}

\label{sec:optimisation}
Given our geometric and functional setup on polydisc spaces, we are now ready to study minimisation problems on $\mathcal{B}^n$. Given a function $\ell\colon\mathcal{B}^n \to \mathbb{R}$, we consider the problem of determining points at which $\ell$ attains its minimum:
\begin{equation*}
  \text{Find }B \in \mathcal{B}^n\text{ such that }\ell(B) = \inf_{B' \in \mathcal{B}^n} \ell(B').
\end{equation*}
In contrast to classical Euclidean settings, the geometry of $\mathcal{B}^n$ presents several technical challenges: The space is not isotropic and exhibits a branching structure, so that the local geometry varies significantly from point to point; see \cref{prop:tangentSpacePolydisk}.  As a result, standard notions and approaches in smooth optimisation do not apply directly and the behaviour of minimisers is instead controlled by the combinatorial and order-theoretic structure of $\mathcal{B}^n$.

Our goal is to develop a basic theory of optimisation adapted to the setting of polydisc spaces. We establish general conditions under which minimisers exist and describe their structure in terms of the geometry of $\mathcal{B}^n$: see \cref{thm:globalMinimumLocallyCompact}, \cref{thm:globalMinimumAffinePower} and \cref{thm:globalMinimumUnivariate}.

We are particularly interested in functions that are monotone with respect to the inclusion order (i.e., $\ell(B_1) \le \ell(B_2)$ whenever $B_1 \subseteq B_2$). This class is natural because monotonicity imposes strong constraints on the structure of the minimisers, reflecting the hierarchical geometry of the space. A central class of examples studied is sums of absolute polynomials, we show that functions of the form
$$
x \mapsto \sum_i a_i |f_i(x)|,
$$
where $f_i \in K[x_1, \ldots, x_n]$ and $a_i > 0$, admit global minimisers under mild conditions on the associated hypersurfaces. These examples illustrate how the algebraic structure previously developed gives rise to well-posed variational problems on $\mathcal{B}^n$.

\subsection{Absolute Polynomial Sums}

An important class of variational problems on polydisc spaces is obtained from positive linear combinations of absolute polynomials. These functions arise naturally from the function theory previously developed and retain enough algebraic structure for a detailed analysis of their minimisation behaviour.

Here, we study the geometry of absolute polynomial sums and relate their optimisation properties to the zero loci of the underlying polynomials. We begin with general structural results and show how the local behaviour of $|f|$ is controlled by the presence or absence of roots in the corresponding limit sets. This study leads to an existence theorem for global minimisers of positive linear combinations of absolute polynomials under natural conditions on the associated projective hypersurfaces. In the univariate setting, the theory simplifies further and yields an explicit bounded region of $\mathcal{B}^1$ containing all global minimisers.

We first establish a basic non-Archimedean rigidity property, which is that away from its zero locus, the absolute value of a polynomial is locally constant on products of discs.

\begin{lemma} \label{lem:clopenDiscsConstant}
    Let $f \in K[x_1, \dotsc, x_n]$ be a polynomial and $P$ be a product of $n$ discs in $\hat{K}$ (which may be open or closed) that contains no root of $f$. Then $|f(z)|=|f(z')|$ for any $z,z'\in P$.
\end{lemma}

\begin{proof}
  Let $z,z'\in P$.
  If $n=1$, write $f = c\cdot \prod_{i=1}^d (x - z_i)$, $c, z_i \in \hat{K}$.  Since $z_i\notin P$, we have $d(z, z') < d(z, z_i)$.  By the ultrametric inequality we then have $\lvert z - z_i \rvert = \max(|z-z'|,|z'-z_i|) = \lvert z' - z_i \rvert$ and thus $|f(z)|=|f(z')|$.
  For $n>1$, write $z=(z_1,\dots,z_n)$ and $z'=(z_1',\dots,z_n')$, and apply the $n=1$ case repeatedly to obtain
  \begin{equation*}
    \lvert f(z) \rvert = |f(z_1,\dots,z_n)| = \lvert f(z_1', z_2, \dotsc, z_n) \rvert = \dotsb = \lvert f(z_1', \dotsc, z_{n -1}', z_n) \rvert = \lvert f(z_1', \dotsc, z_n') \rvert = \lvert f(z') \rvert. \qedhere
  \end{equation*}
\end{proof}

The next technical definition is helpful for simplifying some of the later statements and corresponds to the set of $K$-points contained in all polydiscs near the source of a geodesic. 

\begin{definition}
    Let $\vec{v} = \overline\gamma \in T_p \mathcal{B}^n$ be a tangent vector. The \emph{limit set} of $\vec{v}$ at zero is defined as
    $$
    P + \varepsilon \cdot \mathbf{v} := \{z \in \hat{K}^n \mid z \in \gamma(t) \text{ for small enough } t > 0 \}.
    $$
\end{definition}

The examples below illustrate how the limit set $P + \varepsilon \cdot \mathbf{v}$ reflects the local changes in radii along a given tangent direction.

\begin{example}
    \begin{enumerate}
        \item Suppose $\vec{v}_{PQ}\in T_P\mathcal B^n$ with $P \le Q$. Then $P + \varepsilon \cdot \mathbf{v}_{PQ} = P.$
        \item Suppose that for small $t$, $\gamma(t) = B(x, r(t))$ where each coordinate of $r$ is linear in $t$. If $\mathcal{I}$ is the set of indices $i$ for which $r_i(t)$ is non-decreasing in $t$ then the limit set is given by
        $$
        P + \varepsilon \cdot \mathbf{v} = D_1 \times \dotsb \times D_n
        $$
        where $D_i = B(x_i, r_i(0))$ if $i \in \mathcal{I}$ and $D_i = D(x_i, r_i(0))$ otherwise, where we write
        $$
        D(x, r) \coloneqq \{y \in \hat{K} \mid d(y, x_i) < r \}.
        $$
    \end{enumerate}
\end{example}

The next lemma relates the limit set $P + \varepsilon \cdot \vec v$ to the behaviour of absolute polynomials along tangent directions; in particular, whether $|f|$ varies or is constant along a direction is related the presence of roots in the associated limit set.

\begin{lemma} \label{lem:degreeCriterion}
    Let $P = B(\alpha, r)$ be a polydisc, $f \in K[x_1, \dotsc, x_n]$ a polynomial, and $\vec{v} = \overline \gamma \in T_P \mathcal{B}^n$ a tangent vector. Suppose further that one of the following holds:
    \begin{itemize}
        \item The radii of $\gamma$ are all non-increasing and $P + \varepsilon \cdot \mathbf{v}$ contains no root of $f$; or
        \item The radii are all increasing, and for sufficiently small $t$, the set $\gamma(t)$ contains no root.
    \end{itemize} 
    Then $\deg_{\vec{v}} \lvert f \rvert = 0$. 
    Conversely, if all the radii are non-constant along $\gamma$ (i.e. the geodesic $\gamma$ joins two polydiscs $P, Q$ such that for all $i$, the discs $P_i$ and $Q_i$ differ) and $\deg_{\vec{v}} \lvert f \rvert  = 0$ then $f$ vanishes either nowhere or everywhere on $P + \varepsilon \cdot \mathbf{v}$, assuming that none of the radii are constant along $\gamma$.
\end{lemma}

\begin{proof}

We show the equivalence by relating the slope of $|f|$ along the direction $\vec{v}$ to the behaviour of $|f|$ on the associated limit set $P + \varepsilon \cdot \mathbf{v}$.

    We first show that the condition $\deg_{\vec{v}} \lvert f \rvert  = 0$ is equivalent to local constancy along the path $\gamma$. Let $x \in P + \varepsilon \cdot \mathbf{v}$. Clearly, if $\lvert f \rvert$ is constant along $\gamma(t)$ for small $t$, then $\deg_\mathbf{v} \lvert f \rvert  = 0$.
    Conversely, if $\deg_\vec{v} \lvert f \rvert  = 0$ and none of the coordinates of $\gamma$ are constant then for sufficiently small $t$, $\lvert f(\gamma(t)) \rvert = \lvert f(x) \rvert$. Indeed, for sufficiently small $t$ we may write $\gamma(t) = B(x, r(t))$ for some $r$ that is monotone in $t$. We then have
    $$
    \lvert f(\gamma(t)) \rvert = \max_{\vec{n}} \lvert a_\vec{n} \rvert r(t) ^ n, 
    $$
    where the $a_\mathbf{n}$ are the coefficients of the expansion of $f$ around $x$, i.e.
    $$
    f(\mathbf x) = \sum_n a_\mathbf{n} (\mathbf x - x)^ \mathbf n
    $$
    and the right hand side is only constant for small $t$ if it is equal to $\lvert a_0 \rvert = \lvert f(x) \rvert$. This establishes the claim.
    
    Next, we show that if $\deg_{\vec{v}} \lvert f \rvert =0$ and $\gamma$ has non-constant radii, then $f$ vanishes either nowhere or everywhere on $P + \varepsilon \cdot \mathbf{v}$. By above, if $\deg_\vec{v} \lvert f \rvert  = 0$ then for all $x, y \in P + \varepsilon \cdot \mathbf{v}$ we have $\lvert f(x) \rvert = \lvert f(y) \rvert$. In particular, $f$ either vanishes everywhere on the set $P + \varepsilon \cdot \mathbf{v}$, or nowhere.

    For the converse, note that the case where $f$ vanishes everywhere on the set $P + \varepsilon \cdot \mathbf{v}$ is trivial, so we can assume that $f$ vanishes nowhere on that set. Since $P + \varepsilon \cdot \mathbf{v}$ is a product of open or closed discs in $\hat{K}$, we can apply Lemma \ref{lem:clopenDiscsConstant} to get that $\lvert f \rvert$ is constant on $P + \varepsilon \cdot \mathbf{v}$, which concludes our proof.
\end{proof}

\begin{example}
    The condition that no radius of $\gamma$ is constant is necessary for the converse part of \cref{lem:degreeCriterion} to hold. Otherwise, one can for example take $f = x$ and $\gamma(t) = B_\mathcal{B}(0, 1) \times B_\mathcal{B}(0, 1 + t)$. Then, $\deg_{\overline \gamma} \lvert f \rvert  = 0$ but $f$ has roots in $P + \varepsilon \cdot \vec v = B_\mathcal B(0, 1) \times B_\mathcal B(0, 1)$ without vanishing everywhere on that set.
\end{example}

The preceding lemma identifies the slope with a purely algebraic and geometric condition on roots, by showing that the variation of $|f|$ along a given direction is entirely controlled by the presence of roots in the corresponding limit set.  In particular, away from roots, $|f|$ remains constant along paths, which leads to the following consequence: 

\begin{corollary} \label{cor:constantIfNoRootsPath}
    Suppose $P, Q$ are polydiscs such that the open segment $[P, Q]$ has only finitely many points that lie above a root of the polynomial $f$. Then $\lvert f(P) \rvert = \lvert f(Q) \rvert$. 
\end{corollary}

\begin{proof}
    First, note that by assumption we can write $[P, Q]$ as a union
    $$
        [P, Q] = \bigcup_{i=1}^{n-1} [P_i, P_{i+1}]
    $$
    with $P_1=P$ and $P_n = Q$.
    By continuity, it suffices to show that $\lvert f \rvert$ is constant on each open interval $I := [P_i, P_{i+1}] \setminus \{P_i, P_{i+1}\}$.
    Given any $R \in I$, let $\vec{v} = d_{RP_{i+1}}$ be the tangent vector corresponding to the path $\gamma$ from $R$ to $P_{i+1}$. Note that we have
    $P + \varepsilon \cdot \mathbf{v} \subset R$, so $P + \varepsilon \cdot \mathbf{v}$ does not contain any root of $f$. Thus, $\deg_{\vec{v}} |f| = 0$. This implies that $\lvert f \rvert$ is constant on $I$.
\end{proof}

Effectively, absolute polynomials are constant along large regions of geodesic paths and can vary only when the path lies above the zero locus of the polynomial. As a consequence, minimising positive sums of absolute polynomials is strongly related to the global arrangement of the corresponding zero sets. The next few theorems make this more precise, characterising the existence of local minima in terms of the intersections of the associated projective hypersurfaces, and describing where these global minima lie in relation to these hypersurfaces.

\begin{theorem}
  \label{thm:globalMinimumLocallyCompact}
  Let $\ell \coloneqq a_1|f_1| + \dots + a_m|f_m|$ be a positive linear combination of non-constant absolute polynomials, i.e., $a_i > 0$ and $f_i\in K[x_1,\dots,x_n]$, and denote the projective closure of the zero locus of $f_i$ by $V_i \subseteq \mathbb{P}^n$.  Suppose $K$ is locally compact.  Then $\ell$ has a global minimum in $\mathcal B^n$ if
  \begin{itemize}
    \item the $V_i$ intersect in $K^n$, i.e., $V_1 \cap \dots \cap V_m \cap K^n \neq \emptyset$; or
    \item the $V_i$ never intersect, i.e., $V_1 \cap \dots \cap V_m = \emptyset$.
  \end{itemize}
  Moreover, any such global minimum can be attained on $K^n$.
    In particular if $(V_1 \cap \dots \cap V_m) \setminus K^n = \emptyset$ then $\ell$ has a global minimum.
\end{theorem}
\begin{proof}
  Since $a_i>0$, we have $\ell(B)\leq \ell(B')$ for $B,B'\in\mathcal B^n$ with $B\subseteq B'$.  Consequently, $\ell$ attains a global minimum on $\mathcal B^n$ if and only if it attains a global minimum on $K^n\subseteq \mathcal B^n$.  Thus, it suffices to restrict attention to $\ell$ on $K^n$. Moreover, if the $V_i$ intersect in $K^n$ then such an intersection must be a global minimum of $\ell$ on $K^n$, so we can assume that the $V_i$ never intersect.
  Our strategy is to extend $\ell$ to a continuous function $\PP^n \to \RR\cup\{\infty\}$ by setting $\ell(z) = \infty$ for $z\in\PP^n\setminus K^n$ (note that as shown in \cref{ex:polynomialNoExtension} this must be a property of the whole sum rather than of the individual summands $\lvert f_i \rvert$) and then use a compactness argument.

  We first show that setting $\ell(z) = \infty$ for $z\in\PP^n\setminus K^n$ results in a continuous function. It suffices to show that for any sequence $(z_j) \subset K^n$ converging to $z$, we have $\ell(z_j) \to \infty$. This follows from our assumption of the $V_i$: since $V_1 \cap \dotsb \cap V_m = \emptyset$, there exists an $i$ such that $z \notin V_i$. This implies that $\lvert f_i(z_j) \rvert \to \infty$ as $j \to \infty$, and thus $\ell(z_j) \to \infty$.

  Since $K$ is locally compact, $\PP^n$ is compact, so $\ell$ attains a global minimum $x_*$ over $\PP^n$. Since $\ell = \infty$ on $\PP^n \setminus K^n$, it must be the case that $x_* \in K^n$, which concludes the proof. \qedhere
\end{proof}

\begin{remark}
  We note that the result does not assume that the absolute value on the field $K$ is non-Archimedean. In particular, the theorem above also holds when $K = \RR$ or $K = \CC$.
\end{remark}

The compactness argument used above is in a sense quite natural, but fails in general, as the following example shows.

\begin{example} \label{ex:polynomialNoExtension}
In general, an absolute polynomial does not admit a continuous extension to all of projective space. For instance, $|x_2|$ is not defined at the point $[1:0:0]$, and given any $\alpha \in |K^*|$, there exists a sequence of points $(a_n, b_n)$ in $\BA^2$ converging to $[1:0:0]$ such that $|b_n| = |x_2(a_n, b_n)|$ converges to $\alpha$.
\end{example}

The theorem above does not hold if we remove the local compactness assumption on $K$. In particular, there exist examples of absolute polynomial sums over $\CC_p$ that do not have a global minimum. A similar result can still be recovered at the cost of placing additional assumptions on the $f_i$.

\begin{theorem} \label{thm:globalMinimumAffinePower}
     Let $\ell \coloneqq a_1|f_1| + \dots + a_m|f_m|$ be a positive linear combination of non-constant absolute polynomials, i.e., $a_i > 0$ and $f_i\in K[x_1,\dots,x_n]$, and denote the projective closure of the zero locus of $f_i$ by $V_i\coloneqq V(f_i) \subseteq \mathbb{P}^n$. Suppose further that $f_i = g_i ^ {m_i}$ where $m_i>0$ and $g_i = \sum_{j=1}^na_{i,j} \cdot x_j + b_i$ for some $a_{i,j},b_i\in K$. Then $\ell$ has a global minimum in $\mathcal B^n$ if
  \begin{itemize}
    \item the $V_i$ intersect in $K^n$, i.e., $V_1 \cap \dots \cap V_m \cap K^n \neq \emptyset$; or
    \item the $V_i$ never intersect, i.e., $V_1 \cap \dots \cap V_m = \emptyset$.
  \end{itemize}
  In particular if $(V_1 \cap \dots \cap V_m) \setminus K^n = \emptyset$ then $\ell$ has a global minimum.
\end{theorem}

\begin{proof}
  Many parts of the theorem can be proven exactly as in \cref{thm:globalMinimumLocallyCompact}.  The only notable difference is showing the existence of a global minimum of $\ell$ on $K^n$ if the $V_i$ do not intersect.
  Consider the finite set $S$ of subsets of $K^n$ consisting of maximal non-empty intersections of the $V_i$:
  \begin{equation*}
    S = \Big\{ \bigcap_{j\in J} V_{j} \bigmid J\subseteq [m]\text{ with } \bigcap_{j\in J} V_{j}\neq\emptyset \text{ and } V_{j'}\cap \bigcap_{j\in J} V_{j}=\emptyset \text{ for }j'\notin J\Big\}
  \end{equation*}
  Note that each $V_j$ and thus also each element of $S$ is an affine subspace of $K^n$.

  We first claim that every $f_i$ is constant along each $V_J\coloneqq \bigcap_{j\in J} V_{j}$.  If \(i \in J\), then \(f_i\) vanishes identically on \(V_J\). We therefore assume \(i \notin J\). If $f_i$ were not constant on $V_J$, then $g_i(V_J)$ is a non-zero dimensional affine subspace of $K$, which implies $g_i(z) = 0$ for some $z\in V_J$, contradicting $V_i\cap V_J=\emptyset$.  Thus $\ell$ attains a unique value on each $V_J$, which we denote by $\ell(V_J)$.

  Next, we claim that for any $z' \in K^n$, there exist $V_J \in S$ and $z \in V_J$ such that $\ell(z')\geq\ell(z)$, which can be proven by induction on $n$.

  For $n=1$, each $V_i$ consists of a single point and we can pick $z\in\bigcup_{i=1}^mV_i$ to be a point that minimises the distance to $z'$.  Then we have $\ell(z') = \ell(z' \vee z) \ge \ell(z)$, where the first equality follows from \cref{cor:constantIfNoRootsPath}, and $z\in V_J$ for $J=\{i\in [n]\mid z\in V_i\}$.

  For $n>1$, pick any $V_i$ such that the distance from $z'$ to $V_i$ is minimal. Let $p \in V_i$ be a point that minimises this distance so that we again have $\ell(z') = \ell(z' \vee p) \ge \ell(p)$.  Restricting ourselves to $V_i\cong K^{n-1}$, we can apply induction to obtain some $z\in V_i$ such that $\ell(z')\geq\ell(p)\geq\ell(z)$ and $z\in \bigcap_{j\in J} V_j$ for some $J\subseteq [m]\setminus \{i\}$ with $V_{j'}\cap \bigcap_{j\in J} V_j=\emptyset$ for all $j'\in ([m]\setminus \{i\})\setminus J$.  Consequently, $z\in V_{\{i\}\cup J}\in S$.

 Thus we have $\inf_{x \in K^n} \ell(x) = \min_{V_J\in S} \ell(V_J)$, and $\ell$ has a global minimum in $K^n$.
\end{proof}

\begin{remark} \label{rmk:minimumInMaxiamlAffineSubspace}
  The proof of \cref{thm:globalMinimumAffinePower} actually indicates a stronger result, namely that the global minimum of the function $\ell$ lies in some maximal intersection of affine subspaces
  $
  V (g_i).
  $
  In other words, since the $g_i$ are assumed to have degree 1, the global minimum must lie in some affine subspace that can be directly computed from the $g_i$. 
  Since we can directly compute the value of $\ell$ on any such subspace, this gives us an explicit method for finding a global minimum $x_*$. 
\end{remark}

The geometric criterion just established applies to a broad range of concrete optimisation problems. We illustrate this first in the linear setting, where the existence of a minimiser can be interpreted in terms of the solvability of a system of equations.

\begin{example} \label{ex:linearPreimage}
    The condition on the zero loci is quite natural. For example, given a matrix $A$ and a vector $y$, consider the problem of finding a vector $x$ such that $Ax = y$. We can reframe this as an optimisation problem, namely the ``loss'' minimisation problem:
    \begin{equation} \label{eqn:optimisationProblem}
        \argmin_x \ell (x) \quad \text{where}\quad
        \ell(x) = \lvert \lvert A x - y \rvert \rvert ^ 2 = \sum_i \lvert g_i(x) ^ 2 \rvert \text{ and } g_i(x) = \sum_j A_{ij} x_j - y_i.
    \end{equation}

    Let $V_j$ be the projective closure of the variety defined by $g_j$ (this coincides with the projective closure of the variety defined by $g_j^2$). This is then just the zero locus of the homogenisation of $g_j$, namely 
    $$
        \sum_j A_{ij} x_j - y_i T \in K[x_1, \dotsc, x_n, T]
    $$
    A point $[x_1 : \dotsb : x_n : T]$ lies in the intersection $\bigcap_j V_j$ if and only if it satisfies 
    $
        A x = T y.
    $
    Now, a point $[x_1 : \dotsb : x_n : T]$ lies in $\mathbb{P} ^ n \setminus \mathbb{A}^n$ if and only if $T = 0$, so $\bigcap_j V_j \setminus \mathbb{A}^n \ne \emptyset$ if and only if there is a non-zero $x$ such that $A x = 0$. In particular, Problem (\ref{eqn:optimisationProblem}) has a solution whenever $A$ is non-singular. We shall investigate a generalisation of this problem further on in Section \ref{sec:applications}.
\end{example}

We have seen above how the criterion guarantees existence in a classical setting. We now show that the geometric condition on the zero loci is also essential: if the hypersurfaces meet only at infinity then the infimum need not be attained.

\begin{example}
  Consider the absolute polynomial sum $\ell(x_1,x_2)\coloneqq |x_1| + |x_1 x_2 - 1|$ on $\mathcal{B}^2$, and $K = \CC_p$. Clearly $x_1$ and $x_1x_2-1$ have no common root so $\ell > 0$ on $\mathcal{B}^2$, but one can find points in $\mathcal{B}^2$ where $\ell$ is arbitrarily small, e.g., if we take the sequence of points $(p^{j} + p^{2j}, p^{-j})$ then we have
  $$
  \ell(p^{j} + p^{2j}, p^{-j}) = 2p^{-j}
  $$
  which tends to zero at $j \to \infty$. Note that in $\PP^n$, the sequence of points $(p^{j} + p^{2j}, p^{-j})$ converges to the point $[0:1:0]$, which lies in the intersection of $V_1$ and $V_2$.
\end{example}

The previous example shows that intersections at infinity may obstruct attainment of a minimum. The mechanism behind the failure of existence is illustrated schematically in Figure~\ref{fig:no-global-minimum}.

\begin{figure}[t]
    \centering
\begin{tikzpicture}[x={(1.3,0)}, y={(0,1)}]
    
    \draw[->, thick, black] (-0.5,0) -- (5,0) node[right] {$x$};
    \draw[->, thick, black] (0,-0.5) -- (0,6) node[above] {$y$};
    
    \draw[red, very thick] (0,-0.3) -- (0,6);
    \node[red, anchor=east] at (0, 3) {$V(x)$};
    
    \draw[red, very thick, domain=0.18:4.8, samples=100] 
        plot (\x, {1/\x});
    \node[red, anchor=west] at (3.5, 0.45) {$V(xy - 1)$};
    
    
    \fill[blue] (2.2, 0.6) circle (2.5pt);
    \node[blue, anchor=west, font=\footnotesize] at (2.3, 0.65) {$j=1$};
    
    \fill[blue] (1.4, 1.0) circle (2.5pt);
    \node[blue, anchor=west, font=\footnotesize] at (1.5, 1.05) {$j=2$};
    
    \fill[blue] (0.85, 1.6) circle (2.5pt);
    \node[blue, anchor=west, font=\footnotesize] at (0.95, 1.65) {$j=3$};
    
    \fill[blue] (0.5, 2.5) circle (2.5pt);
    \node[blue, anchor=west, font=\footnotesize] at (0.6, 2.55) {$j=4$};
    
    \fill[blue] (0.28, 3.8) circle (2.5pt);
    \node[blue, anchor=west, font=\footnotesize] at (0.38, 3.85) {$j=5$};
    
    \fill[blue] (0.15, 5.2) circle (2.5pt);
    
    \draw[blue, dashed, thin] (2.2, 0.6) -- (1.4, 1.0) -- (0.85, 1.6) -- (0.5, 2.5) -- (0.28, 3.8) -- (0.15, 5.2);
    
    \draw[blue, thin, ->, dashed] (0.15, 5.2) -- (0.11, 5.8);
    
    \node[anchor=north east] at (-0.1,-0.1) {$O$};
    
    \node[blue, anchor=north west, font=\small] at (2.2, 5.8) {%
        $(p^{j} + p^{2j}, p^{-j})$
    };
    
\end{tikzpicture}\vspace{-2mm}
\caption{Absolute polynomial sum with no global minimum. As $j \to \infty$, the sequence converges to a point at infinity in $V(x) \cap V(xy-1)$}
  \label{fig:no-global-minimum}
\end{figure}
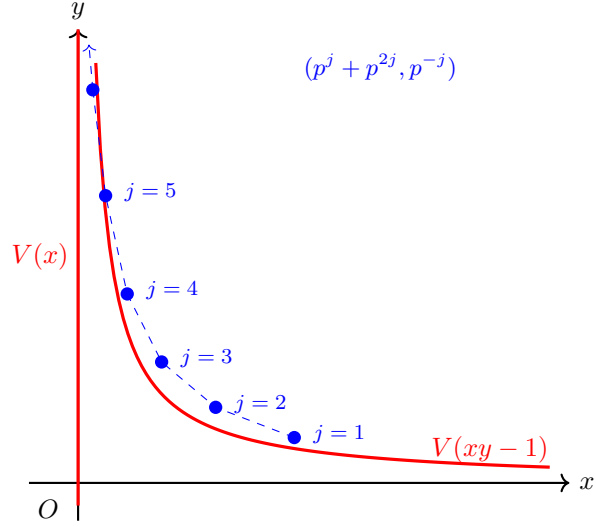

The absence of common affine roots does not by itself exclude the existence of global minima. We now consider a situation in which the minimum is attained despite the zero loci having no common point in \(\mathbb A^2\).

\begin{example}
    Suppose $K$ has characteristic different from 2.
    We work with polynomials in $K[x, y]$ and their homogenisations, which live in $K[x, y, z]$, for instance, $\ell = \lvert x ^ 2 + y^2 \rvert + \lvert x + 1 \rvert + \lvert y + 1 \rvert$. Let us first compute the projective hypersurfaces cut out by each of these polynomials. These are $V_1 = V(x ^ 2 + y ^ 2)$, $V_2 = V(x + z)$, and $V_3 = V(y + z)$.
    Then,
    $$
    V_2 \cap V_3 = \left\{[x : y : z] \mid x + z = y + z\right\} = \left\{[-z : -z : z] \mid z \ne 0 \right\}.
    $$
    In particular, points in $V_2 \cap V_3 \setminus \mathbb{A}^2 = \emptyset$. Thus,
    $$
    \left( V_1 \cap V_2 \cap V_3 \right) \setminus \mathbb{A}^2 = \emptyset
    $$
    and by Theorem \ref{thm:globalMinimumLocallyCompact}, $\ell$ has a global minimum. Notice that this cannot lie at a common root of the polynomials, since there is no such root.
\end{example}

\begin{example}
    Note that the converse of \cref{thm:globalMinimumLocallyCompact} and \cref{thm:globalMinimumAffinePower} does not hold. For example, consider the absolute polynomial sum
    $
    \ell(x, y) = \lvert y \rvert + \lvert y + 1 \rvert.
    $
    This attains a global minimum at any point in the lines $\{y = 0\}$ and $\{y = -1\}$, but these two lines only intersect at infinity.
\end{example}

The preceding examples concern existence of minimisers. We now turn to their geometry and describe where they lie in terms of the hypersurfaces cut out by the $f_i$. 

\begin{proposition}
  \label{prop:minIsRoot}
  Suppose $K$ is algebraically closed. If $\ell = a_1|f_1| + \dotsc + a_N |f_N|$ with $f_i\in K[x_1,\dots,x_n]\setminus K$ is a positive linear combination of non-constant absolute polynomials over $K^n$ and $B_*$ is a global minimum of $\ell$ over $\mathcal B^n$. Then $B_*$ contains a root of one of the $f_i$. 
\end{proposition}

\begin{proof}
    Suppose for a contradiction that $B_*(x_*, r_*)$ does not contain any roots of the $f_i$.
    Suppose $p$ is a root of some $f_j$, chosen to minimise the distance (with respect to the $\infty$-norm on $K^n$) to $x_*$. Without loss of generality, we can assume that for all $i$, $(x_*)_i \ne p_i$ (otherwise, just substitute $X_i = p_i$ and work with the resulting absolute polynomial sum defined over $\mathcal B ^ {n-1}$). Then, note that by \cref{lem:degreeCriterion} $\ell$ is constant along the geodesic path from $B_*$ to $B_* \vee p$. Moreover, each $f_i$ is decreasing along the path from $x_* \vee p$ to $p$, and $f_j$ is non-constant (again by \cref{lem:degreeCriterion}). Thus we have
    $$
        \ell(B_*) = \ell(B_* \vee p) > \ell(p),
    $$
    contradicting the fact that $B_*$ is a global minimum.
\end{proof}

The univariate case allows for a much more precise description. The one-dimensional space $\mathcal{B}^1$ is a tree, which therefore reduces the global minimisation problem to locating roots along a one-dimensional branching structure. Thus, global minima always exist and all minimisers are confined to an explicit compact region.

\begin{theorem}
  \label{thm:globalMinimumUnivariate}
  Let $\ell = a_1|f_1| + \dotsc + a_N |f_N|$ be a sum of non-constant absolute polynomials in one variable. Suppose that $K$ is locally compact.
  We write
  $$
    f_i(x)  = \sum_{j=0}^{n_i} b_{ij}x^j.
  $$
  Then $\ell$ has a global minimum on $\mathcal B^1$ and all global minima lie below the disc  $B(0, R)$ where
    $$
    R = \max_i \frac{|f_i(B_\mathrm{Gauss})|}{|b_{in_i}|}, \quad B_\mathrm{Gauss} = B_\mathcal{B}(0, 1).
    $$
    Moreover, if all the $f_i$ split over $K$ then any global minimum is attained at a root of one of the $f_i$.
\end{theorem}
\begin{proof}

  The existence of global minima follows directly from Theorem \ref{thm:globalMinimumLocallyCompact} and the fact that the projective variety defined by a univariate polynomial does not contain any point at infinity.
  We begin by showing that
  \begin{equation} \label{eqn:inclusion}
    \bigcup_i V_{\hat{K}} (f_i) \subset B(0, R).
  \end{equation}
  To see this, note that we have
  $
  |f_i(B(0, r))| = \max_j |b_{ij}| r^j.
  $
  In particular for $r \ge R$ this is a polynomial of degree $n$ in $r$, so if we set $\vec{u}$ to be the upwards-pointing tangent vector to $\mathcal{B}^1$ at $B(0, R)$ we have
  $
  \deg_\vec{u} \lvert f \rvert  = n.
  $
  
  We now check that if $B_* = B_\mathcal{B}(x_*, r_*)$ is a global minimum then $B_*$ is contained in a \emph{minimal disc containing a root of one of the $f_i$}. Indeed, let $p$ be a root of some $f_j$ that minimises the distance to $x_*$. Then, using \cref{lem:degreeCriterion}, one can check that $\ell(x_*) = \ell(x_* \vee p)$ and that for any strict subdisc $C \subset B_* \vee p$ containing $p$ we have $\ell(C) < \ell(B_* \vee p)$.
  The conclusion now follows from this fact, noting that \eqref{eqn:inclusion} implies that all minimal points of $\mathcal{B}^1$ containing a point in $\bigcup_i V_{\hat{K}} (f_i)$ must lie in $B(0, R)$.

  The fact that global minima are attained at roots if the polynomials split follows from an argument similar to the proof of \cref{prop:minIsRoot}.
\end{proof}

The case where all of the polynomials split over K is particularly straightforward: the global minimum is attained at a root of one of the polynomials. Moreover, no assumptions on the field K are required.

This theorem may be viewed as a geometric reinterpretation of the fact that the valuation ring $\mathcal{O}_K$ is integrally closed: if $f$ is a polynomial with coefficients in $\mathcal{O}_K$ then every root of $f$ (possibly in some larger valued field $L$) must also be integral (equivalently, have absolute value at most 1).

The bounded localisation established in Theorem \ref{thm:globalMinimumUnivariate} is a particular feature of one-dimensional geometry and does not extend to higher dimensions, which we illustrate in the example below.

\begin{example}
    In the higher dimensional case, we can in general no longer bound the size of the space we search over. For example in dimension 2, if we take $f = ax + by + c$ and $g = \alpha x + \beta y + \gamma$ such that $a\beta - b\alpha \ne 0$
    then the absolute polynomial sum
    $$
      \ell(x, y) = \lvert f(x, y) \rvert + \lvert g(x, y) \rvert 
    $$
    has exactly one global minimum, at the point
    $$
        \begin{pmatrix}
            u \\ v
        \end{pmatrix} =- \begin{pmatrix}
            a & b \\
            \alpha & \beta
        \end{pmatrix} ^ {-1} 
        \begin{pmatrix}
            c \\ \gamma
        \end{pmatrix}
    $$
    The coefficients $a, b, c, \alpha, \beta, \gamma$ may be chosen so that the minimiser $(u, v)$ lies arbitrarily far from the origin.
\end{example}

This shows that in dimensions $n\geq 2$, global minimisers need not lie in any uniformly bounded search region determined only by the coefficients of the defining polynomials.

\subsection{Extended Example: Fr\'echet Means}

We now illustrate the theory developed in this section with two extended examples. The first concerns Fr\'echet means, a classical notion of average in general metric spaces (see e.g., \cite{Lin2025-qn,Ferry2026-ns,talbut2026uniqueness}), which we study in the non-Archimedean setting.

\begin{definition}
    Let $x_1, \dotsc, x_m$ be points in $K^n$. We say $p \in K^n$ is a Fr\'echet mean of $(x_1, \dotsc, x_m)$ if it minimises the \emph{Fr\'echet variance}
    $$
        \psi(p) = \sum_{i=1}^m |x_i - p|^2.
    $$
    Here we are endowing $K^n$ with the natural $\ell^2$ norm.
\end{definition}

\begin{lemma}
    Every sequence of data points $x_1, \dotsc, x_m$ in $K^n$ has a Fr\'echet mean.
\end{lemma}

\begin{proof}
    This is a straightforward consequence of \cref{thm:globalMinimumAffinePower}.
\end{proof}

\begin{proposition}
    Every Fr\'echet mean of $x_1, \dotsc, x_m$ lies in the finite set with at most $m^n$ elements
    $$
    S = \left\{ a=(a_1,\dots,a_n) \in K^n \mid \text {for all } i=1,\dots,n \text{ there is some } j=1,\dots,m \text{ with } a_i = x_{ij} \right\}
    $$
\end{proposition}

\begin{proof}
    The Fr\'echet variance is an absolute polynomial on $\hat{K}^n$, and decomposes as the sum
    $$
        \psi = \sum_j \psi_j\qquad \text{where } \psi_j(a) = \sum_i |x_{ij} - a_j|^2.
    $$
    Notice that $\psi_j$ only depends on the $j$-th coordinate, and we will also write $\psi_j(a) = \psi_j(a_j)$. In particular, $a \in \hat{K}$ is a global minimum of $\psi$ if and only if each $a_j$ is a global minimum of $\psi_j$. In particular by \cref{thm:globalMinimumUnivariate}, if $a \in \hat{K}$ is a global minimum then $a \in S$. In order to conclude, we need to check that this is still true for global minima of $\psi$ when considered as a function over $K^n$. This follows from the fact that $\psi$ has a global minimum over $\hat{K}^n$ that lies inside $K^n$.
\end{proof}

\begin{example}
    The one-dimensional case is particularly rigid: the Fr\'echet mean of a set of points in $K$ is always one of the points. Note also that in this setting Fr\'echet means need not be unique. For instance, if $x_1 = 0$ and $x_2 = 1$ in $\mathbb{Q}_p$, then both 0 and 1 are Fr\'echet means of the data set.
\end{example}

We point out that the proof also shows that the problem of minimising the Fr\'echet variance is equivalent to minimising the $n$ components $\psi_j$.

\subsection{Extended Example: The Ordinary Least Squares Problem}

The second extended example concerns ordinary least squares, where the structure of global minimisers is made especially explicit by the theory developed above. The \emph{ordinary least squares problem} for data points $(x_1, y_1), \dotsc, (x_N, y_N)$ in $K^n \times K^{m}$ is the optimisation problem
$$
    \argmin_{A, b} \ \sum_i ||A x_i + b - y_i||^2 \qquad\text{where } A \in K^{m \times n} \text{ and } b \in K^{m}.
$$

\begin{corollary} \label{cor:ordinaryLeastSquares}
  The ordinary least squares problem has a solution over $\mathcal B ^ {(n+1)m}$ and over $K ^ {(n+1)m}$ whenever the data points are sufficiently general, in the sense that there is no pair $(A, b) \ne 0$ with $A x_i = b$ for all $i$.
\end{corollary}

\begin{proof}
    Again, we need to show that we can apply \cref{thm:globalMinimumAffinePower}. Consider the absolute polynomial sum
    $$
        \ell(A, b) = \sum_i \lvert \lvert A x_i + b - y_i \rvert \rvert ^2 = \sum_{i, j}  \lvert (A x_i)_j + b_j - y_{ij} \rvert ^2.
    $$
    We need to check that
    \begin{equation*}
        (\PP^{(n+1)m} \setminus \mathbb{A}^{(n+1)m}) \cap \bigcap_{i, j} V_{ij} = \emptyset
    \end{equation*}
    where $V_{ij}$ is the zero locus of the homogeneous polynomial $(A x_i)_j + b_j - Ty_{ij}$ in the variables $A, b, T$.

    The infinity points in this zero locus correspond to the points in $V_{ij}$ where $T = 0$. Thus if $P$ is a point in the right hand side of $(*)$ then we have
    $
    (A x_i)_j + b_j = 0
    $
    for all $i, j$, where $[A : b : T]$ are projective coordinates for $P$. This contradicts the assumption on the $x_i$.
\end{proof}

Note that the non-degeneracy assumption on the data points $(x_1, y_1), \dotsc, (x_N, y_N)$ implies that $N \ge n + 1$.
In the case where $m = 1$, \cref{rmk:minimumInMaxiamlAffineSubspace} allows us to recover a version of Theorem 1 in \cite{BMP2025}:

\begin{corollary}
    Assume that the data points $(x_i, y_i)$ lie in $K ^ n \times K$ and satisfy the assumptions of \cref{cor:ordinaryLeastSquares} and let $(A, b)$ be a solution to the Ordinary Least Squares problem. Then we have $A x_i + b = y_i$ for at least $n+1$ different indices $i$.
\end{corollary}

\begin{proof}
    This follows from the fact that $(A, b)$ lies in a maximal intersection of hypersurfaces of the form
    $$
        \left\{ (A, b) \mid A \cdot x_i + b = y_i \right\}
    $$
    and that because of our assumptions on the data, any such maximal intersection consists of at least $n+1$ hypersurfaces.
\end{proof}

In the case where $m  > 1$, a similar result holds coordinate-wise.
The linear nature of this optimisation problem means that it is rather well behaved. In particular, in this situation, moving from $K$ to a larger field does not give us better minima, i.e. we can already find ``the best possible solution'' over $K$.

\begin{lemma}
    Assume that we are given an absolute polynomial sum $\ell = \sum_i |f_i|^{r_i}$ such that each $f_i$ has degree 1 and coefficients in $K$. Let $L/K$ be an extension of ultrametric fields. Then any global minimum of $\ell$ over $K^n$ is also a global minimum of $\ell$ over $L^n$.
\end{lemma}

\begin{proof}
    Assume we are given a global minimum $P$ with coordinates in $L$. We need to show that we can obtain a global minimum $P'$ with coordinates in $K$. As seen in the proof of \cref{thm:globalMinimumAffinePower}, this minimum must lie in some affine subspace $A$ of $L ^ {n}$ that is the zero locus of several of the $f_i$, and $\ell$ must be constant on $A$. Since the $f_i$ have coefficients in $K$, it follows that $A$ is determined by an affine equation that has coefficients in $K$. This implies that $A \cap K^n \ne \varnothing$, and any point in this intersection yields a global minimum of $\ell$ over $L ^ n$ with coefficients in $K$.
\end{proof}

\begin{example}
    Here uniqueness also fails: if we take data points $(0, 1), (1, \frac{1}{3}), (2, 0) \in \QQ_3^2$ then any of the lines joining two of the points minimises the loss.
\end{example}

We note that the results do not use the fact that the loss function was constructed using square errors, and also apply for a loss function of the form
$$
    \ell(A, b) = \sum_i \lvert \lvert A x_i + b - y_i \rvert \rvert ^ q \qquad\text{for }q \in \{1, 2, \dotsc, \infty\}.
$$

\subsection{Optimisation Algorithms}

Where we previously considered structural existence results, we now turn to the algorithmic problem of computing minimisers on polydisc spaces. As opposed to Euclidean settings, optimisation on $\mathcal{B}^n$ must account for the non-Archimedean order structure and branching geometry of the underlying space.

Two broad classes of optimisation problems naturally arise in this setting:
\begin{enumerate}
\item \textbf{Lifted problems from $K^n$:} A function originally defined on $K^n$ is extended to $\mathcal{B}^n$ and optimisation is carried out on the larger space in order to ultimately recover a minimiser in $K^n$.
\item \textbf{Intrinsic problems on $\mathcal{B}^n$:} The objective function is defined directly on $\mathcal{B}^n$ and the minimiser search is carried out within polydisc space.
\end{enumerate}
Theoretically, the first class may be viewed as a special case of the second. However, from the algorithmic standpoint, the distinction is important, since lifted problems typically require procedures that will guarantee convergence to $K$-points.

Here, we mostly focus on the lifted setting, though the techniques proposed here are relevant more broadly: given a function $\ell:K^n\to\mathbb R$, we consider its natural extension to $\mathcal{B}^n$,
$$
\ell(B)=\sup_{z\in B} \ell(z),
$$
and develop several optimisation strategies adapted to the geometry of $\mathcal{B}^n$. These include descent methods based on the partial order and tree structure, best-first variants that prioritise promising branches, gradient-guided search procedures using directional information, deterministic optimistic optimisation methods, and Monte Carlo tree search methods balancing exploration and exploitation.

\begin{convention}
    \label{con:optimisation}
    Throughout this section, we assume that $K$ is locally compact and discretely valued. We will generally refer to the point $B_\mathcal{B}(0, 1) \in \mathcal B ^n$ as the \emph{Gauss point} and denote it $B_\mathrm{Gauss}$.
\end{convention}

\subsubsection{Descent Methods and Tree Search}

We begin with a broad and computationally natural class of optimisation procedures based on the hierarchical structure of polydisc spaces. The key idea is to search for minimisers by successively refining candidate regions: starting from a large polydisc, we repeatedly pass to smaller sub-polydiscs that appear more promising according to the objective function. This converts optimisation over the continuous space $\mathcal{B}^n$ into a discrete search problem on a nested family of regions.

Such methods are particularly well-suited to the non-Archimedean setting. The inclusion order on polydiscs provides a canonical notion of descent, while the branching structure of the space naturally induces tree- or graph-based search algorithms. Moreover, each iterate retains geometric meaning: rather than producing a single point estimate, the algorithm maintains a region known to contain candidate minimisers at the current resolution.

To make this notion precise, we focus first on a discrete skeleton of \(\mathcal B^n\) given by rational polydiscs, which serve as computational states for the search procedures developed below.

\paragraph{Hierarchical Structure on Rational Polydiscs.} In the next few sections, we focus on a discrete substructure of $\mathcal{B}^n$, namely the set of all rational polydiscs.

\begin{definition}
    The rational skeleton $\mathcal{B}^n_\text{rat}$ is the set of all rational polydiscs, i.e. all polydiscs whose absolute radii lie in the value group of $K$.
\end{definition}

These form an infinite directed acyclic graph, where directed edges $a \rightarrow b$ occur whenever polydisc $b$ can be obtained from polydisc $a$ by decreasing \emph{exactly} one radial coordinate. 
Equivalently, there exists an edge $a \to b$ whenever $b \subset a$ and there is no other rational polydisc contained strictly between the two. 
In particular, the set of rational discs $\mathcal{B}^1_\text{rat}$ forms an infinite (discrete) tree; see Figure \ref{fig:nestedPolydiscs3D}.

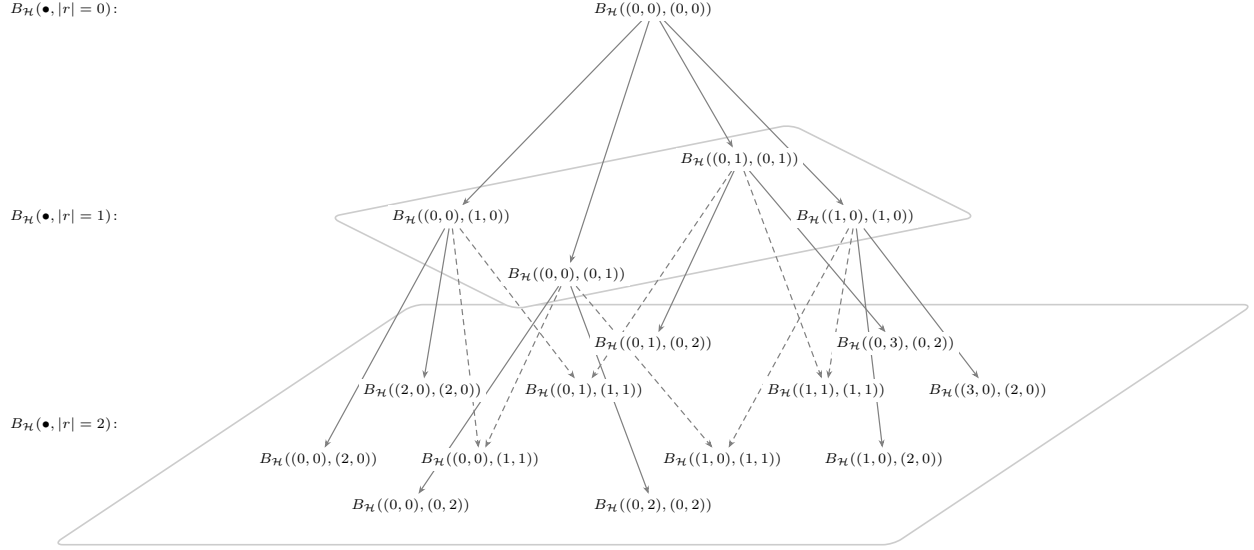
\begin{figure}[t]
  \centering
  \resizebox{\textwidth}{!}{%
  \begin{tikzpicture}[
      x={(1.4cm, 0cm)},         
      y={(0.6cm, 0.4cm)},       
      z={(0cm, 1.8cm)},         
      every node/.style={inner sep=2pt, font=\scriptsize, fill=white}, 
      arr/.style={-{Stealth[length=4pt,width=3pt]}, semithick, black!50},
      sarr/.style={-{Stealth[length=4pt,width=3pt]}, semithick, densely dashed, black!50}
    ]

    \begin{scope}[on background layer]
      \draw[draw=black!20, fill=white, thick, rounded corners=5pt] 
        (-5.2, -5.2, 0) -- (5.2, -5.2, 0) -- (5.2, 5.2, 0) -- (-5.2, 5.2, 0) -- cycle;

      \draw[draw=black!20, fill=white, thick, rounded corners=5pt] 
        (0, 4.0, 2) -- (4.0, 0, 2) -- (0, -4.0, 2) -- (-4.0, 0, 2) -- cycle;
    \end{scope}

    \node (root) at (0, 0, 4) {$B_{\mathcal{H}}((0,0),(0,0))$};

    \node (L1a) at (-2.5,  0.0, 2) {$B_{\mathcal{H}}((0,0),(1,0))$};
    \node (L1b) at ( 2.5,  0.0, 2) {$B_{\mathcal{H}}((1,0),(1,0))$};
    \node (L1c) at ( 0.0, -2.5, 2) {$B_{\mathcal{H}}((0,0),(0,1))$};
    \node (L1d) at ( 0.0,  2.5, 2) {$B_{\mathcal{H}}((0,1),(0,1))$};

    \node (n00_11) at (-1.5, -1.5, 0) {$B_{\mathcal{H}}((0,0),(1,1))$};
    \node (n01_11) at (-1.5,  1.5, 0) {$B_{\mathcal{H}}((0,1),(1,1))$};
    \node (n10_11) at ( 1.5, -1.5, 0) {$B_{\mathcal{H}}((1,0),(1,1))$};
    \node (n11_11) at ( 1.5,  1.5, 0) {$B_{\mathcal{H}}((1,1),(1,1))$};

    \node (n00_20) at (-3.5, -1.5, 0) {$B_{\mathcal{H}}((0,0),(2,0))$};
    \node (n20_20) at (-3.5,  1.5, 0) {$B_{\mathcal{H}}((2,0),(2,0))$};
    \node (n10_20) at ( 3.5, -1.5, 0) {$B_{\mathcal{H}}((1,0),(2,0))$};
    \node (n30_20) at ( 3.5,  1.5, 0) {$B_{\mathcal{H}}((3,0),(2,0))$};
    \node (n00_02) at (-1.5, -3.5, 0) {$B_{\mathcal{H}}((0,0),(0,2))$};
    \node (n02_02) at ( 1.5, -3.5, 0) {$B_{\mathcal{H}}((0,2),(0,2))$};
    \node (n01_02) at (-1.5,  3.5, 0) {$B_{\mathcal{H}}((0,1),(0,2))$};
    \node (n03_02) at ( 1.5,  3.5, 0) {$B_{\mathcal{H}}((0,3),(0,2))$};

    \begin{scope}[on background layer]
      \foreach \dest in {L1a, L1b, L1c, L1d}
        \draw[arr] (root) -- (\dest);

      \draw[arr]  (L1a) -- (n00_20);
      \draw[arr]  (L1a) -- (n20_20);
      \draw[sarr] (L1a) -- (n00_11);
      \draw[sarr] (L1a) -- (n01_11);

      \draw[arr]  (L1b) -- (n10_20);
      \draw[arr]  (L1b) -- (n30_20);
      \draw[sarr] (L1b) -- (n10_11);
      \draw[sarr] (L1b) -- (n11_11);

      \draw[sarr] (L1c) -- (n00_11);
      \draw[sarr] (L1c) -- (n10_11);
      \draw[arr]  (L1c) -- (n00_02);
      \draw[arr]  (L1c) -- (n02_02);

      \draw[sarr] (L1d) -- (n01_11);
      \draw[sarr] (L1d) -- (n11_11);
      \draw[arr]  (L1d) -- (n01_02);
      \draw[arr]  (L1d) -- (n03_02);
    \end{scope}

    \node[anchor=east] at (-6.5, 0, 4) {$B_{\mathcal{H}}(\bullet, |r|=0)\colon$};
    \node[anchor=east] at (-6.5, 0, 2) {$B_{\mathcal{H}}(\bullet, |r|=1)\colon$};
    \node[anchor=east] at (-6.5, 0, 0) {$B_{\mathcal{H}}(\bullet, |r|=2)\colon$};

  \end{tikzpicture}%
  }
  \caption{A portion of the DAG for $\mathcal{B}^2$ over $K = \QQ_2$, starting with the Gauss point at the top.}
  \label{fig:nestedPolydiscs3D}
\end{figure}

From now on, we will focus on descent algorithms on $\mathcal{B}^n$, i.e., optimisation procedures that construct sequences $B_1 \ge  B_2 \ge \dotsc$ where the $B_i$ are rational points in $\mathcal{B}^n$ and $B_{i+1}$ is computed from $B_i$ by some known procedure, which may be deterministic or stochastic.

Focusing on such algorithms has two immediate advantages. First, the optimisation process is reduced to a search over a simpler discrete hierarchical structure. Second, nested descent sequences are automatically well behaved: under completeness of the ground field, successive refinements cannot oscillate indefinitely and must converge to a limiting polydisc or point. The next lemma makes this idea precise.

\begin{lemma}
  Every chain of decreasing polydiscs converges in $\mathcal B^n$.
\end{lemma}

\begin{proof}
  Let $B_0 \supseteq B_1 \supseteq B_2 \supseteq \dots $ be a chain of decreasing polydiscs.  By local compactness of $K^n$ and \cref{lem:recentre}, we may write $B_i=B_{\mathcal B}(a,r_i)$ for some fixed $a\in K^n$ and $r_i\in\RR_{\geq 0}^n$ coordinatewise decreasing.  Then the $r_i$ must converge in $\RR_{\geq 0}^n$, which implies that the $B_i$ also converge.
\end{proof}

Given a rational polydisc $B \in \mathcal{B}^n$, we denote by $\mathcal{C}(B)$ the set of children of $B$, namely those rational polydiscs obtained from $B$ by a single refinement step. More specifically, let $\mathcal{C}_i(B)$ be the subset consisting of children obtained by increasing the $i$th valuative radial coordinate. Thus
$$
\mathcal{C}(B)=\bigcup_{i=1}^n \mathcal{C}_i(B),
$$
and we refer to the elements of $\mathcal{C}_i(B)$ as the \emph{children of $B$ along the $i$th coordinate}.

To describe these children explicitly, fix throughout this section a set $\mathcal S \subseteq \mathcal O_K$ of representatives for the residue classes modulo $\mathfrak m_K$, together with a uniformiser $\pi$ of $K$. In particular,
$$
|\mathcal S|=|\mathfrak K|.
$$
For $K=\QQ_p$, we take the standard choice $\mathcal S=\{0,\dotsc,p-1\}$ and $\pi=p$.

The next proposition shows that the branching structure of the search graph is completely explicit and uniformly controlled.

\begin{proposition}
 Suppose we are given a set of coset representatives $\{c_1, \dotsc, c_m \} \subset \mathcal{O}_K$ of the maximal ideal $\mathfrak{m}_K$ (so in particular $m = \lvert \mathfrak{K} \rvert$), and a uniformiser $\pi$ of $K$. Suppose $B$ is a rational polydisc with centre $a$ and valuative radius $v$. Then we have 
    \begin{equation}
    \label{eq:children_coset}
        \mathcal{C}_i(B) = \left\{ B_{\mathcal{H}}(a + c_j \pi ^ {r_i} e_i, v + e_i) \mid j = 1, \dotsc, m \right\}
    \end{equation}
    In particular, 
    $
        \lvert \mathcal{C}_i(B) \rvert = \lvert \mathfrak{K} \rvert
    \text{ and }
    \lvert \mathcal{C}(B) \rvert = n \lvert \mathfrak{K} \rvert.
    $
\end{proposition}

\begin{proof}

It suffices to consider the $i$th coordinate, since all other coordinates are unchanged when passing to a child in $\mathcal C_i(B)$.

We write $B=B_{\mathcal H}(a,v)$. Refining the $i$th coordinate means replacing the disc $B_{\mathcal H}(a_i,v_i)$ by one of its maximal proper rational subdiscs of valuative radius $v_i+1$. These subdiscs are indexed by the residue classes modulo $\mathfrak m_K$. For representatives $\mathcal S=\{c_1,\dots,c_m\}$, they are precisely
$$
B_{\mathcal H}(a_i+c_j\pi^{v_i},\, v_i+1),
\qquad j=1,\dots,m.
$$
Leaving the remaining coordinates fixed gives
$$
B_{\mathcal H}\big(a+c_j\pi^{v_i}e_i,\; v+e_i\big),
\qquad j=1,\dots,m,
$$
which proves \eqref{eq:children_coset} for \(\mathcal C_i(B)\).

Since the representatives $c_1,\dots,c_m$ give distinct residue classes, these children are distinct. Hence
$$
|\mathcal C_i(B)|=m=|\mathfrak K|.
$$
Summing over the $n$ possible coordinate refinements gives $|\mathcal C(B)|=n|\mathfrak K|$.
\end{proof}

The expression for the set of children is entirely explicit and therefore readily and easily computable. We illustrate the computation in the simplest nontrivial two-dimensional case.

\begin{example}
    Say $K = \QQ_3$ and $B = B_{\mathcal{H}}((0, 0), (0, 0)) = \ZZ_3 \times \ZZ_3$. Then, we have $\pi = \pi_{\mathbb Q_3} = 3$ and $\mathcal{S} = \{0, 1, 2\}$. Then
    \begin{align*}
            \mathcal{C}_1(B) & = \left\{ B_{\mathcal{H}}((0, 0), (1, 0)), B_{\mathcal{H}}((1, 0), (1, 0)), B_{\mathcal{H}}((2, 0), (1, 0)) \right\} \\
            \mathcal{C}_2(B) & = \left\{ B_{\mathcal{H}}((0, 0), (0, 1)), B_{\mathcal{H}}((0, 1), (0, 1)), B_{\mathcal{H}}((0, 2), (0, 1)) \right\}
    \end{align*}
\end{example}

The directed graph structure introduced above is not unique. It reflects the choice to refine a single coordinate radius at each step, which is natural for coordinate-wise descent methods. In some optimisation procedures, however, it may be advantageous to refine several coordinates simultaneously: for $d \in \{1,\dots,n\}$, we say that a polydisc $y$ is a \emph{degree $d$ child} of a polydisc $B$ if it is obtained from $B$ by increasing $d$ distinct valuative radii by 1. We denote the set of all degree $d$ children of $B$ by $\mathcal C^d(B)$. In particular, $\mathcal C(B)=\mathcal C^1(B).$ Different choices of $d$ lead to different search graphs and correspond to different levels of branching complexity.

\paragraph{Hierarchical Partitioning.}
A closely related viewpoint comes from hierarchical optimisation and state-space search, where the goal is to optimise a function $f\colon \mathcal X \to \mathbb R$ under a finite evaluation budget by recursively partitioning the search space into nested cells $X_h$, where $h$ denotes the depth or resolution level. The optimisation problem is then converted into a sequential decision problem, where at each stage, the child cell to be explored next is selected. 

The rational polydiscs of $\mathcal B^n$ provide a natural instance of such a hierarchy. Each refinement step replaces a current region by smaller descendant polydiscs, yielding an adaptive partition of $K^n$ (or more generally of the relevant search domain). Consequently, optimisation over polydisc spaces can be interpreted through the lens of hierarchical search. This perspective is useful algorithmically, since it connects our setting to established ideas from global optimisation and bandit methods. In particular, effective procedures must balance \emph{exploration}, by testing new branches of $\mathcal B^n_{\mathrm{rat}}$, with \emph{exploitation}, by refining branches that already appear promising.

The methods we discuss next may be viewed as concrete realisations of this general search principle.

\subsubsection{Best-First Descent} \label{subsec:best-first-search}

A natural baseline strategy is to follow a greedy refinement rule, selecting at each step the child polydisc with the smallest currently observed objective value. We refer to this procedure as \emph{best-first descent}. Starting from a current iterate \(x_i\), the next iterate is chosen according to
$$
B_{i+1}\in \argmin_{C\in \mathcal S(B_i)} \ell(C),
$$
where $\mathcal S(B)\subseteq \mathcal C(B)$ is a prescribed collection of admissible children.

The choice of $\mathcal{S}(B)$ determines the computational cost of each update. Taking $\mathcal{S}(B) = \mathcal{C}(B)$ compares all immediate refinements and therefore makes the most informed local decision. In contrast, attention may be restricted to subsets such as $\mathcal{C}_i(B)$, for example, by cycling through coordinates or selecting coordinates adaptively. This yields a standard trade-off between per-iteration cost and quality of the descent direction.

Minimisers within $\mathcal{S}_i$ need not be unique. For example, if $\ell = |X^2-1|$ over $K = \QQ_2$, then the Gauss point has two children with identical objective value. In such cases, we require a tie-breaking rule. In our work, ties may be resolved deterministically or by uniform random choice.

When an algorithm requires selecting uniformly from a finite set $S$ of candidate points or branches, we will write $B \sim S$.

\begin{algorithm}[H]
  \begin{algorithmic}[1]
    \Require $\ell : \mathcal{B}^n \to \RR$, $N_\mathrm{epochs} \ge 1$, $B_0 \in \mathcal{B}^n$
    \For{$i = 1, \dotsc, N_\mathrm{epochs}$} 
      \State $B_i \sim \argmin_{C \in \mathcal{C}(B_{i-1})} \ell(C)$ \Comment{Pick a random minimising child}
    \EndFor
  \Return $B_{N_\mathrm{epochs}}$
  \end{algorithmic}
  \caption{Best-First Descent}
  \label{alg:best-first-value-search}
\end{algorithm}

\begin{algorithm}[H]
  \begin{algorithmic}[1]
    \Require $\ell : \mathcal{B}^n \to \RR$, $N_\mathrm{epochs} \ge 1$, $B_0 \in \mathcal{B}^n$
    \State Compute set of coset representatives for $\mathfrak{m}_K$.
    \For{$i = 1, \dotsc, N_\mathrm{epochs}$}
      \State $j \gets (i-1 \mod n) + 1$.
      \Comment{At each step, switch to the next coordinate or loop back to the first}
      \State $\mathcal{S} \gets \mathcal{C}_j(B_{i-1})$.
      \State $B_i \sim \argmin_{C \in \mathcal{S}} \ell(C)$ \Comment{Pick a random minimising child}
    \EndFor
  \Return $B_{N_\mathrm{epochs}}$
  \end{algorithmic}
  \caption{Coordinate-wise Best-First Descent}
  \label{alg:coordinate-wise-best-first-value-search}
\end{algorithm}

Notice these two algorithms can exhibit radically different behaviour in practice: In Algorithm \ref{alg:best-first-value-search}, no constraint is imposed on how coordinates are refined, so some radii may stabilise at positive values and the sequence need not necessarily converge to a $K$-point. In contrast, Algorithm \ref{alg:coordinate-wise-best-first-value-search} enforces refinement in every coordinate which guarantees convergence to a $K$-point, i.e., $\lim_{i \to \infty} B_i \in K^n$. This makes Algorithm \ref{alg:coordinate-wise-best-first-value-search} well suited to problems whose minimisers lie in $K^n$, but more limited when the minimum is attained only at a non-$K$-point.  This observation suggests that enforcing coordinate-wise refinement may be necessary when the goal is to recover classical solutions in $K^n$, which we formalise in the case of polynomial objectives below. We show that the coordinate-wise best-first descent procedure indeed converges to a root.

\begin{proposition}\label{thm:best-first-descent-convergence}
    Let $f$ be a univariate polynomial that splits over $K$ and $x_0$ a polydisc containing a root of $f$. Let $B_0, B_1, \dotsc$ be the sequence of polydiscs obtained by running the best-first descent algorithm indefinitely. Then this sequence converges to a root of $f$ in $K$.
\end{proposition}

\begin{proof}
    It suffices to show that (i) for all $m$, $B_m$ contains a root of $f$, and that (ii) $\lim_{m \to \infty} B_m \in K$.
    To prove (i), notice that by assumption, $x_0$ contains a root. Suppose that $B_m$ contains a root but $B_{m+1}$ does not. By construction, $B_{m+1} \in \mathcal{C}_i(B_m)$ for some $i$.
    Since 
    $$
        B_m \cap K = \bigcup_{C \in \mathcal{C}_i(B_m)} C \cap K,
    $$
    the root must be contained in some other $\alpha \in \mathcal{C}_i(B_m)$. Moreover, we have
    $
    f(B_{m+1}) = f(B_m) > f(\alpha)
    $
    which is a contradiction.
    By construction, the radius of $B_m$ is equal to $\lvert \pi \rvert ^{m}r_i$ where $r$ is the radius of $B_0$. This tends to $0$ as $m \to \infty$, so $\lim_{m \to \infty} B_m \in K$, proving (ii).
\end{proof}

\subsubsection{Best-First Gradient Descent}
\label{subsec:best-first-gradient-search}

A key advantage of working over the polydisc space $\mathcal{B}^n$ rather than $K^n$ is that we have a notion of geodesics, which allows for calculus with real valued functions, and thus a calculus-based selection of branch and descent directions, rather than purely combinatorial descent.

To formalise this idea, we require a notion of directional derivative along tangent directions.

\begin{definition}[Directional derivatives]
Let $f$ be a continuous real valued function defined on an open neighbourhood $U$ of $B \in \mathcal{B}^n$. Given $\vec{v} = \overline\gamma \in T_B \mathcal{B}^n$ where $\gamma$ has unit speed geodesic, we define the \emph{directional derivative} of $f$ along $\vec{v}$ to be the value
$$
d_{\vec{v}} f = \lim_{h \to 0^+} \frac{f \circ \gamma (h) - f(B)}{h}
$$
provided this limit exists. If this quantity always exists for all $\vec{v} \in T_B \mathcal{B}^n$, the function $f$ is said to be \emph{differentiable} at $B$. 
\end{definition}

When $f$ is differentiable, each tangent direction $\vec{v} \in T_B \mathcal{B}^n$ is assigned a directional derivative $d_{\vec{v}} f$, which quantifies the local variation of $f$ along the corresponding branch.

The usefulness of this notion lies in the fact that it behaves in close analogy with classical differentiation. In particular, directional derivatives interact well with algebraic operations, which allows them to be computed explicitly for structured functions such as absolute polynomial sums. The following lemma collects these basic properties.

\begin{lemma}
Assume that $f, g\colon \mathcal{B}^n \to \RR$ are differentiable at $B$, and $\vec{v} \in T_B \mathcal{B}^n$. Then
\begin{enumerate}
    \item $d_{\vec{v}} (f + g) = d_{\vec{v}} f + d_{\vec{v}} g$, 
    \item $d_{\vec{v}} (fg) = f(B)  d_{\vec{v}} g + g(B) d_{\vec{v}} f$, 
    \item $d_{\vec{v}} (h \circ f) = h'(f(B)) d_{\vec{v}} f$ for any $h\colon \RR \to \RR$ differentiable at $B$.
\end{enumerate}
\end{lemma}

\begin{proof}
The definition reduces directional derivatives to one-dimensional limits along paths. Writing $f \circ \gamma$ for a representative of $\vec{v}$, each statement follows from the corresponding differentiation rule for real-valued functions.
\end{proof}

The notion of directional derivatives identifies, at each point, the directions along which the function decreases most rapidly. To turn this into an algorithmic update rule, a discrete analogue of ``moving along a direction'' is required within the rational polydisc hierarchy.

\begin{definition}
    Given a rational polydisc $P$ and a tangent vector $\vec{v} = \overline\gamma \in T_P\mathcal{B}^n$, we define the \emph{translation of $P$ by $\vec{v}$}, $P + \vec v$, to be the closest rational polydisc to $P$ on the path $\gamma$ if it exists (up to extending $\gamma$). This definition is independent of the choice of representative path $\gamma$.
\end{definition}

This construction recovers the usual refinement step when the direction corresponds to a child: if $C$ is a child of $B$ and $\vec{v} = \overline\gamma_{BC}$, then $B + \vec v = C$.

We can then propose a gradient-guided refinement rule, in which the next iterate is chosen by selecting the direction of steepest local decrease and moving along the corresponding branch. More precisely, given a current iterate \(P_i\), define
\begin{align*}
    \vec v_i = \argmin_{\vec v \in \mathcal{T}(B_i)} d_{\vec{v}} \ell \qquad\text{and}\qquad
    B_{i + 1} = B_i + \vec v_i
\end{align*}
where $\mathcal{T}(B) \subset T_B \mathcal{B}^n$ is a subset of tangent vectors pointing downwards (i.e. each component is defined by a point in $\mathfrak K$ in the correspondence given by \cref{prop:tangentSpacePolydisk}. Concretely, this means that the vector is of the form $[d_{BC}]$ for some rational polydisc lying below $B$). We can take, for instance, $\mathcal{T}(B) = \{ d_{BC} \mid C \in \mathcal{C}(B)\}$. 
We point out that this allows for certain components to have magnitude zero, which corresponds to the fact that this descent rule does not force all radii to shrink.

\begin{proposition}\label{thm:best-first-gradient-descent-convergence}
    Let $f$ be a univariate polynomial that splits over $K$ and $x_0$ a polydisc containing a root of $f$. Let $B_0, B_1, \dotsc$ be the sequence of polydiscs obtained by running the best-first gradient descent algorithm indefinitely. Then this sequence converges to a root of $f$ in $K$.
\end{proposition}
\begin{proof}
    Same as \cref{thm:best-first-descent-convergence}
\end{proof}

\begin{algorithm}[H]
  \begin{algorithmic}[1]
    \Require $\ell\colon \mathcal{B}^n \to \RR$, $N_\mathrm{epochs} \ge 1$, $P_0 \in \mathcal{B}^n$
    \For{$i = 1, \dotsc, N_\mathrm{epochs}$} 
      \State Compute $\mathcal{V}(B_{i-1}) = \{d_{B_{i-1}C} \mid C \in \mathcal{C}(B_{i-1})\}$ 
      \State $\vec{v} \sim \argmin_{\vec{v} \in \mathcal{V}(B_{i-1})} d_{\vec{v}}\ell$ \Comment{Pick a random minimising vector}
      \State $B_i \gets B_{i-1} +\vec{v}$
    \EndFor
  \Return $B_{N_\mathrm{epochs}}$
  \end{algorithmic}
  \caption{Best-First Gradient Descent}
  \label{alg:best-first-gradient-search}
\end{algorithm}

\subsubsection{Deterministic Optimistic Optimisation on the Polydisc space} \label{subsec:doo}

The descent-based methods considered so far are inherently greedy: at each step, they refine only the most promising branch given by local information. While such approaches can be effective in well-behaved settings, they may overlook regions that initially appear suboptimal but contain better solutions at finer scales. This observation motivates methods that explicitly balance local refinement with global exploration.

Towards this end, we connect our setting to the framework of \emph{optimistic optimisation} introduced by \cite{munos2011optimistic}, where optimisation of a function $f\colon \mathcal{X} \to \mathbb{R}$ over a metric space $\mathcal{X}$ is recast as a hierarchical search problem. The space $\mathcal{X}$ is partitioned into a nested family of cells $X_{h, i}$, indexed by a level parameter $h \geq 0$ with $X_{0, 0} = \mathcal{X}$ and each subsequent level providing a finer decomposition. Optimisation then proceeds by selecting at each stage which cell to refine further, with the goal of isolating a global maximiser through progressively finer partitions,
$$
    \bigcap_h X_{h, i_h} = \{ x^* \}.
$$

In our setting, we consider a function $g \colon K^n \to \RR$ with a global minimum below the Gauss point $B_\text{Gauss} = B(0, 0)$, and restrict attention to the subspace $\mathcal{X} = B_\text{Gauss}$. Since the optimisation problem is formulated for maximisation problems, we instead work with $f = -g$.

To implement this framework, we require a suitable hierarchical partition of $\mathcal{X}$. A natural choice is given by the valuative radius: for each level $h \geq 0$, define
$$
\{ X_{h,i} \}
= \left\{ B_{\mathcal{H}}\big(a, (h,\dotsc,h)\big) \cap \mathcal{O}_K^n
\;\middle|\;
a \in \mathcal{O}_K^n \text{ and for all } i, r_i = h \right\}.
$$
These cells form a uniform partition of $\mathcal{X}$ at resolution level $h$, with refinement corresponding to increasing the common radius.

\begin{figure}[t]
  \centering
  \begin{tikzpicture}[
      x={(1.4cm, 0cm)},         
      y={(0.6cm, 0.4cm)},       
      z={(0cm, 1.8cm)},         
      every node/.style={inner sep=2pt, font=\scriptsize, fill=white}, 
      arr/.style={-{Stealth[length=4pt,width=3pt]}, semithick, black!50}
    ]

    \node (root) at (0, 0, 2) {$B_{\mathcal{H}}((0,0),(0,0))$};

    \node (n00_11) at (-1, -1, 0) {$B_{\mathcal{H}}((0,0),(1,1))$};
    \node (n01_11) at (-1,  1, 0) {$B_{\mathcal{H}}((0,1),(1,1))$};
    \node (n10_11) at ( 1, -1, 0) {$B_{\mathcal{H}}((1,0),(1,1))$};
    \node (n11_11) at ( 1,  1, 0) {$B_{\mathcal{H}}((1,1),(1,1))$};

    \begin{scope}[on background layer]
      \draw[arr] (root) -- (n00_11);
      \draw[arr] (root) -- (n01_11);
      \draw[arr] (root) -- (n10_11);
      \draw[arr] (root) -- (n11_11);
    \end{scope}

    \node[anchor=east, fill=none] at (-3.0, 0, 2) {$X_{0, 0}$};
    \node[anchor=east, fill=none] at (-3.0, 0, 0) {$X_{1, \bullet}$};
  \end{tikzpicture}%
  \caption{The first two levels of the cell decomposition for $n=2$ when $K = \mathbb{Q}_2$.}
\end{figure}
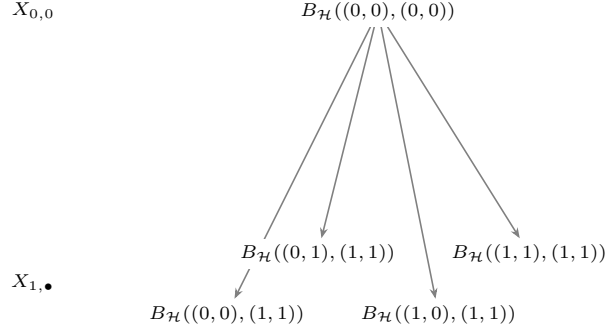

At level $h$, this decomposition consists of $|\mathfrak{K}|^{nh}$ cells. Each cell admits a canonical representative point, obtained by truncating the expansion of its centre:
$$
a = \sum_{j=0}^{h-1} a_j \pi^j,
\quad a_j \in \mathcal{S}.
$$
We denote this representative by $x_{h,i}$ for the corresponding cell $X_{h,i}$.

Different hierarchical decompositions induce different search trees and therefore different exploration--exploitation trade-offs. 
Consider, for instance, the alternative decomposition obtained by refining coordinates sequentially rather than simultaneously: take $X_{0, 0}' = \mathcal{X}$, and set the $X_{1, i}'$ to be the children of $X_{0, 0}'$ along the first coordinate, the $X_{2, j}'$ to be the children of the $X_{1, i}'$ along the second coordinate, and so on, looping back to the first coordinate once the $n$-coordinate has been reached. At level $h$, there are now $|\mathfrak{K}|^{h}$ nodes and canonical centres for these cells may be chosen. More generally, for any given $1 \le d \le n$, one can obtain a decomposition by shrinking $d$ radii at each step. Given an enumeration $S_1, \dotsc, S_N$, $N = \binom{n}{d}$ of the subset of size $d$ of $\{1, \dotsc, n\}$, one can construct the decomposition $X^{(d)}_{h, i}$
by taking $X^{(d)}_{0, 0} = \mathcal X$, setting $X^{(d)}_{1, i}$ to be the children of $X^{(d)}_{0, 0}$ along the coordinates in $S_1$, $X^{(d)}_{2, i}$ to be the children of the $X^{(d)}_{1, i}$ in the coordinates in $S_2$ and so on. Such choices lead to different tree structures which may be advantageous depending on the optimisation task.

The hierarchical partition constructed above serves as the backbone for optimistic optimisation methods. To apply these techniques, the cells must satisfy certain structural properties ensuring that the search procedure is well-defined and behaves consistently across levels. The following result collects these required properties.

\begin{proposition}
    \label{prop:doo-requirements}
    Let $g = \sum_i a_i \lvert f_i \rvert$ with $f_i\in \mathcal O_K[x_1,\dots,x_n]$ be an absolute polynomial sum whose polynomials have absolute value $\le 1$. Let $\mathcal X = \mathcal O_K^n.$
    Then the following properties hold:
    \begin{enumerate}
        \item[(i)] There exists a semimetric $d_L \colon \mathcal{X} \times \mathcal{X} \to \mathbb{R}^+$ such that for all $x, y \in \mathcal{X}$, we have 
        $$
            g(x) - g(y) \le d_L(x, y).
        $$
        In particular, $d_L(x, y) = L d(x, y)$, where $d$ is the $\ell^q$ metric on the polydisc space and 
        $
        L = \sum_i a_i 
        $, satisfies this inequality.
        \item[(ii)] There exists a decreasing function $\delta \colon \NN \to \RR_{>0}$ such that for all $h, i$ we have
        $$
            \sup_{x \in X_{h, i}} d_L(x_{h, i}, x) \le \delta(h)
        $$
        For instance, if $d$ is induced by the $\ell^p$ norm, then
        $
            \delta(h) = Lp^{-h} \lvert \lvert (1, \dotsc, 1) \rvert \rvert_{\ell^p}
        $
        satisfies this bound.
        \item[(iii)] Each cell $X_{h,i}$ contains a ball (with respect to $d_L$) of radius proportional to its diameter: more precisely, for every $h,i$ there exists $\nu > 0$ such that $X_{h,i}$ contains a $d_L$-ball of radius $\nu\,\delta(h)$.
    \end{enumerate}  
\end{proposition}

The assumption that the coefficients of the $f_i$ lie in $\mathcal{O}_K$ is not restrictive. Indeed, any absolute polynomial sum can be reduced to this case by rescaling the coefficients $a_i$, which leaves the optimisation problem unchanged up to an overall multiplicative factor.

\begin{proof}
    We begin with (i). Take
    $$
    d_L(x, y) = L d(x, y),
    $$
    which is a semimetric. To prove the inequality, we can reduce to the case where only one polynomial appears in the sum, i.e., $g = a \lvert F \rvert$ where $a > 0$. 
    We write
    $
    F(x) = \sum_\alpha a_\alpha x^\alpha
    $
    so that we have: 
    $$
    g(x) - g(y) \le a \lvert F (x) - F(y) \rvert \le a \lvert x - y \rvert,
    $$
    where the last inequality uses the fact that polynomials are $1$-Lipschitz on $\mathcal{O}_K^n$.
    
    For (ii), we need to bound $\diam(X_{i, h})$. By definition of the $\ell^q$ metric on $\mathcal{B}^n$, we have
    $$
    \diam(B_{\mathcal{B}}(a, r)) = L \lvert \lvert r \rvert \rvert_{\ell^q}.
    $$
    Since $X_{i, h}$ has valuative radius $(h, \dotsc, h)$, we have
    $$
        \diam(X_{i, h}) = L \lvert \lvert (\pi^{h}, \dotsc, \pi^{h}) \rvert \rvert_{\ell^p} = L\lvert \pi \rvert^{h} \lvert \lvert (1, \dotsc, 1) \rvert \rvert_{\ell^p} = \delta_h.
    $$

    Finally, (iii) follows from the fact that $X_{i, h}$ is always a polydisc whose absolute radii are all positive.
\end{proof}

An analogous statement holds for the alternative cell decomposition $(X'_{h,i})$ described above, with $\delta$ replaced by $\delta'_h = L\lvert \pi \rvert^{\lfloor h / n \rfloor } \lvert \lvert (1, \dotsc, 1) \rvert \rvert_{\ell^p}$. More generally, one can check that for the decomposition $X^{(d)}_{h, i}$ whose construction was sketched above, one can take 
$$
\delta^{(d)}_h = \delta_{\binom{n-1}{d-1}\lfloor ({n} /{\binom{n}{d}}\rfloor}
$$

The previous proposition provides exactly the components required to implement the optimistic optimisation framework on $\mathcal{B}^n$. In particular, each cell $X_{h,i}$ is equipped with an upper confidence bound 
$$
b_{h,i} = f(x_{h,i}) + \delta_h,
$$
which balances the observed value of $f$ at the centre with the uncertainty induced by the size of the cell. 

This leads to a direct adaptation of the deterministic optimistic optimisation (DOO) algorithm of \cite{munos2011optimistic}, adapted to the polydisc setting, which we summarise in Algorithm~\ref{alg:doo}.

\begin{algorithm}
  \begin{algorithmic}[1]
    \Require $\ell \colon \mathcal{B}^n \to \RR$, $N_\mathrm{epochs} \ge 1$, $x_0 \in \mathcal{B}^n$
    \State Set $\mathcal{L} = \mathcal{T} = \{x_0\}$
    \For{$i = 1, \dotsc, N_\mathrm{epochs}$} 
      \State Compute the set $\mathcal{L}$ of leaves of the tree $\mathcal{T}$
      \State Find the node $(h, i) \in \mathcal{L}$ that maximises $b_{h, i}$
      \State Add to $\mathcal{T}$ the children of $(h, i)$ \Comment{Expand the node $(h, i)$.}
    \EndFor
  \Return $x_{h, i}$ where $(h, i) = \argmax_{(h, i) \in \mathcal{T}} f(x_{h, i})$.
  \end{algorithmic}
  \caption{Polydisc Space DOO}
  \label{alg:doo}
\end{algorithm}

\begin{theorem}
    \label{thm:doo}
    Suppose that $f = -g$ are as in \cref{prop:doo-requirements}, that $x_*$ is a global minimum of $g$ and let $x_0, x_1, \dotsc$ be the sequence of iterates from running \cref{alg:doo} indefinitely. Then $(x_i)$ admits a subsequence that converges to a global minimum of $g$. In fact, any convergent subsequence of $(x_i)$ converges to a global minimum of $g$.
\end{theorem}

\begin{proof}
    This result follows from \cite[Theorem 1]{munos2011optimistic}. Let $x_* \in \mathcal X$ be a global minimum of $g$. Then we can define the \emph{loss}
    $$
        r_i = g(x_i) - g(x_*)
    $$
    and set for any $\varepsilon > 0$
    $$
        \mathcal X_\varepsilon = \{ x \in \mathcal X \mid g(x_*) + \varepsilon \ge g(x) \}.
    $$
    Recall that we set $d_L(x, y) = Ld(x, y)$ where $d$ is the $\ell^\infty$ metric on $K^n$ and $L = \sum_i a_i$.
    We will see that there exists a minimal integer $d_\mathrm{opt}$ such that the maximal number of disjoint balls of radius $\varepsilon$ (with respect to the $\ell$ metric defined above) is at most $C\varepsilon^{-d_\mathrm{opt}}$ for some constant $C > 0$.
    Assuming for now the existence of such a $d_\mathrm{opt}$, let $h_i$ be the smallest natural number such that
    $$
        C \sum_{j=0}^{h_i} \delta(j)^{-d_\mathrm{opt}} \ge i.
    $$
    where $\delta$ is as defined in \cref{prop:doo-requirements}.
    By \cite[Theorem 1]{munos2011optimistic}, we have $r_i \le \delta_{h_i} \to 0$ as $i \to \infty$.
    Thus, if $(y_j) \to y$ is a convergent subsequence of $(x_i)$ then 
    $$f(x_*) - f(y) = \lim_{i \to \infty} r_i = 0.$$

    We now give a proof of the existence of $d_\mathrm{opt}$ and in fact that it is less than $n$. Let $\mu$ be the Haar measure on $K^n$ (this exists because $K$ is locally compact), normalised so that $\mu(\mathcal X) = 1$. Then, the maximum number of disjoint $\varepsilon$-balls (with respect to the $d_L$ metric from \cref{prop:doo-requirements}) contained in $\mathcal X_\varepsilon$ is bounded by the maximum number of disjoint $\varepsilon$-balls in $\mathcal X$. This in turn is bounded above by
    $$
        \frac{\mu(\mathcal X)}{\mu(B_\mathcal{B}(0, \varepsilon / L))} = \frac{L^n}{\varepsilon^n},
    $$
    which finishes the proof.
\end{proof}

\subsubsection{Monte-Carlo Tree Search on the Polydisc Space} \label{subsec:mcts}

\begin{figure}[t]
        \centering
        \begin{subfigure}{.24\linewidth}
        \includegraphics[width=\linewidth]{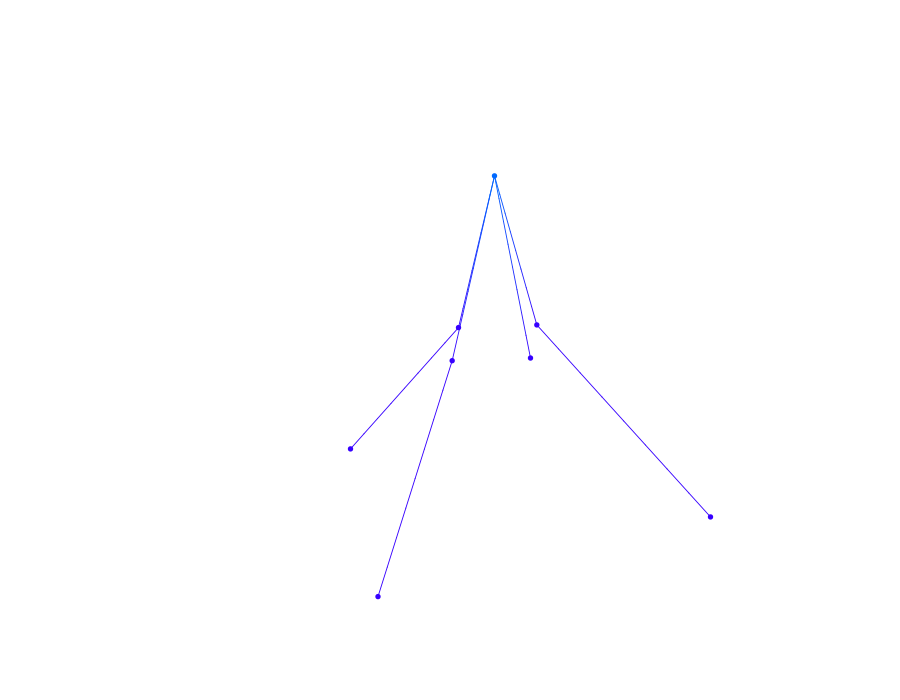}
        \caption{$N = 4$}
        \end{subfigure}
        \begin{subfigure}{.24\linewidth}
        \includegraphics[width=\linewidth]{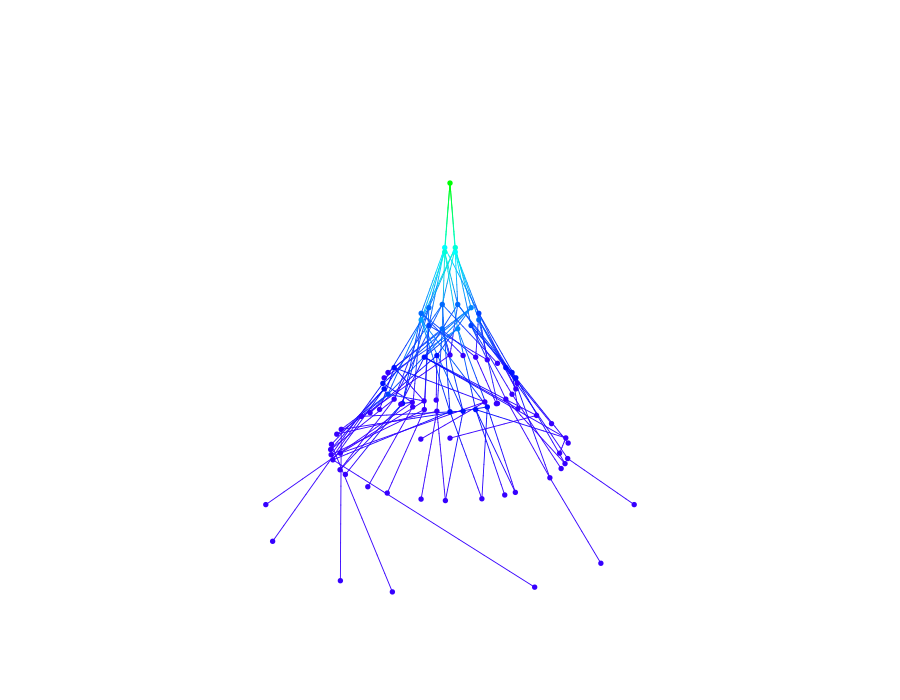}
        \caption{$N = 50$}
        \end{subfigure}
        \begin{subfigure}{.24\linewidth}
        \includegraphics[width=\linewidth]{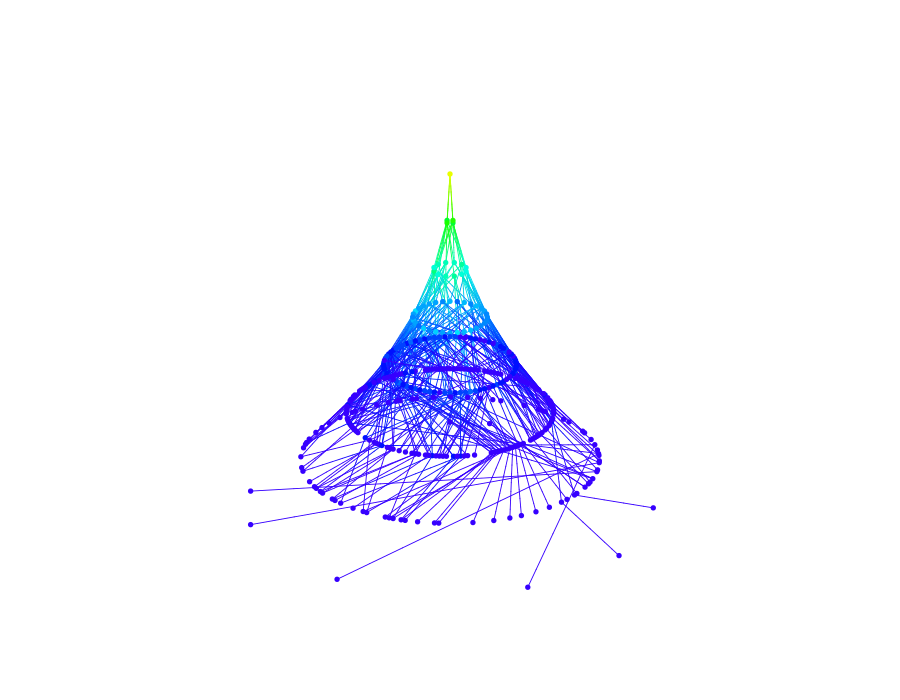}
        \caption{$N = 200$}
        \end{subfigure}
        \begin{subfigure}{.24\linewidth}
        \includegraphics[width=\linewidth]{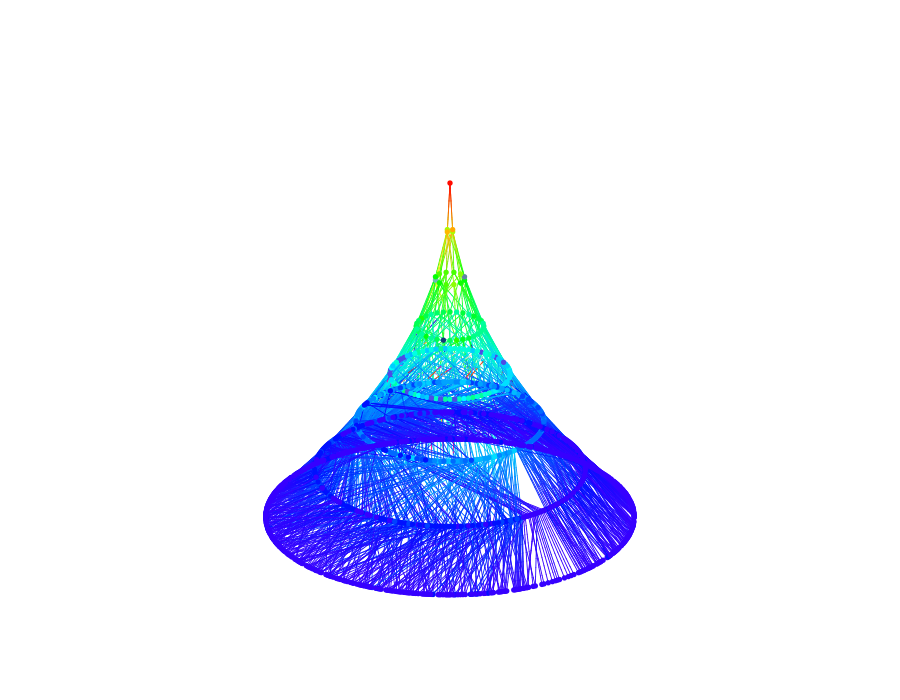}
        \caption{$N = 1000$}
        \end{subfigure}
        \caption{DAG Monte-Carlo Tree Search over $\mathcal{B}^2(\mathbb{Q}_2)$ after $N$ simulations. Brighter colours correspond to larger visit counts. The depth of the final tree is truncated.}\label{fig:optimisationHeatMap}
    \end{figure}

The deterministic optimistic optimisation framework described above relies on explicit geometric control of the objective through upper bounds on each cell. Although this leads to strong guarantees, it can be conservative in practice, as the bounds are designed to hold uniformly over the entire space.

An alternative approach is to guide the search using empirical information gathered along the trajectory of the algorithm, which is the premise of \emph{Monte-Carlo tree search} (MCTS). MCTS was developed in the context of two-player games such as Go where the branching factor makes exhaustive search infeasible \cite{coulom2006efficient,kocsis2006bandit,browne2012survey}. Rather than assigning fixed optimistic bounds to all nodes in advance, it balances exploration and exploitation dynamically using statistics accumulated during the search.

From the perspective of optimisation, this amounts to treating the hierarchical decomposition of $\mathcal{B}^n$ as a sequential decision process, where each node represents a region of the space and each descent step corresponds to selecting a promising refinement. The key idea is to allocate computational effort adaptively, focusing on regions that appear favourable based on past evaluations while concurrently ensuring sufficient exploration of the search space.

\paragraph{Choice of Tree Structure.} As in the DOO setting, MCTS operates on an incrementally constructed search structure $\mathcal{T}$ whose nodes are rational polydiscs and whose edges are determined by a children function $\mathcal{C}$. Several natural choices are available, such as the cell decompositions $\{X_{h,i}\}$ and $\{X'_{h,i}\}$, as well as the degree-$d$ refinement $\mathcal{C}^d$. Each choice induces a different way of exploring the search space, and our formulation of this approach is given abstractly in terms of a generic $\mathcal{C}$.

When $\mathcal{C}$ defines a true tree (for instance in dimension $n=1$ or when using the decompositions $\{X_{h,i}\}$ or $\{X'_{h,i}\}$ in arbitrary dimension), standard MCTS applies directly. In higher dimensions a polydisc may admit multiple parents giving rise to a directed acyclic graph; this case is discussed explicitly further on in this section.

Each node $x \in \mathcal{T}$ is assigned a visit count $N(x) \in \mathbb{N}$ and a cumulative reward $W(x) \in \mathbb{R}$; we write $Q(x) = W(x)/N(x)$ for the empirical mean value. As before, we work with the maximisation objective $-\ell$.

In modern reinforcement-learning systems such as AlphaZero, MCTS is often combined with learned policy and value functions that guide the search and evaluate leaf positions \cite{silver2018general}. While our use of MCTS includes a value function to guide the search, in our setting the value function is not a learned predictor, but rather a deterministic objective $-\ell$. For example, in the case where $\ell$ is an absolute polynomial sum, the value of $\ell(y)$ corresponds roughly to a bound on the worst possible outcome in the portion of the tree below $y$. In this respect, it is closer in spirit to MCTS variants that incorporate fixed evaluation or heuristic information into the search process \cite{lanctot2014monte,lorentz2016using}.


\begin{definition}
    Let $x \in \mathcal{T}$ be a node with parent $p$, the \emph{upper confidence bound (UCB) score} of $x$ is
\begin{equation}
\label{eq:UCB}
        \mathrm{UCB}(x) = Q(x) + c \cdot \sqrt{\frac{\ln N(p)}{N(x)}}
\end{equation}
    where $c > 0$ is a fixed exploration parameter. Nodes that have not yet been visited are assigned score $+\infty$.
\end{definition}

The first term of \eqref{eq:UCB} promotes exploitation of high-value nodes, while the second encourages exploration of less-visited regions. The UCB score provides a principled rule for navigating the search tree, balancing exploitation of regions with low objective value and exploration of less-sampled branches. In contrast to deterministic optimistic optimisation, which relies on global bounds that are required to hold uniformly over the space, MCTS adapts to the values encountered along the search trajectory, allocating computational effort based on observed behaviour. We now describe the resulting search procedure.

\begin{algorithm}[H]
  \begin{algorithmic}[1]
    \Require $\ell \colon \mathcal{B}^n \to \RR$, $N_\mathrm{sim} \ge 1$, $x_0 \in \mathcal{B}^n$, children function $\mathcal{C}$, $c > 0$
    \State $\mathcal{T} \gets \{x_0\}$; $N(x_0) \gets 0$; $W(x_0) \gets 0$
    \For{$i = 1, \dotsc, N_\mathrm{sim}$}
      \State $x \gets x_0$; $P \gets (x_0)$ \Comment{Selection: traverse $\mathcal{T}$ by UCB}
      \While{$\mathcal{C}(x) \subseteq \mathcal{T}$ and $\mathcal{C}(x) \ne \varnothing$}
        \State $x \gets \argmax_{y \in \mathcal{C}(x)} \mathrm{UCB}(y)$
        \State append $x$ to $P$
      \EndWhile
      \State Add $\mathcal{C}(x)$ to $\mathcal{T}$, set $N(y) \gets 0$, $W(y) \gets 0$ for $y \in \mathcal{C}(x)$ \Comment{Expansion}
      \State $y \sim \mathcal{C}(x)$
      \State append $y$ to $P$
      \State $v \gets -\ell(y)$ \Comment{Evaluation}
      \For{each $z \in P$} \Comment{Backpropagation}
        \State $N(z) \gets N(z) + 1$; $W(z) \gets W(z) + v$
      \EndFor
    \EndFor
    \Return $\argmax_{y \in \mathcal{C}(x_0)} Q(y)$ \Comment{Select child node with the largest mean value}
  \end{algorithmic}
  \caption{Update Rule for MCTS on Polydisc Space}
  \label{alg:mcts-update}
\end{algorithm}

\paragraph{DAG--MCTS.} When the chosen children function produces a directed acyclic graph rather than a tree, distinct traversal paths may lead to the same polydisc, similarly to the way several move sequences may lead to the same state in game-tree search.
Naively identifying such nodes may artificially inflate visit counts and bias the exploration--exploitation balance.  Transposition-aware variants of MCTS and UCT-style methods for rooted directed acyclic graphs have been studied in \cite{childs2008transpositions,saffidine2012ucd}. To address this, \emph{directed acyclic graph MCTS (DAG--MCTS)} maintains statistics at the level of root-to-node paths rather than individual nodes.

Specifically, for a path $P = (x_0, x_1, \ldots, x_k)$, we record its visit count $N(P)$ and cumulative value $W(P)$ and write $Q(P) = W(P)/N(P)$ for the corresponding empirical mean. If $y$ is a child of $x_k$, we denote by $P \cdot y\coloneqq (x_0, x_1, \ldots, x_k, y)$ the path obtained by appending $y$ to $P$. A \emph{transposition table} $H$ caches evaluations of $f(x) = -\ell(x)$, ensuring that each node is evaluated at most once, irrespective of how many paths reach it. The exploration strategy is then defined in terms of a path-dependent analogue of the UCB score. The approach taken here is similar to the UCT0 strategy outlined in \cite{childs2008transpositions}.

\begin{definition}
    Let $P = (x_0, \dotsc, x_k)$ be a path and $y$ a child of $x_k$. The \emph{DAG upper confidence bound (DUCB) score} of $(P, y)$ is
    $$
        \mathrm{DUCB}(P, y) = Q(P \cdot y) + c\sqrt{\ln N(P)\,/\,N(P \cdot y)}
    $$
    where $c > 0$ is a fixed exploration parameter.
\end{definition}

This leads to a natural extension of MCTS to the DAG setting, in which selection is performed using the DUCB score while evaluation is shared across nodes via the transposition table.

\begin{algorithm}[H]
  \begin{algorithmic}[1]
    \Require $\ell \colon \mathcal{B}^n \to \RR$, $N_\mathrm{sim} \ge 1$, $x_0 \in \mathcal{B}^n$, children function $\mathcal{C}$, $c > 0$
    \State $\mathcal{T} \gets \{x_0\}$; $H \gets \varnothing$; $N(P) \gets 0$, $W(P) \gets 0$ for each path $P$ in $\mathcal{T}$
    \For{$i = 1, \dotsc, N_\mathrm{sim}$}
      \State $P \gets (x_0)$; $x \gets x_0$ \Comment{Selection: traverse using path-based UCB}
      \While{$\mathcal{C}(x) \subseteq \mathcal{T}$ and $\mathcal{C}(x) \ne \varnothing$}
        \State $x \gets \argmax_{y \in \mathcal{C}(x)} \mathrm{DUCB}(P, y)$; $P \gets P \cdot x$
      \EndWhile
      \State Add $\mathcal{C}(x)$ to $\mathcal{T}$; $y \sim \mathcal{C}(x)$; $P \gets P \cdot y$ \Comment{Expansion}
      \If{$y \in H$} \Comment{Evaluation}
        \State $v \gets H(y)$
      \Else
        \State $v \gets -\ell(y)$; $H(y) \gets v$
      \EndIf
      \For{each prefix $P'$ of $P$} \Comment{Backpropagation}
        \State $N(P') \gets N(P') + 1$; $W(P') \gets W(P') + v$
      \EndFor
    \EndFor
    \Return $\argmax_{y \in \mathcal{C}(x_0)} Q\bigl((x_0) \cdot y\bigr)$\Comment{Select child yielding the largest mean value}
  \end{algorithmic} 
  \caption{Update rule for DAG-MCTS on the Polydisc Space}
  \label{alg:dag-mcts-update}
\end{algorithm}

In the polydisc setting, these procedures serve as update rules within a broader descent scheme. Starting from an initial polydisc $x_0$, each iteration applies Algorithm~\ref{alg:mcts-update} or Algorithm~\ref{alg:dag-mcts-update} to select a promising child and the process is repeated for a fixed number of steps. MCTS thus provides a flexible mechanism for navigating the hierarchical structure of $\mathcal{B}^n$ by adaptively concentrating the search in regions that exhibit favourable behaviour of the objective.

\section{Implementation and Applications:\\ A Julia Library for Non-Archimedean Machine Learning}
\label{sec:applications}
The mathematical framework developed throughout this paper is accompanied by a software implementation, \textsc{NonArchimedeanMachineLearning.jl} (NAML), a Julia library for optimisation, learning, and geometric computation on non-Archimedean polydisc spaces. Beyond serving as a proof of concept, the library provides a practical environment for experimenting with the structures introduced in this paper, including polydisc spaces, absolute and valuation polynomials, tangent directions, embeddings, and the optimisation algorithms in Sections \ref{subsec:best-first-search} to \ref{subsec:mcts}.
Our implementation plays a central role in the present work, allowing the theoretical constructions developed in the previous sections to be tested experimentally. It provides a unified interface for optimisation over polydisc spaces and forms a foundation for future work on non-Archimedean machine learning and optimisation.

In this section, we first describe the design and capabilities of NAML. We then present a collection of computational experiments illustrating both the expressiveness of the framework and the practical behaviour of the optimisation algorithms introduced in this paper. A discussion of our experimental results can be found in \cref{sec:discussion}.

\subsection{The \textsc{NonArchimedeanMachineLearning.jl} library}


\textsc{NonArchimedeanMachineLearning.jl} (NAML) is an open-source Julia library for machine learning and optimisation over non-Archimedean valued fields. It implements the mathematical framework developed throughout this paper and provides a unified interface for working with polydisc spaces, real-valued functions defined on them, and optimisation algorithms adapted to their geometry.

At the core of the library is a compositional algebra of functions on polydisc spaces. NAML supports multivariate polynomials, rational functions, absolute polynomials, valuation polynomials, and combinations of such objects through addition, multiplication, and composition with real-valued functions. The library further provides automatic differentiation capabilities for computing directional derivatives on polydisc spaces.

On top of this functional framework, NAML implements all optimisation algorithms introduced previously in Sections \ref{subsec:best-first-search} to \ref{subsec:mcts}, including best-first value descent, best-first gradient descent, Monte-Carlo tree search (MCTS), DAG--MCTS, and deterministic optimistic optimisation (DOO). The package also includes commonly used loss functions, utilities for computing Fr\'echet means, and tools for $p$-adic linear regression.

To facilitate experimentation and visualisation, the library provides routines for plotting loss landscapes, geodesics, convex hulls, and optimisation trajectories. The computational experiments reported in this section were all implemented using NAML and are fully reproducible through a companion repository containing experiment scripts, datasets, and analysis code.
The codebases associated with this work are available at:

\begin{center}
    \begin{tabular}{|l|l|}
        \hline
        Package & \href{https://github.com/Paul-Lez/NonArchimedeanMachineLearning.jl}{\texttt{github.com/Paul-Lez/NonArchimedeanMachineLearning.jl}} \\
        \hline
        Experiments & \href{https://github.com/Paul-Lez/Non-Archimedean-Polydisc-Spaces-and-Applications-to-Optimisation}{\texttt{github.com/Paul-Lez/Non-Archimedean-Polydisc-Spaces}} \\
        & \hspace{22mm}\href{https://github.com/Paul-Lez/Non-Archimedean-Polydisc-Spaces-and-Applications-to-Optimisation}{\texttt{-and-Applications-to-Optimisation}}\\
        \hline
    \end{tabular}
\end{center}

Taken together, these implementations provide, to our knowledge, the first software framework specifically designed for optimisation and machine learning on non-Archimedean polydisc spaces. The experiments below illustrate several aspects of the framework, ranging from polynomial root-finding and absolute polynomial sum minimisation to interpolation and function learning tasks.

\subsection{Computational Examples}
In order to make the methods developed above more concrete, we also run various illustrative experiments, which are intended to play a proof-of-concept role. 
Each subsection below briefly describes one such experiment and provides a summary of the outcome, with more detailed tables included in the appendix.

\subsubsection{Solving Polynomial Equations}

\begin{figure}[t]
  \centering
  \begin{subfigure}{0.45\textwidth}
    \includegraphics[width=\linewidth]{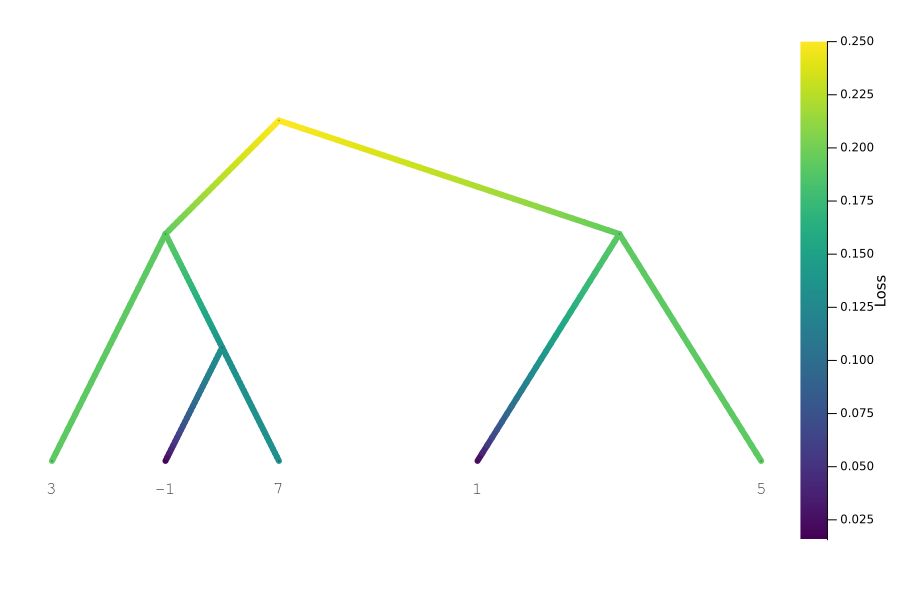}\vspace{-5mm}
    \caption{loss landscape on a subset of $\mathcal{B}^1$.}
  \end{subfigure}\hfill
  \begin{subfigure}{0.45\textwidth}
    \includegraphics[width=\linewidth]{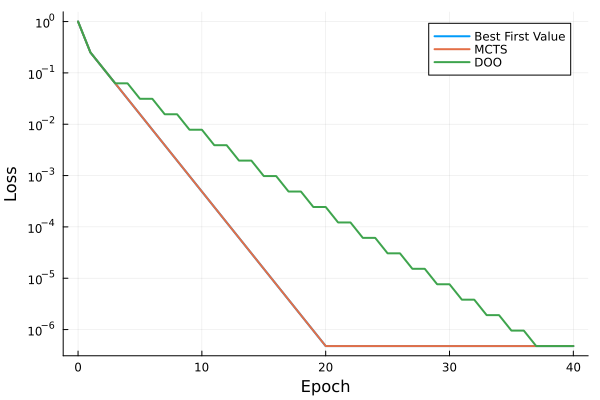}
    \caption{loss curves for best-first value, MCTS, DOO.}
  \end{subfigure}
  \caption{Loss landscape and curves for $|f|=\lvert x^2 - 1 \rvert$ over $\mathbb{Q}_2$.}
  \label{fig:x2_minus_1_loss_landscape_and_curves}
\end{figure}

\begin{example}
    Let $p=2$ and $f = X^2 - 1$. The loss landscape of $\lvert X^2 - 1 \rvert$ on a subset of the disc space $\mathcal{B}^1$ is plotted in  \cref{fig:x2_minus_1_loss_landscape_and_curves} (a).
    Running some of the optimisers described above starting from the Gauss point, we obtain the results presented in \cref{tab:x2_minus_1_results}.
    \begin{table}
        \centering
        \begin{tabular}{lcrrc}
          \toprule
          Optimizer & Center $c$ & Radius $r$ & Loss $f\!\left(B(c,\,2^{-r})\right)$ & Time\,(s) \\
          \midrule
          Greedy Descent   & $1$ & $20$ & $4.77 \times 10^{-7}$ & $0.077$ \\
          MCTS             & $1$ & $20$ & $4.77 \times 10^{-7}$ & $0.027$ \\
          DOO              & $1$ & $20$ & $1.53 \times 10^{-5}$ & $0.018$ \\
          \bottomrule
        \end{tabular}
        \caption{Optimisation results for $|f|=\lvert x^2 - 1 \rvert$ over $\mathbb{Q}_2$.}
    \label{tab:x2_minus_1_results}
    \end{table}
    The loss curves plotted over time yield those shown in \cref{fig:x2_minus_1_loss_landscape_and_curves} (b).

    As shown in \cref{tab:x2_minus_1_results} and \cref{fig:x2_minus_1_loss_landscape_and_curves} (b), each optimisation algorithm converges to a global minimum of the function $\lvert f \rvert$, namely one of the roots of $f$.
    One notable difference here is the \emph{loss dynamics}: the convergence of DOO is slower than the other two algorithms. 
    This is due to the fact that throughout the process, DOO expands the search tree towards both roots, whilst the other two algorithms ``commit'' to a specific root. This is also visible at the level of the search trees of the algorithms which are plotted in \cref{fig:search-trees-1d}. 
    \begin{figure}[t]
        \centering
        \begin{subfigure}{0.45\textwidth}
            \includegraphics[width=.8\textwidth]{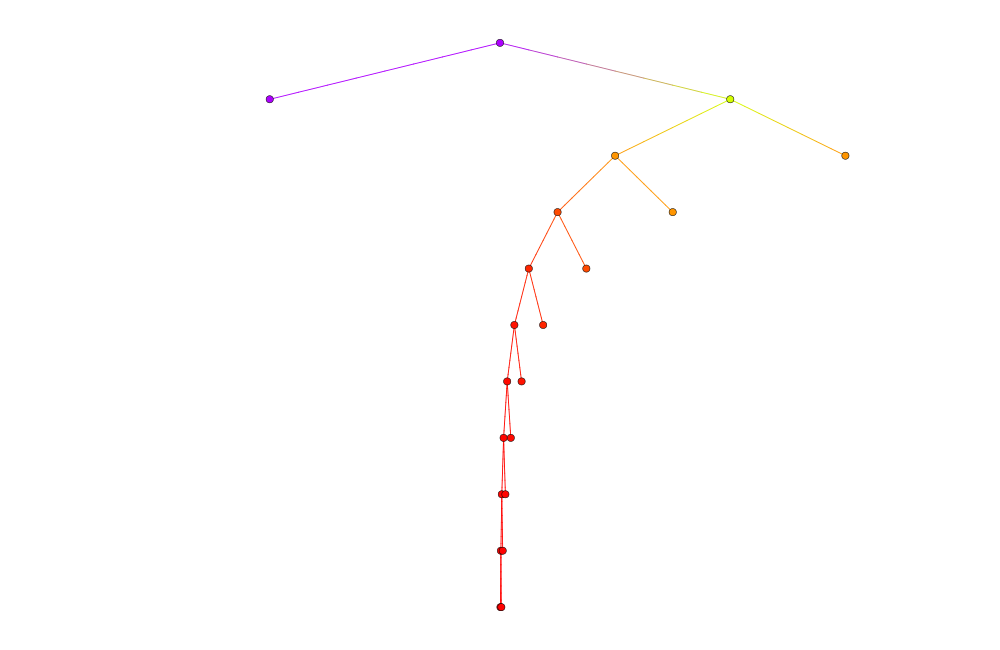}
            \caption{Best first search, full search tree. Brighter colours mean lower loss.}
        \end{subfigure}\hfill
        \begin{subfigure}{0.45\textwidth}
            \includegraphics[width=.8\textwidth]{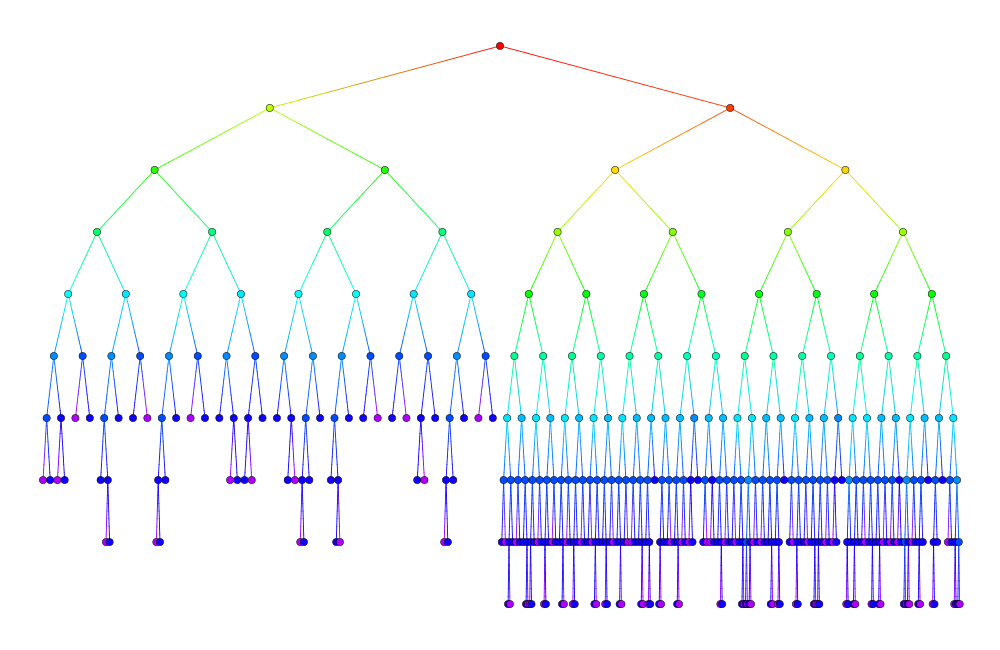}
            \caption{MCTS, \emph{search tree for the first step}. Brighter colours indicate nodes that have been more visited.}
        \end{subfigure}\\
        \begin{subfigure}{0.45\textwidth}
        \includegraphics[width=.8\textwidth]{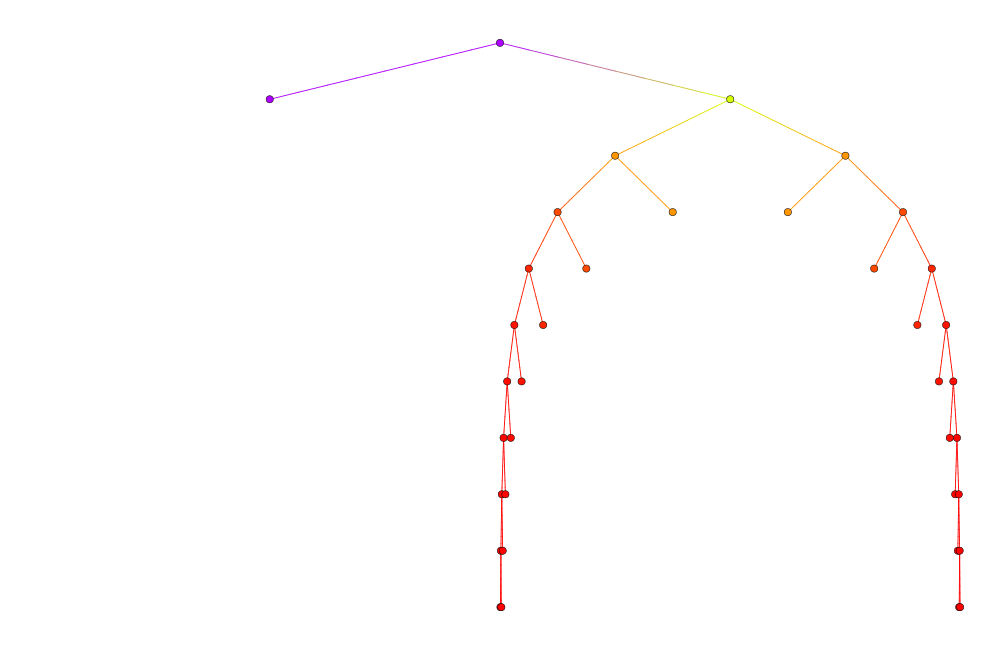}
            \caption{DOO, full search tree. Brighter colours mean lower loss.}
        \end{subfigure}\hfill
        \begin{subfigure}{0.45\textwidth}
    \includegraphics[width=.8\textwidth]{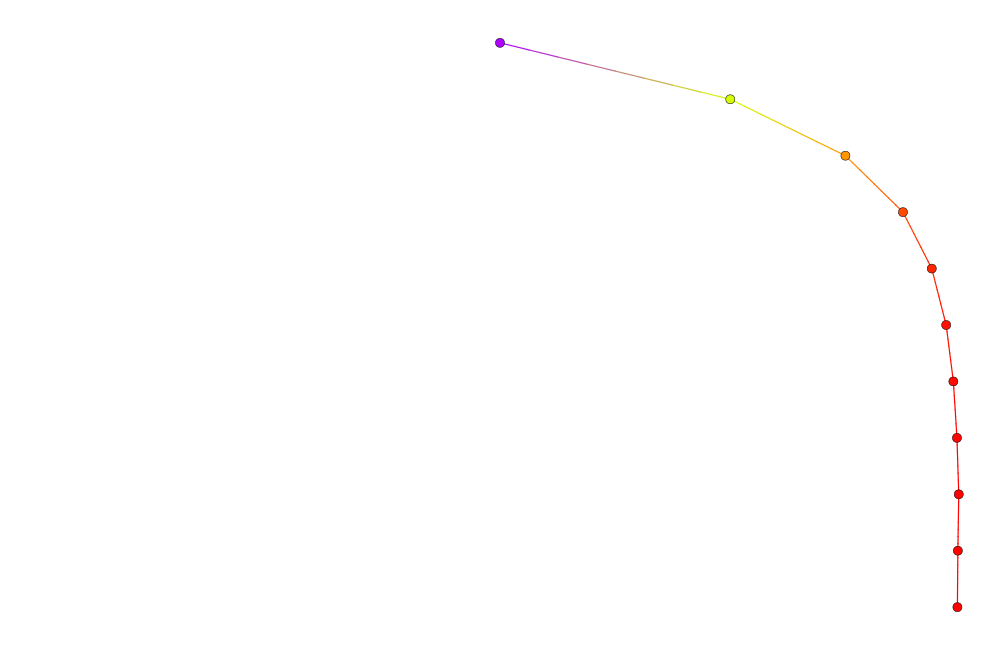}
            \caption{MCTS sequence of iterates. Brighter colours mean lower loss.}
        \end{subfigure}
    \caption{The search trees for the minimisation of $\lvert X ^ 2 - 1 \rvert$.}
    \label{fig:search-trees-1d}
    \end{figure}
\end{example}

\paragraph{Random Polynomials.} We construct a larger scale experiment by generating random polynomials $f$ with degrees $1, 2, 3$ and $1, 2, 3$ variables and minimising the function $z \mapsto |f(z)|$. To ensure that $f$ has a root below the Gauss point, we generate random $p$-adic numbers $\alpha_i$ and write $f$ as a product $f(z)$ of degree 1 multivariate polynomials with coefficients the $\alpha_i$. We pick the Gauss point as the starting point in the optimisation process.
We vary the prime $p \in \{2, 3, 5\}$ and test each configuration across many random instances. We rank several optimisers: a random baseline, best-first value search, best first gradient search, MCTS, DAG-MCTS, and DOO. Summary results are shown in \cref{fig:polynomial_solving_comparison}; more details are given in Section \ref{subsec:experimental-results}.

\begin{figure}[t]
    \centering
    \begin{subfigure}{0.48\textwidth}
        \centering
        \includegraphics[width=\linewidth]{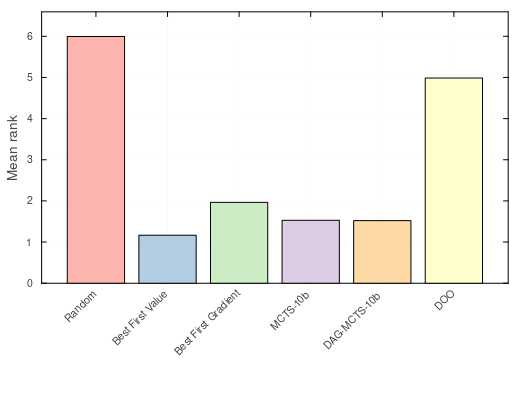}\vspace{-6mm}
        \caption{Average rank}
    \end{subfigure}
    \hfill
    \begin{subfigure}{0.48\textwidth}
        \centering
        \includegraphics[width=\linewidth]{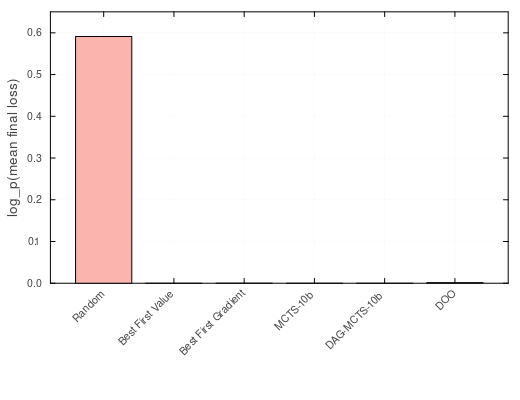}\vspace{-6mm}
        \caption{Mean final loss}
    \end{subfigure}
    \caption{Average ranks and mean final loss of the optimisers across several polynomial minimisation tasks}
    \label{fig:polynomial_solving_comparison}
\end{figure}

\subsubsection{Minimising Absolute Polynomial Sums}

\begin{example}
    Consider the absolute polynomial sum
    $$
      f\colon z \mapsto \sum_{x \in X} \lvert (z - x)(z - 2x)(z - 4x) \rvert
    $$
    where $X = \{1, 2, 4\}$.
    The loss landscape of $f$ on a subset of the disc space $\mathcal{B}^1$ is plotted in Figure \ref{fig:cubic_sum_landscape_and_curves} (a).
    \begin{figure}[t]
      \centering
      \begin{subfigure}{0.45\textwidth}
        \centering
        \includegraphics[width=\linewidth]{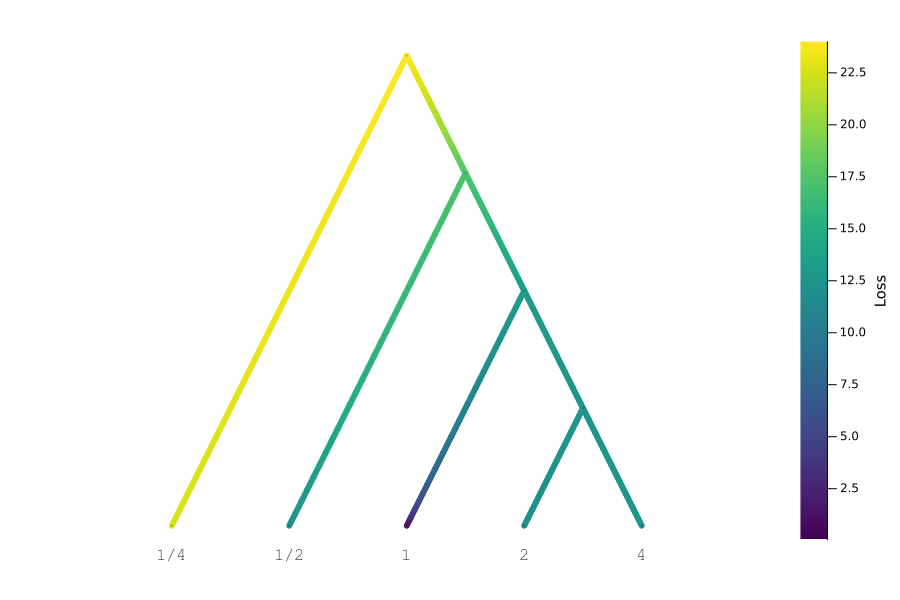}\vspace{-3mm}
        \caption{Loss landscape on a subset of $\mathcal{B}^1$}
      \end{subfigure}
      \hfill
      \begin{subfigure}{0.45\textwidth}
        \centering
        \includegraphics[width=\textwidth]{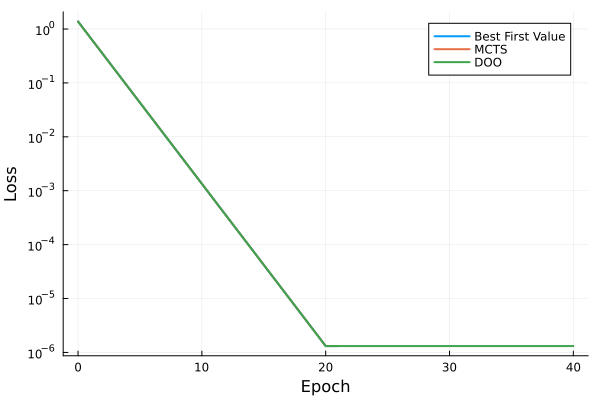}
        \caption{Loss curves for best-first-value, MCTS, DOO.}
      \end{subfigure}
      \caption{Loss landscape and curves of $\sum_{x \in X} \lvert (z - x)(z - 2x)(z - 4x) \rvert$ over $\QQ_2$.}
      \label{fig:cubic_sum_landscape_and_curves}
    \end{figure}

    Running the optimisers starting at the point $B_{\mathcal{H}}(0, 4)$ (the smallest disc containing all the roots of the polynomials appearing in the sum) yields the results shown in \cref{tab:cubic_sum_optimisation}.
    \begin{table}[t]
        \centering
        \begin{tabular}{lcrrc}
          \toprule
          Optimizer & Center $c$ & Radius $r$ & Loss $f\!\left(B(c,\,2^{-r})\right)$ & Time\,(s) \\
          \midrule
          Greedy Descent   & $1$ & $20$ & $1.31 \times 10^{-6}$ & $0.077$ \\
          MCTS             & $-1$ & $20$ & $1.31 \times 10^{-6}$ & $0.048$ \\
          DOO              & $1$ & $20$ & $1.31 \times 10^{-6}$ & $0.025$ \\
          \bottomrule
        \end{tabular}
        \caption{Results of optimisation algorithms for $f$}
        \label{tab:cubic_sum_optimisation}
    \end{table}
    The loss curves plotted over time are displayed in \cref{fig:cubic_sum_landscape_and_curves} (b).
    As shown in \cref{tab:cubic_sum_optimisation,fig:cubic_sum_landscape_and_curves} (b), all three algorithms yield the same result, and have identical trajectories when it comes to the observed loss at each step.
\end{example}

To demonstrate the practical effectiveness of our optimisation algorithms on polydisc spaces, we consider the following benchmark problem: given polynomials $f_1, \dotsc, f_m \in \QQ_p[x_1, \dotsc, x_n]$ with random coefficients, minimise the loss function
$$
\ell(x) = \sum_{i=1}^m |f_i(x)|
$$
over the $n$-dimensional polydisc space $\mathcal{B}^n$.

For each experiment, we generate multiple random problem instances by sampling polynomial coefficients uniformly from $p$-adic numbers with valuations between 0 and 8, and initialize all optimisers at the Gauss point $B_{\mathcal{H}}(1, 0)$. We vary the prime $p \in \{2, 3, 5\}$, the problem dimension $n \in \{1, 2, 5\}$, the number of polynomials in the sum $m \in \{2, 3, 5\}$, the polynomial degree (linear or quadratic), and the optimisation degree parameter (which controls the granularity of the search; see Section \ref{sec:optimisation}).
We compare several optimisers: a random baseline, best-first value search (branching over degree 1, and degree 2 children), best first gradient search, MCTS, DAG-MCTS, and DOO.

\begin{figure}[t]
    \centering
    \begin{subfigure}{0.48\textwidth}
        \centering
        \includegraphics[width=\linewidth]{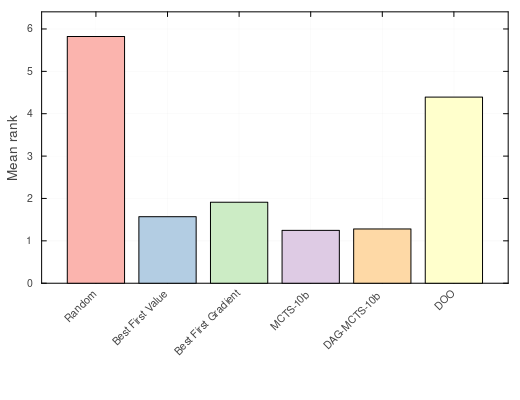}\vspace{-6mm}
        \caption{Average rank}
    \end{subfigure}
    \hfill
    \begin{subfigure}{0.48\textwidth}
        \centering
        \includegraphics[width=\linewidth]{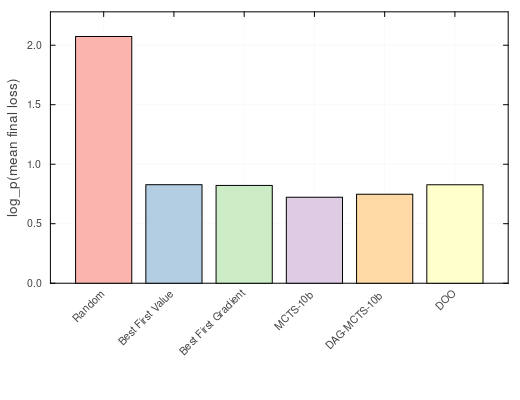}\vspace{-6mm}
        \caption{Mean final loss}
    \end{subfigure}
    \caption{Average ranks and mean final loss of the optimisers across several absolute polynomial sum minimisation tasks}
    \label{fig:absolute_sum_comparison}
\end{figure}

The summary results are shown in \cref{fig:absolute_sum_comparison}; more details are given in Section \ref{subsec:experimental-results}.

\subsubsection{Polynomial Interpolation}
Suppose we are given data points
$$
(x_0, y_0), \dotsc, (x_n, y_n) \in K \times K
$$
Consider the problem of finding the unique polynomial $f =  \sum_{i = 0}^n a_ix^i$ of degree at most $n$ such that for each $i$, we have $f(x_i) = y_i$. We can consider the loss function
$$
\ell(a_0, \dotsc, a_n) = \sum_i \lvert f(x_i) - y_i \rvert ^ 2 = \lvert \lvert X a - y \rvert \rvert ^ 2
$$
where
$$
X = 
\begin{pmatrix}
1 & x_0 & \ldots & x_0^n \\
1 & x_1 & \ldots & x_1^n \\
\vdots & \vdots & & \vdots \\ 
1 & x_n & \ldots & x_n ^ n
\end{pmatrix}, 
\ a = \begin{pmatrix}
    a_0 \\ a_1 \\ \vdots \\ a_n
\end{pmatrix}, 
\
y = \begin{pmatrix}
    y_0 \\ y_1 \\ \vdots \\ y_n
\end{pmatrix}.
$$

As we saw in Example \ref{ex:linearPreimage}, this has a global minimum whenever $X$ is non-singular, i.e., whenever the $x_i$ are distinct.

To test our optimisation algorithms on polynomial interpolation, we generate random distinct $p$-adic numbers $x_i$ and try to solve the optimisation problem described above. 
As before, we always initialise the optimisation process at the Gauss point $B_\mathcal{H}(0, 0)$ and vary several experimental parameters: the prime $p \in \{2, 3, 5\}$ and the number of points $x_i$. We use the same optimisers as in the previous experiment.

\begin{figure}[t]
    \centering
    \begin{subfigure}{0.48\textwidth}
        \centering
        \includegraphics[width=\linewidth]{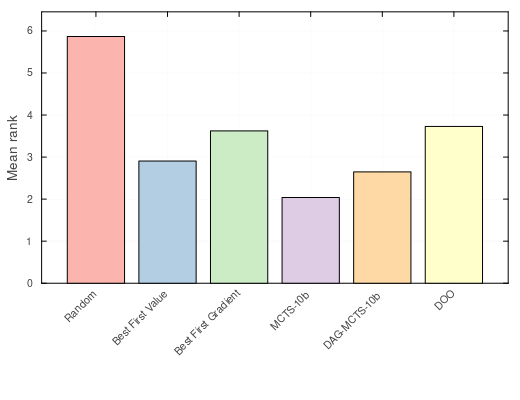}\vspace{-6mm}
        \caption{Average rank}
    \end{subfigure}
    \hfill
    \begin{subfigure}{0.48\textwidth}
        \centering
        \includegraphics[width=\linewidth]{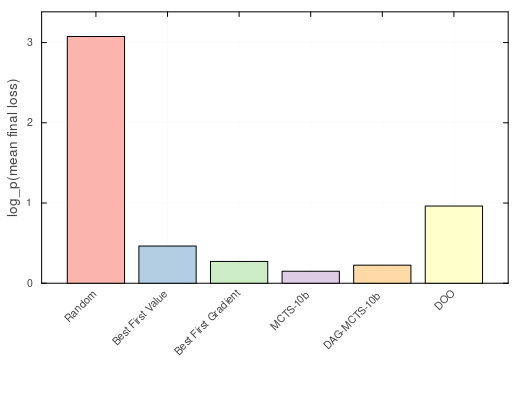}\vspace{-6mm}
        \caption{Mean final loss}
    \end{subfigure}
    \caption{Average ranks and mean final loss of the optimisers across several polynomial interpolation tasks}
    \label{fig:polynomial_interpolation_comparison}
\end{figure}

The summary results are shown in \cref{fig:polynomial_interpolation_comparison}; more detail is given in Section \ref{subsec:experimental-results}.

\subsection{Learning a Function on Random Data}

We experiment with the following standard learning problem: given data points $x_1, \dotsc, x_n$ and outputs $y_1, \dotsc, y_n \in \{0, 1\}$, find a function $f : K \to \RR$ such that $f(x_i) \approx y_i$. 

To do so, we generate small datasets $x_1, \dotsc x_n$ of random $p$-adic numbers, and random corresponding $y_i \in \{0, 1\}$, and attempt to fit this data using various choices of a learnable function $f_\theta$, which we determine by attempting to minimise the cross entropy loss function
$$
  \ell(\theta) = - \sum_i y_i \log f_\theta(x_i) + (1-y_i) \log \left(1 - f_\theta(x_i) \right) 
$$
We construct the function $f_\theta$ as follows:
$$
f_\theta(x) = \sigma_{\mu, \tau} \circ g_\theta(x)\qquad\text{where }
\sigma_{\mu, \tau}(x) = \sigma(\tau (x - \mu)) \text{ and }
\sigma(x) = \frac{1}{1 + e ^ {-x}}
$$
and $g_\theta$ is the absolute value of a polynomial (or a sum of such absolute values). 

We test our optimisation methods against this problem by generating random $p$-adic values $x_i$ and binary labels $y_i$ and trying to fit the data with a function of the form $f_\theta$.
As before, we always start from the Gauss point, and try several optimisers: a random baseline, best-first value search, best first gradient search, MCTS, and DAG--MCTS. We vary several parameters of the experiment:
the prime $p \in \{2, 3, 5\}$, the degree of the polynomial used to fit the data (degrees $2$, $3$, and $4$), and the target function, which is either the constant zero function, the constant one function, or a function with random binary labels. In each experiment, we take $\mu = \frac 1 2$ and $\tau = 1$. Summary results are shown in \cref{fig:function_learning_comparison}; more details about the experimental setup are given in Section \ref{subsec:experimental-results}.

\begin{figure}[t]
    \centering
    \begin{subfigure}{0.48\textwidth}
        \centering
        \includegraphics[width=\linewidth]{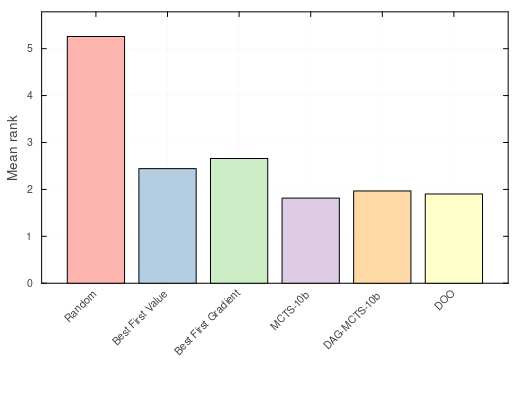}\vspace{-6mm}
        \caption{Average rank}
    \end{subfigure}
    \hfill
    \begin{subfigure}{0.48\textwidth}
        \centering
        \includegraphics[width=\linewidth]{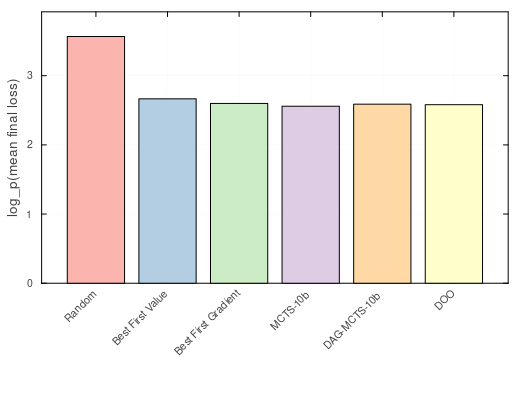}\vspace{-6mm}
        \caption{Mean final loss}
    \end{subfigure}
    \caption{Average ranks and mean final loss of the optimisers across several function learning tasks}
    \label{fig:function_learning_comparison}
\end{figure}

\subsection{Experimental Results} \label{subsec:experimental-results}

In this section, we analyse the outcome of the experiments described above.

\subsubsection{Methodology for Comparing Optimisers}

Before giving a summary of the outcome of the experiments, we briefly outline the methodology used to make comparisons between optimisers. The experiments have been designed so that for any given problem configuration (e.g., ``minimise a cubic univariate polynomial with coefficients in $\mathbb{Z}_3$''), we can draw many random sample problems (i.e., randomly sampled cubic polynomials). For each problem configuration, we select $N = 50$ sample problems $P_1, \dotsc, P_N$, and run each optimiser on each $P_i$, and rank the optimisers according to the final loss achieved on each given problem $P_i$. If two optimisers tie for a position $j$ in the ranking, they are both given the rank $j$, meaning that several optimisers can achieve a good ranking if they perform identically. We then report the mean rank of each optimiser across all given samples of a given problem configuration. We use the mean rank as the main metric in order to have a measure that is robust to outlier problems that are harder than average and might drive the mean loss of a given optimiser upwards. In particular, a given optimiser may rank better than another on average, but have worse average loss because of such outliers. For completeness, we also often report the mean loss.

\begin{table}[t]
    \centering
    \begin{tabular}{|l|c|c|c|}
        \hline
        Name & Branching & Algorithm & Additional Parameters \\
        \hline
        Random & $\mathcal{C}^1$ & --- & --- \\
        Best First Value & $\mathcal{C}^1$ or $\mathcal{C}^2$ & \cref{alg:best-first-value-search} & --- \\ 
        Best First Gradient & $\mathcal{C}^1$ or $\mathcal{C}^2$ & \cref{alg:best-first-gradient-search} & --- \\ 
        MCTS & $\mathcal{C}^1$ or $\mathcal{C}^2$ & \cref{alg:mcts-update} & $N_\mathrm{sim} = b$, $5b$ or $10b$ \\
        DAG-MCTS & $\mathcal{C}^1$ or $\mathcal{C}^2$ & \cref{alg:dag-mcts-update} & $N_\mathrm{sim} = b$, $5b$ or $10b$ \\
        DOO & $\mathcal{C}^2$ & \cref{alg:doo} & --- \\ 
        \hline
    \end{tabular}
    \caption{Optimisers used in the experiments}
    \label{table:experiment-optimisers}
\end{table}

\cref{table:experiment-optimisers} describes the various optimisers that are run on each experiment. These cover a variety of different tree search approaches which offer different solutions to the exploration vs exploitation trade-off: exploitation focused methods like the best first search algorithms, and more balanced approaches such as MCTS, DAG--MCTS and DOO. The variations in the hyperparameters (branching, i.e., $\mathcal{C} ^ 1$ vs $\mathcal{C} ^ 2$; number of simulations; and the exploration constant in the case of MCTS) used here reflect the fact that different choices can lead to very different behaviour and performance in practice.

\subsubsection{Summary of the Experiments}
The experiments cover several different levels of difficulty for optimisation problems.
\begin{itemize}
    \item \emph{Polynomial Solving}. This is a very natural example of a minimisation problem over $\mathcal{B}^n$, and perhaps the simplest: As seen in the proof of \cref{thm:best-first-descent-convergence}, it suffices to ensure that the loss is always decreasing and that all valuative radii tend to zero to ensure convergence to a root (and thus a global minimum).
    \item \emph{Absolute Polynomial Sum Minimisation}. This is a natural extension of the first problem, leading to more complicated loss landscapes. Unlike the first problem, we no longer have the guarantee on convergence: merely following the steepest path down is no longer enough to ensure one reaches the global minimum.
    In particular in this experiment and the following ones, it might be expected that algorithms searching longer before making an update will perform better than algorithms that perform a very shallow search.
    When running this experiment, we restrict to relatively easy instances of this problem (few terms in the sum, and low degree); more complicated instances of this problem occur in the next experiment. 
    \item \emph{Polynomial Interpolation}. This is a special case of the absolute polynomial sum minimisation problem. Here we tackle more complicated instances of this problem than in the previous experiment: the terms of the sums are always quadratic, and the sums generally have more terms. 
    \item \emph{Function Learning}. Unlike all previous experiments, the functions we deal with in this experiment are no longer absolute polynomial sums, and can have more complicated loss landscapes: for example, they may not be monotone on the polydisc space. This gives us a good testing ground to understand how the various optimisers designed in this paper behave on more complicated loss landscapes that go beyond the setting of absolute polynomial sums.
\end{itemize}

\subsubsection{Findings}

\paragraph{Branching.} To investigate the impact of the choice of the child generation function $\mathcal C$, we run pairwise comparisons for Best First Value, Best First Gradient and MCTS. In each case, we run the algorithm twice on all four experimental setups, using $\mathcal C ^ 1$ and $\mathcal C ^ 2$ respectively as child generating functions, and measure the average rank.
As shown in \cref{fig:branching_comparison}, in all cases we find that increasing the amount of branching, i.e., using $\mathcal C ^ 2$ rather than $\mathcal C ^ 1$ yields better results.

\begin{figure}[t]
    \centering
    \begin{subfigure}{0.3 \textwidth}
        \centering
        \includegraphics[width=0.8 \textwidth]{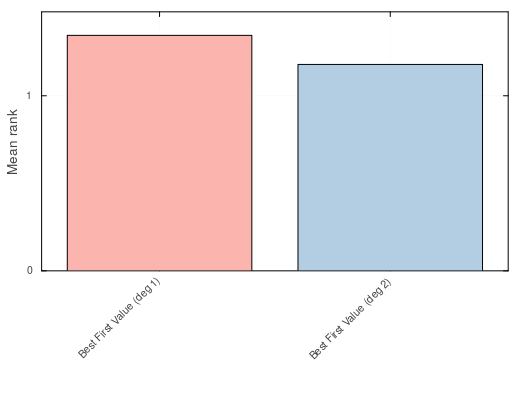}\vspace{-3mm}
        \caption{Best First Value}
    \end{subfigure}
    \begin{subfigure}{0.3 \textwidth}
        \centering
        \includegraphics[width=0.8 \textwidth]{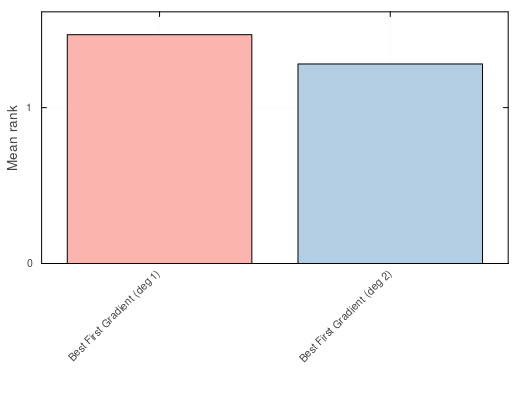}\vspace{-3mm}
        \caption{Best First Gradient}
    \end{subfigure}
    \begin{subfigure}{0.3 \textwidth}
        \centering
        \includegraphics[width=0.8 \textwidth]{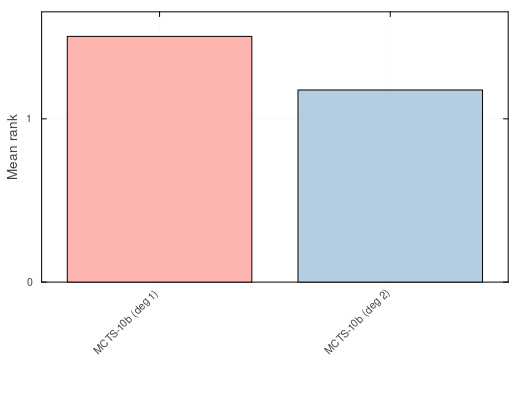}\vspace{-3mm}
        \caption{MCTS}
    \end{subfigure}
    \caption{Effect of branching ($\mathcal{C}^1$ vs.\ $\mathcal{C}^2$) on optimiser performance}
    \label{fig:branching_comparison}
\end{figure}

\paragraph{Comparison Between Optimisers.}

\begin{figure}[t]
    \centering
    \includegraphics[width=0.5 \textwidth]{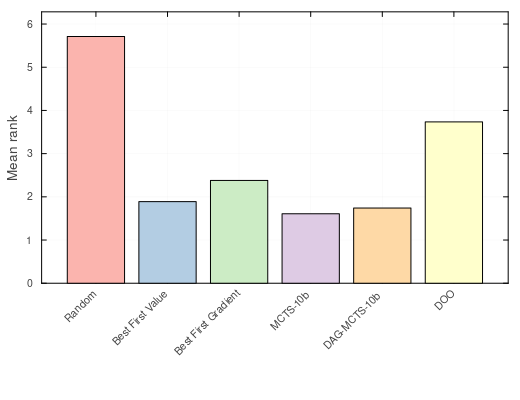}\vspace{-8mm}
    \caption{Average ranks of optimisers across all experiments}
    \label{fig:overall_ranking}
\end{figure}

As shown in \cref{fig:overall_ranking}, the best performing optimiser across all experiments is MCTS, which matches the intuition that allowing search algorithms to explore deeper in the DAG allows them to make better decisions. Perhaps surprisingly, DAG-MCTS performs slightly less well than MCTS, while still performing better than other shallower search algorithms.
This performance of MCTS and DAG--MCTS does come at the cost of being far more computationally intensive, compared to best-first value search methods. 

We note that on the problems considered, best first value search performs remarkably well on average compared to other methods, relative to the number of function evaluations required. One possible explanation for this fact is that the problems we test on here have relatively simple loss landscapes which lend themselves well to the ``best child'' heuristic it uses.

\paragraph{Number of Simulations.}

\begin{figure}[t]
    \centering
    \begin{subfigure}{0.4 \textwidth}
        \centering
        \includegraphics[width=0.8 \textwidth]{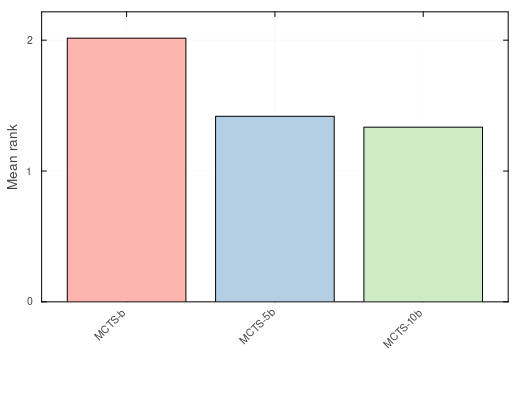}\vspace{-5mm}
        \caption{MCTS}
    \end{subfigure}
    \begin{subfigure}{0.4 \textwidth}
        \centering
        \includegraphics[width=0.8 \textwidth]{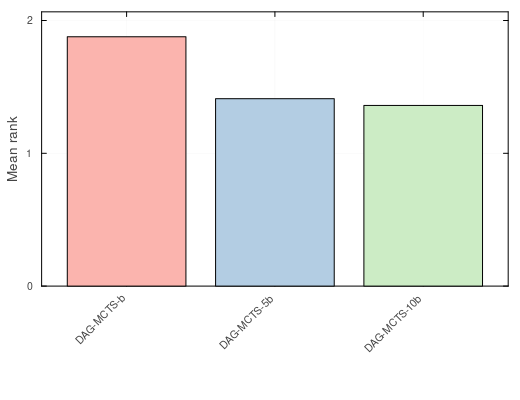}\vspace{-5mm}
        \caption{DAG-MCTS}
    \end{subfigure}
    \caption{Effect of number of simulations on MCTS and DAG--MCTS performance}
    \label{fig:simulations_comparison}
\end{figure}

As shown in \cref{fig:simulations_comparison}, as expected, the performance of MCTS-based optimisers improves as the number of simulations increases. Intuitively, this corresponds to the fact that we allow the optimiser to look at larger parts of the tree before choosing which path to descend.
In particular, this allows the optimiser to discard branches that might look promising but turn out to be poor choices, as well as choose branches which turn out to be promising after having been explored more.

\paragraph{Performance Considerations.}
\begin{figure}[t]
    \centering
    \begin{subfigure}{.45 \textwidth}
        \centering
        \includegraphics[width=.8 \textwidth]{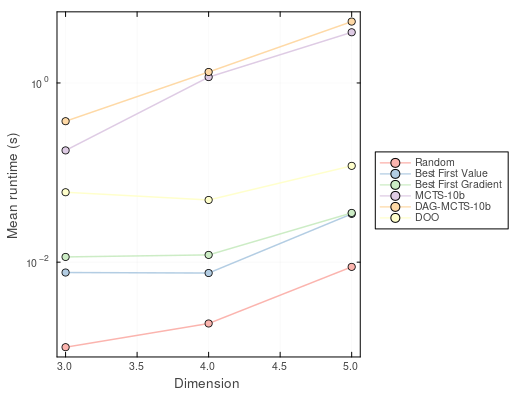}
        \caption{Average runtime per sample, $p = 2$}
    \end{subfigure}
    \begin{subfigure}{.45 \textwidth}
        \centering
        \includegraphics[width=.8 \textwidth]{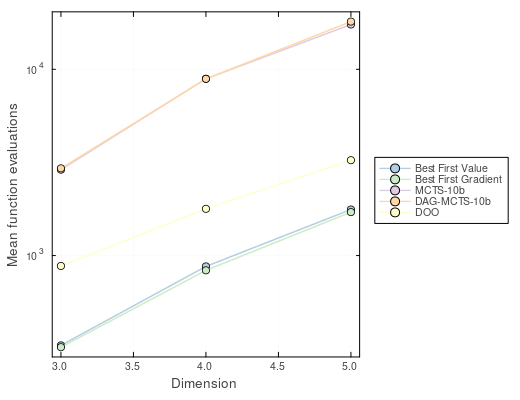}
        \caption{Average number of evaluations per sample, $p=2$}
    \end{subfigure}
    \caption{Function learning experiment performance}
    \label{fig:performance_comparison}
\end{figure}

As shown in \cref{fig:performance_comparison}, across all optimisers tested, both the average runtime and the average number of function evaluations scale roughly log-linearly with the dimension. The MCTS-based methods are the most expensive by a significant margin, which is consistent with their deeper exploration of the search DAG; the best-first methods sit at the opposite end of the spectrum.

The two main bottlenecks of the current implementation are the child generating function $\mathcal{C}$ and the evaluation of the loss at nodes, both of which can be called on the order of $10^3$ times per optimisation run when using MCTS. While the current implementation already supports the experiments reported in this paper, several opportunities for further optimisation remain. 


\section{Discussion}
\label{sec:discussion}
In this paper, we introduce polydiscs as a natural geometric setting both to embed hierarchical data, and to optimise functions defined over non-Archimedean fields. 
While various objects such as hyperbolic spaces have been proposed by the community, we argue that non-Archimedean geometry provides a natural setting to deal with hierarchical data, because of the well-known connection between hierarchies and ultrametrics.
The spaces we construct are motivated by Berkovich spaces in non-Archimedean geometry, but are simpler to represent and to compute with, and possess various additional structure, including a metric, a partial order, a notion of tangent directions and differentiation, making them a natural setting for optimisation.
We showed how these spaces are endowed with a natural class of well-behaved functions, linear combinations of absolute polynomials, that are universal approximators and can be easily computed in practice as they are locally representable by polynomials in the radii of polydiscs.
One of the key developments of this paper is the theory of optimisation presented in \cref{sec:optimisation}. 
There, we focus on positive linear combinations of absolute polynomials, which are particularly amenable to study and can be seen as an analogue of \emph{convex} functions on the polydisc space.
We show that such functions, under some mild geometric conditions on the zero loci of their defining polynomials, admit global minima on the polydisc space, and that these global minima can be related to the zero loci of these polynomials.
This allows us to derive several results about non-Archimedean analogues of well-known statistical optimisation problems such as Fr\'echet mean and ordinary least squares.
In the one dimensional case, we give a full characterisation of the global minima of positive absolute polynomial sums, and compute a bounded region of the polydisc space that contains all possible global minima of such functions.
We develop a more computational approach to optimisation of monotone decreasing functions over the polydisc space, and propose several algorithms for computing local minima in practice.
A crucial connection made here is that one can rephrase optimisation problems over the polydisc space as sequential decision making problems over the subspace of \emph{rational polydiscs}, which can be thought of as the elements of the polydisc space where the branching is ``maximal''.
This provides a unifying conceptual framework for designing optimisation procedures over the polydisc space, and allows us to propose polydisc versions of several powerful search algorithms such as DOO and MCTS that are well-adapted to search spaces with a large branching factor.
Importantly, we provide a Julia library called \textsc{NonArchimedeanMachineLearning.jl} that implements these algorithms, and we provide various experiments to demonstrate how these optimisers behave on concrete problems and compare performance across various configurations. Notably, we find that MCTS (and its DAG variant, DAG--MCTS) perform better on average than the other optimisers presented here, and that the performance of the algorithm increases as we scale the computational resources available to it (i.e., the number of simulations per step).

\subsection{Directions for Future Research}

Our work inspires several directions for future research, particularly towards complete optimisation and a fully-structured machine learning pipeline, which we now discuss.

One first direction would be in the function theory of polydisc spaces.  In concrete applications, having a rich class of functions that can be used to approximate quantities of interest is key, and thus studying broader classes of functions than just linear combinations of absolute polynomials may be of direct practical interest.
For example, compositions of such functions with real valued functions, kernels over polydisc spaces or even perhaps non-Archimedean analogs of neural network layers could be of direct practical use. Such a study would involve broadening the optimisation techniques presented here to a larger class of functions, and could benefit from the input of other techniques that have been developed in the context of non-Archimedean geometry such as potential theory. We note that in this more general setting, certain results may no longer always hold; for example, if one considers general linear combinations of absolute polynomials, it is no longer the case that the global minimum (if it exists) must lie in the base vector space $K^n \subset \mathcal B ^ n$.

Another promising direction would be more algorithmic, and would investigate more deeply the computational methods proposed here for finding local minima. 
We believe that the connection proposed here between sequential decision making processes and optimisation problems over $\mathcal B ^ n$ particularly merits further investigation. 
Similarly, it is plausible that using deeper ideas from non-Archimedean geometry might be beneficial here for developing more efficient methods that exploit the specific structure of the polydisc space and of the subset $\mathcal B ^ n _ \mathrm{rat}$.
At a very practical level, developing principled methods for reducing the dimension or branching factor of the decision problem seems achievable (note that the coordinate-wise best first descent algorithm can be seen as a basic example of this) and could result in significant computational improvements.
In a similar manner, tree (and DAG) search problems have been the object of much attention from the computer science community, and it seems plausible that many of the algorithmic improvements that have been proposed for MCTS-like techniques (adaptive branching, memory management, etc.) might be brought to bear on our specific context.

The final proposed directions bring us back to the notion that lies at the origin of this work, namely applications to concrete hierarchical datasets.
While the experiments described in \cref{sec:applications} show that the framework we have developed can be used for concrete computations, we point out that all examples considered are synthetic, and some work remains to be done to bridge the gap to concrete computational applications.
Our results on embeddings suggest that (hyperbolic) polydisc spaces provide a natural place to embed hierarchical data structures such as phylogenies or ontologies.
Given such an embedding, the function classes developed here provide a natural framework to learn quantities, for example in regression or classification problems.
More generally, an important next step in this context would be to develop methods to learn an embedding jointly with functions of interest.
A key test of the methods proposed would then be to benchmark their performance on some standard hierarchical datasets, and compare them with hyperbolic and Euclidean techniques.

\section*{Acknowledgments}

Y.F.~and P.L.~are funded by a London School of Geometry and Number Theory--Imperial College London/King's College London/University College London PhD studentship, which is supported by the Engineering and Physical Sciences Research Council [EP/S021590/1].  A.M.~and Y.R.~are supported by the EPSRC AI Hub on Mathematical Foundations of Intelligence: An ``Erlangen Programme'' for AI [EP/Y028872/1]. Y.R.~is supported by the UKRI FLF Computational tropical geometry and its applications [MR/Y003888/1].

\renewcommand{\bibfont}{\small}
\printbibliography

\end{document}